\documentclass[11pt,oneside,article]{memoir}

\usepackage{mathtools}
\usepackage{amsthm}
\usepackage{amssymb}
\usepackage{makecell}
\usepackage{fixltx2e}
\usepackage[T1]{fontenc}
\usepackage[sc]{mathpazo}
\usepackage[cal=euler,scr=rsfso]{mathalfa}
\usepackage{bm}
\usepackage{inputenc}
\usepackage{microtype}
\usepackage[usenames,dvipsnames]{xcolor}
\usepackage{paralist}
\usepackage{booktabs}
\usepackage{tikz}
\usepackage{todonotes}
\usepackage[bookmarks=false,colorlinks, linkcolor={blue!60!black}, citecolor={red!50!black}, urlcolor={blue!80!black}]{hyperref}

\usetikzlibrary{positioning,decorations.markings,arrows.meta,calc,fit,quotes,cd}

\DeclareMathOperator{\id}{id}

\DeclareMathOperator{\Ob}{Ob}

\DeclareMathOperator{\Tr}{Tr}

\theoremstyle{plain}
\newtheorem{theorem}{Theorem}[chapter]
\newtheorem*{theorem*}{Theorem}

\newtheorem{proposition}[theorem]{Proposition}
\newtheorem*{proposition*}{Proposition}
\newtheorem{corollary}[theorem]{Corollary}
\newtheorem{lemma}[theorem]{Lemma}
\newtheorem*{lemma*}{Lemma}

\theoremstyle{definition}
\newtheorem{definition}[theorem]{Definition}
\newtheorem{construction}[theorem]{Construction}

\theoremstyle{remark}
\newtheorem{example}[theorem]{Example}
\newtheorem{remark}[theorem]{Remark}

\newcommand{\RR}{\mathbb{R}}
\newcommand{\NN}{\mathbb{N}}
\newcommand{\PP}{\mathbb{P}}

\newcommand{\partials}[2]{\left(\frac{\partial #1}{\partial #2}\right)}

\def\tn{\textnormal}

\newcommand{\op}{^{\tn{op}}}
\newcommand{\cat}[1]{\mathscr{#1}} %category variable
\newcommand{\ncat}[1]{\mathbf{#1}} %named category
\newcommand{\nfun}[1]{\mathsf{#1}}  %named functor

\newcommand{\Set}{\ncat{Set}} % 1-cat Set
\newcommand{\Meas}{\ncat{CSMeas}}

\newcommand{\inmark}{\tn{in}}
\newcommand{\outmark}{\tn{out}}
\newcommand{\dProd}[1]{\widehat{#1}}
\newcommand{\inp}[1]{#1^\inmark}
\newcommand{\outp}[1]{#1^\outmark}
\newcommand{\midp}[1]{#1^{\tn{mid}}}
\newcommand{\dyn}[1]{#1^{\tn{dyn}}}
\newcommand{\upd}[1]{#1^{\tn{upd}}}
\newcommand{\rdt}[1]{#1^{\tn{rdt}}}
\newcommand{\vinp}[1]{\dProd{\inp{#1}}}
\newcommand{\voutp}[1]{\dProd{\outp{#1}}}

\newcommand{\seq}[1]{\left\langle #1\right\rangle}
\newcommand{\Appx}{\tn{Appx}_{\epsilon}}
\newcommand{\Stst}{\tn{Stst}}
\newcommand{\StstS}{\mathcal{S}\tn{tst}}
\newcommand{\StLin}{\mathcal{S}\tn{tLin}}

\newcommand{\TFS}{\ncat{TFS}}
\newcommand{\Vect}{\ncat{Vect}}
\newcommand{\Euc}{\ncat{Euc}}
\newcommand{\FinSet}{\ncat{FinSet}}
\newcommand{\Lin}{\ncat{Lin}}

\renewcommand{\ss}{\subseteq}

\newcommand{\Strm}{\text{Strm}}
\newcommand{\OS}{\nfun{OS}}
\newcommand{\Mat}{\nfun{Mat}}
\newcommand{\MatS}{\bm{\mathcal{M}}\nfun{at}}
\newcommand{\DS}{\nfun{DS}}
\newcommand{\CS}{\nfun{CS}}
\newcommand{\LS}{\nfun{LS}}
\newcommand{\MS}{\nfun{MS}}

\newcommand{\To}[1]{\xrightarrow{#1}}

\newcommand{\too}{\longrightarrow}

\newcommand{\pr}{\tn{pr}}

\tikzset{
   oriented WD/.style={
      every to/.style={out=0,in=180,draw},
      label/.style={
         font=\everymath\expandafter{\the\everymath\scriptstyle},
         inner sep=0pt,
         node distance=2pt and -2pt},
      semithick,
      node distance=1 and 1,
      decoration={markings, mark=at position .5 with {\arrow{stealth};}},
      ar/.style={postaction={decorate}},
      execute at begin picture={\tikzset{
         x=\bbx, y=\bby,
         every fit/.style={inner xsep=\bbx, inner ysep=\bby}}}
      },
   bbx/.store in=\bbx,
   bbx = 1.5cm,
   bby/.store in=\bby,
   bby = 1.75ex,
   bb port sep/.store in=\bbportsep,
   bb port sep=2,
   bb port length/.store in=\bbportlen,
   bb port length=4pt,
   bb min width/.store in=\bbminwidth,
   bb min width=1cm,
   bb rounded corners/.store in=\bbcorners,
   bb rounded corners=2pt,
   bb small/.style={bb port sep=1, bb port length=2.5pt, bbx=.4cm, bb min width=.4cm, bby=.7ex},
   bb/.code 2 args={
      \pgfmathsetlengthmacro{\bbheight}{\bbportsep * (max(#1,#2)+1) * \bby}
      \pgfkeysalso{draw,minimum height=\bbheight,minimum width=\bbminwidth,outer sep=0pt,
         rounded corners=\bbcorners,thick,
         prefix after command={\pgfextra{\let\fixname\tikzlastnode}},
         append after command={\pgfextra{\draw
            \ifnum #1=0{} \else foreach \i in {1,...,#1} {
               ($(\fixname.north west)!{\i/(#1+1)}!(\fixname.south west)$) +(-\bbportlen,0) coordinate (\fixname_in\i) -- +(\bbportlen,0) coordinate (\fixname_in\i')}\fi
            \ifnum #2=0{} \else foreach \i in {1,...,#2} {
               ($(\fixname.north east)!{\i/(#2+1)}!(\fixname.south east)$) +(-\bbportlen,0) coordinate (\fixname_out\i') -- +(\bbportlen,0) coordinate (\fixname_out\i)}\fi;
         }}}
   },
   bb name/.style={append after command={\pgfextra{\node[anchor=north] at (\fixname.north) {#1};}}}
}

\newcommand{\smallBox}[2]{\begin{tikzpicture}[oriented WD, baseline=(X.south), bb small]\node [bb={1}{1}] (X) {};\draw[label] node[left=.1 of X_in1]  {$#1$} node[right=.1 of X_out1] {$#2$};\end{tikzpicture}}

\settrims{0pt}{0pt} % page and stock same size
\settypeblocksize{*}{37pc}{*} % {height}{width}{ratio}
\setlrmargins{*}{*}{1} % {spine}{edge}{ratio}
\setulmarginsandblock{1in}{1in}{*} % height of typeblock computed
\setheadfoot{\onelineskip}{2\onelineskip} % {headheight}{footskip}
\setheaderspaces{*}{1.5\onelineskip}{*} % {headdrop}{headsep}{ratio}
\checkandfixthelayout

\allowdisplaybreaks[1]

\setsecheadstyle{\bfseries\large\raggedright}
\setsubsecheadstyle{\bfseries\raggedright}

\title{The steady states of coupled dynamical systems\\compose according to matrix arithmetic}
\author{
   David I. Spivak
 }

\begin{document}
\tightlists
\firmlists

\maketitle

\begin{abstract}
Open dynamical systems are mathematical models of machines that take input, change their internal state, and produce output. For example, one may model anything from neurons to robots in this way. Several open dynamical systems can be arranged in series, in parallel, and with feedback to form a new dynamical system---this is called compositionality---and the process can be repeated in a fractal-like manner to form more complex systems of systems. One issue is that as larger systems are created, their state space grows exponentially. 

In this paper a technique for calculating the steady states of an interconnected system of systems, in terms of the steady states of its component dynamical systems, is provided. These are organized into "steady state matrices" which generalize bifurcation diagrams. It is shown that the compositionality structure of dynamical systems fits with the familiar monoidal structure for the steady state matrices, where serial, parallel, and feedback composition of matrices correspond to multiplication, Kronecker product, and partial trace operations. The steady state matrices of dynamical systems respect this compositionality structure, exponentially reducing the complexity involved in studying the steady states of composite dynamical systems.
\end{abstract}

\tableofcontents*
 
\chapter{Introduction}\label{sec:introduction}

Open dynamical systems can be composed to make larger systems. For example, they can be put together in series or in parallel
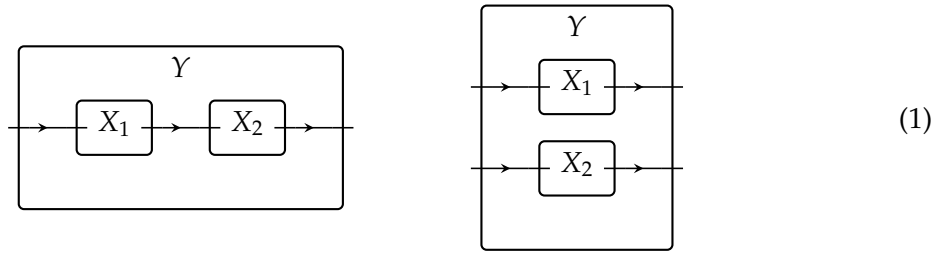
\begin{align}\label{dia:series_parallel}
\begin{tikzpicture}[oriented WD, baseline=(Y.center), bbx=1em, bby=1ex]
 \node[bb={1}{1},bb name=$X_1$] (X1) {};
 \node[bb={1}{1},right =2 of X1, bb name=$X_2$] (X2) {};
 \node[bb={1}{1}, fit={($(X1.north west)+(-1,3)$) ($(X1.south)+(0,-3)$) ($(X2.east)+(1,0)$)}, bb name = $Y$] (Y) {};
 \draw[ar] (Y_in1') to (X1_in1);
 \draw[ar] (X1_out1) to (X2_in1);
 \draw[ar] (X2_out1) to (Y_out1');
\end{tikzpicture}
\qquad\qquad
\begin{tikzpicture}[oriented WD,baseline=(Y.center), bbx=1em, bby=1ex]
 \node[bb={1}{1},bb name=$X_1$] (X1) {};
 \node[bb={1}{1},below =2 of X1, bb name=$X_2$] (X2) {};
 \node[bb={2}{2}, fit={($(X1.north west)+(-1,3)$) ($(X2.south)+(0,-3)$) ($(X2.east)+(1,0)$)}, bb name = $Y$] (Y) {};
 \draw[ar] (Y_in1') to (X1_in1);
 \draw[ar] (Y_in2') to (X2_in1);
 \draw[ar] (X1_out1) to (Y_out1');
 \draw[ar] (X2_out1) to (Y_out2');
\end{tikzpicture}
\end{align}
or in a more complex combination, possibly with feedback and splitting wires
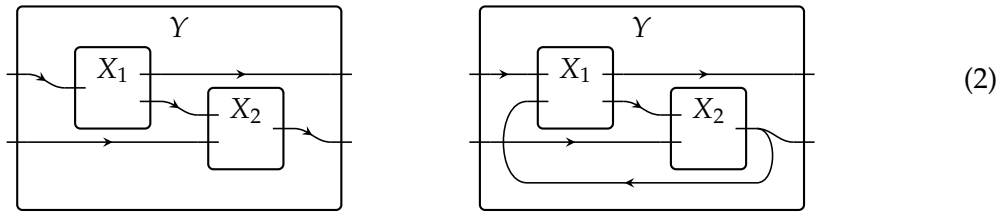
\begin{align}\label{dia:more_complex}
\begin{tikzpicture}[oriented WD,baseline, bbx=1em, bby=1ex]
 \node[bb={1}{2},bb name=$X_1$] (X1) {};
 \node[bb={2}{1},below right = -3 and 2 of X1, bb name=$X_2$] (X2) {};
 \node[bb={2}{2}, fit={($(X1.north west)+(-1,2)$) ($(X2.south)+(0,-2)$) ($(X2.east)+(1,0)$)}, bb name = $Y$] (Y) {};
 \draw[ar] (Y_in1') to (X1_in1);
 \draw[ar] (X1_out2) to (X2_in1);
 \draw[ar] (Y_in2') to (X2_in2);
 \draw[ar] (X1_out1) to (Y_out1');
 \draw[ar] (X2_out1) to (Y_out2');
\end{tikzpicture}
\qquad\qquad
\begin{tikzpicture}[oriented WD,baseline, bbx=1em, bby=1ex]
 \node[bb={2}{2},bb name=$X_1$] (X1) {};
 \node[bb={2}{1},below right = -3 and 2 of X1, bb name=$X_2$] (X2) {};
 \node[bb={2}{2}, fit={($(X1.north west)+(-1,2)$) ($(X2.south)+(0,-2)$) ($(X2.east)+(1,0)$)}, bb name = $Y$] (Y) {};
 \draw[ar] (Y_in1') to (X1_in1);
 \draw[ar] (X1_out2) to (X2_in1);
 \draw[ar] (Y_in2') to (X2_in2);
 \draw[ar] (X1_out1) to (Y_out1');
 \draw (X2_out1) to (Y_out2');
 \draw[ar] let \p1=(X2.south east), \p2=(X1.south west), \n1={\y1-\bby}, \n2=\bbportlen in
 (X2_out1) to[in=0] (\x1+\n2,\n1) -- (\x2-\n2,\n1) to[out=180] (X1_in2);
\end{tikzpicture}
\end{align}
A dynamical system has a set or space of states and a rule for how the state changes in time. An \emph{open} dynamical system also has an interface $X$ (as shown above), which indicates the number of input ports and output ports that exist for the system. Signals passed to the system through its input ports influence how the state changes. An output signal is generated as a function of the state, and it is then passed through output port to serve as an input to a neighboring system. The precise notions of dynamical systems we use in this paper (including discrete and continuous models) will be given in Section~\ref{sec:ODS_and_matrices}; for now we speak about dynamical systems in the abstract.

For any interface $X$, let $\OS(X)$ denote the set of all possible open dynamical systems of type $X$. The idea is that a diagram, such as any of those found in \eqref{dia:series_parallel} or \eqref{dia:more_complex}, determines a function 
\[\OS(X_1)\times \OS(X_2)\to \OS(Y).\]
This function amounts to a formula that produces an open system of type $Y$ given open systems of type $X_1$ and $X_2$, arranged---in terms of how signals are passed---according to the wiring diagram. The formula enforces that wires connecting interfaces correspond to variables shared by the dynamical systems.%
\footnote{Of course, one can wire together an arbitrary number of internal dynamical systems $X_1,\ldots,X_n$ to form a single interconnected system $Y$, but for the introduction, we assume $n=2$ for concreteness.} 

\section{Compositional viewpoints of dynamical systems}

We are interested in looking at open dynamical systems in ways that respect arbitrary interconnection (variable coupling) via wiring diagrams, as we now briefly explain. Above, we explained that if open systems inhabit each interior box in a wiring diagram, we can construct a composite open system for the outer box. But wiring diagrams can interconnect things besides dynamical systems.

For example, wiring diagrams make sense for interconnecting matrices too. That is, to each interface $X$, one can assign a set $\Mat(X)$ of the associated type; then, given a matrix in each interior box of a wiring diagram one can put them together to form a composite matrix for the outer box. In other words, a wiring diagram, such as any found in \eqref{dia:series_parallel} or \eqref{dia:more_complex}, should determine a function
\[\Mat(X_1)\times \Mat(X_2)\to \Mat(Y).\] 
This function is computed using matrix product, tensor (Kronecker) product, and partial trace (adding up diagonal entries in each block). One of the major goals of this paper is to prove the following.

\begin{theorem*}
There is a compositional mapping $\Stst_X\colon \OS(X)\to \Mat(X)$, given by arranging steady states into a matrix form. 
\end{theorem*}
We say that a mapping is \emph{compositional} if it behaves correctly with respect to wiring diagrams in the following sense. Given open dynamical systems of type $X_1$ and $X_2$, we can either compose first and apply the mapping to the result, or apply the mapping first and then compose the results. We want these to give the same answer. Formally, we express this by saying that we want the following diagram to commute:
\[
\begin{tikzcd}[row sep=large]
\OS(X_1)\times\OS(X_2)\ar[r]\ar[d,"\Stst_{X_1}\times\Stst_{X_2}"']&\OS(Y)\ar[d,"\Stst_Y"]\\
\Mat(X_1)\times\Mat(X_2)\ar[r]&\Mat(Y)
\end{tikzcd}
\]

In this paper, we provide a compositional mapping from open dynamical systems of several sorts---discrete, measurable, and continuous---to the matrix domain. The entries of these matrices list, count, or measure the steady states---also known as \emph{equilibria} or \emph{fixed points}---of the dynamical system for each input and output. The topology of a dynamical system is to a large degree determined by its set of steady states and their stability properties, and these are generally organized into bifurcation diagrams (e.g., as in \cite{Strogatz}). Our classification is a generalization bifurcation diagrams (see Remark~\ref{rem:bifurcation}). The reason we refer to bifurcation diagrams as matrices, is that they compose according to matrix arithmetic. That is, when several dynamical subsystems are put together in series, parallel, or with feedback to form a larger system, the classifying matrix for the whole can be computed by multiplying, tensoring, or computing a partial trace of the subsystem matrices. 

\begin{figure}
\begin{center}
\begin{tikzpicture}[oriented WD, baseline=(X.center), bbx=1em, bby=1ex, bb port sep=1]
  \node[bb={2}{1}] (N1) {$\scriptstyle N_1$};
  \node[bb={1}{1}, below = 1 of N1] (N2) {$\scriptstyle N_2$};
  \node[bb={3}{1}, below=1 of N2] (N3) {$\scriptstyle N_3$};
  \node[bb={3}{1}, above right = -3 and 3.5 of N3] (N6) {$\scriptstyle N_6$};
  \node[bb={2}{1}, above =of N6] (N5) {$\scriptstyle N_5$};
  \node[bb={2}{1}, above= of N5] (N4) {$\scriptstyle N_4$};
  \node[bb={4}{4}, fit={($(N2.west)-(.5,0)$) ($(N4.north)+(0,2)$) ($(N5.east)+(1.5,0)$) ($(N3.south)-(0,1)$)}, bb name={$\small X$}] (X) {};
  \draw (X_in1') to (N1_in2);
  \draw (X_in2') to (N2_in1);
  \draw (X_in3') to (N3_in1);
  \draw (X_in4') to (N3_in2);
  \draw (N1_out1) to (N4_in1);
  \draw (N2_out1) to (N4_in2);
  \draw (N2_out1) to (N5_in1);
  \draw (N2_out1) to (N6_in1);
  \draw (N4_out1) to (X_out1');
  \draw (N3_out1) to (N6_in2);
  \draw (N5_out1) to (X_out2');
  \draw (N6_out1) to (X_out3');
  \draw (N6_out1) to (X_out4'); 
  \draw let \p1=(N2.south east), \p2=(N3_out1),\n2=\bbportlen in
  	(N3_out1) -- (\x1+\n2,\y2) to (N5_in2);
  \draw let \p1=(N6.south east), \p2=(N3.south west), \n1={\y2-\bby}, \n2=\bbportlen in
          (N5_out1) to[in=0, looseness=.8] (\x1+\n2,\n1) -- (\x2-\n2,\n1) to[out=180] (N3_in3);
  \draw let \p1=(N6.south east), \p2=(N6.south west), \n1={\y2-\bby}, \n2=\bbportlen in
          (N6_out1) to[in=0] (\x1+\n2,\n1) -- (\x2-\n2,\n1) to[out=180] (N6_in3);
  \draw let \p1=(N1.north east), \p2=(N1.north west), \n1={\y2+\bby}, \n2=\bbportlen in
          (N1_out1) to[in=0] (\x1+\n2,\n1) -- (\x2-\n2,\n1) to[out=180] (N1_in1);
\end{tikzpicture}
\end{center}
\caption{Networks of neurons can be modeled as dynamical systems wired together}
\label{fig:complex_WD}
\end{figure}
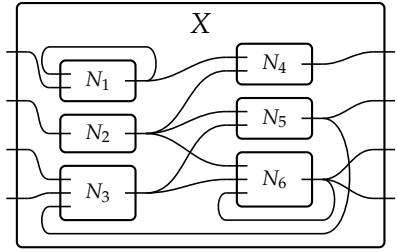

A potential interpretation of the steady state matrix in neuroscience is as follows. In perception, it is not uncommon to consider neurons as dynamical systems  \cite{Izhikevich}, and input signals can be classified as either expected or unexpected \cite{Perception}. One way to think about this is that expected input signals are those that do not change the state of the system, or at least do not change it by very much. When the state is unchanged, so is the output of the system i.e., expected perception does not cause a change in behavior. The steady state matrix presented here measures, for each (perception, behavior) pair, the set of states that are expected in that context. The purpose of the present paper is to show that this measurement is compositional, i.e., that it respects any given network structure, as in Figure~\ref{fig:complex_WD}.

One can give a similar interpretation for the steady state matrix for a discrete dynamical system. For example, consider the machine described by Alan Turing in \cite{Computing_Machinery_Intelligence}, which turns a light on and off every few minutes, unless stopped by a lever. Its state is the position of an internal wheel, its input signals are given by the lever, and its output is electrical current running the lightbulb. The steady state matrix tells us something important about the system: that when the lever is "off", every state is fixed, but when the lever is in "on", the states are constantly changing.

\section{Plan of paper}

While this paper has category theory as its underlying framework, the audience is intended to include scientists and engineers with little or no background in category theory. Hence, much of the paper is spent without reference to category theory, so the background (given in Section~\ref{sec:notation_background}) can be made fairly modest. The content of the paper begins in Section~\ref{sec:ODS_and_matrices}, where we discuss five different interpretations of the box and wiring diagram syntax: discrete, measurable, linear, and continuous dynamical systems, as well as matrices. In Section~\ref{sec:introducing_classification}, we preview the classification function that extracts a matrix of steady states from a dynamical system, and show that it is compositional with respect to serial wiring diagrams. We prove that it is in fact compositional with respect to all wiring diagrams in Section~\ref{sec:compositional_mappings}. In that secion, we also show that Euler's method (of discretizing a continuous dynamical system) is compositional, as is calculating the linear stability of steady states. 

In order to get to this point, we need to formally define wiring diagrams and their composition (Section~\ref{sec:CT_formulation}); the category-theoretic idea is to use symmetric monoidal categories $\cat{W}$. In Section~\ref{sec:formal_specs}, we show that our "five interpretations" are lax monoidal functors $\cat{W}\to\Set$. Bringing it down to earth, each of our five interpretations (discrete, measurable, and continuous dynamical systems, and matrices) is compositional with respect to wiring diagrams. Finally in Section~\ref{sec:extended_example} we give an extended example.

\section{Notation and Background}\label{sec:notation_background}

When we write $X\in\Set$, we mean that $X$ is a set. If $X$ and $Y$ are sets, we denote their cartesian product by $X\times Y$; it is the set of pairs $\{(x,y)\mid x\in X, y\in Y\}$. We denote their coproduct---i.e., their disjoint union---by $X+Y$. 

Given a set $S$, we define its \emph{count}, denoted $\#S\in\NN_+$, to be the cardinality of $S$, if it is finite, and $\infty$ if it is infinite. Note that counts add and multiply correctly: $\#(X+Y)=(\#X)+(\#Y)$ and $\#(X\times Y)=(\#X)\cdot(\#Y).$ By $\PP(S)$ we mean the \emph{power set of $S$}, i.e., the set of sets $\PP(S)=\{U\ss S\}$.

We will briefly discuss measurable spaces in Section~\ref{sec:MDS}, manifolds in Section~\ref{sec:CDS}, and complete semirings in Section~\ref{sec:Mat}, but we will not need any advanced theory and will give all the necessary background and references at that time. In Sections~\ref{sec:CT_formulation}~and~\ref{sec:formal_specs} we will use a small but significant amount of category theory (Section~\ref{sec:more_background}). However, the extended example (Section~\ref{sec:extended_example}) requires almost no background.

\section*{Acknowledgements}

Thanks go to Gaurav Venkataraman and to Rosalie B\'{e}langer-Rioux for interesting and helpful conversations. This work was supported by AFOSR grant FA9550--14--1--0031, ONR grant N000141310260, and NASA grant NNH13ZEA001N.

\chapter{Open dynamical systems and matrices}\label{sec:ODS_and_matrices}

For a wiring diagram such as the one shown here
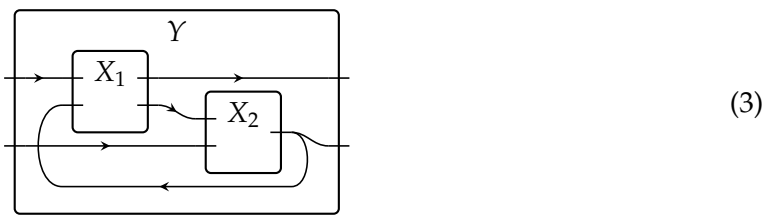
\begin{equation}\label{eqn:my_WD}
\begin{tikzpicture}[oriented WD, baseline=(Y.center), bbx=1em, bby=1ex]
 \node[bb={2}{2},bb name=$X_1$] (X1) {};
 \node[bb={2}{1},below right = -3 and 2 of X1, bb name=$X_2$] (X2) {};
 \node[bb={2}{2}, fit={($(X1.north west)+(-1,2)$) ($(X2.south)+(0,-2)$) ($(X2.east)+(1,0)$)}, bb name = $Y$] (Y) {};
 \draw[ar] (Y_in1') to (X1_in1);
 \draw[ar] (X1_out2) to (X2_in1);
 \draw[ar] (Y_in2') to (X2_in2);
 \draw[ar] (X1_out1) to (Y_out1');
 \draw (X2_out1) to (Y_out2');
 \draw[ar] let \p1=(X2.south east), \p2=(X1.south west), \n1={\y1-\bby}, \n2=\bbportlen in
 (X2_out1) to[in=0] (\x1+\n2,\n1) -- (\x2-\n2,\n1) to[out=180] (X1_in2);
\end{tikzpicture}
\end{equation}
there are many interpretations for what can inhabit each box. The implicit rule, however, is as follows: given inhabitants of the interior boxes $X_1$ and $X_2$, the wiring diagram yields a "combined", or "composite", or "interconnected" inhabitant of the outer box $Y$. In this section, we discuss five such interpretations of boxes and their interconnections. 

In Section~\ref{sec:DDS} the inhabitants of each box are discrete dynamical systems, and we provide formulas for how they can be combined in serial, parallel, splitting, and feedback diagrams to form new dynamical systems. In Section~\ref{sec:MDS} we briefly cover how this idea extends to measurable spaces, so that the dynamical systems inhabiting each box change their state and produce output in a more structured (namely, measurable) way. In Section~\ref{sec:LDS}, we discuss a linear variant, which will eventually be used for classifying the stability of continuous dynamical systems introduced in Section~\ref{sec:CDS}. Continuous dynamical systems are based on ordinary differential equations, so that states can evolve continuously. For each of the interpretations, we will eventually give composition formulas for every possible wiring diagram and show (in Section~\ref{sec:formal_specs}) that these formulas are self-consistent, i.e., that they support nesting systems of systems.

In Section~\ref{sec:Mat}, we discuss matrices, in the same context. While probably more familiar, readers may not be aware that matrices serve as another compositional interpretation of the boxes in wiring diagrams such as \eqref{eqn:my_WD}. For example, we can associate matrix multiplication to serial composition, matrix tensor (i.e., Kronecker) product to parallel composition, etc. The check that these formulas are consistent under nesting is again relegated to Section~\ref{sec:formal_specs}. Finally, in Section~\ref{sec:introducing_classification} we briefly look at the steady-state classification, which compositionally produces a matrix of steady states for dynamical systems of any sort.

\section{Discrete dynamical systems}\label{sec:DDS}

\begin{definition}\label{def:DDS}

Let $A,B\in\Set$ be sets. We define an \emph{$(A,B)$-open discrete dynamical system}, or \emph{$(A,B)$-discrete system} for short, to be a 3-tuple $(S,\rdt{f},\upd{f})$, where 
\begin{compactitem}
\item $S\in\Set$ is a set, called the \emph{state set}, 
\item $\rdt{f}\colon S\to B$ is a function, called the \emph{readout function}, and 
\item $\upd{f}\colon A\times S\to S$ is a function, called the \emph{update function}.
\end{compactitem}
We call $A$ the \emph{input set} and $B$ the \emph{output set} in this case. An \emph{initialized $(A,B)$-discrete system} is a four-tuple $(S,s_0,\rdt{f},\upd{f})$, where $(S,\rdt{f},\upd{f})$ is a discrete system and $s_0\in S$ is a chosen element, called the \emph{initial state}.

Let $\DS(A,B)$ denote the set of all $(A,B)$-discrete systems, and let $\DS_*(A,B)$ denote the set of all initialized $(A,B)$-discrete systems.

\end{definition}

\begin{remark}

The box $\smallBox{A}{B}$ will have many different interpretations in this paper, including discrete, measurable, linear, and continuous dynamical systems, as well as matrices (see also Sections~\ref{sec:MDS},~\ref{sec:CDS},~and~\ref{sec:Mat}). For discrete systems, we will think of this box as being inhabited by any $(A,B)$-discrete system, where $A$ is the input set and $B$ is the output set. Similarly, the box $\begin{tikzpicture}[oriented WD, bb small]\node [bb={2}{1}] (X) {};\draw[label] node[above left=-.1 and .1 of X_in1]  {$\scriptscriptstyle A$} node[below left =-.1 and .1 of X_in2] {$\scriptscriptstyle B$} node[right=.1 of X_out1] {$\scriptscriptstyle C$};\end{tikzpicture}$ can be inhabited by any $(A\times B,C)$-discrete system.

\end{remark}

\begin{remark}

An initialized $(A,B)$-discrete system is the same thing as a Moore machine \cite{MooreMachine} with input alphabet $A$ and output alphabet $B$. It is also what Alan Turing called a \emph{discrete state machine} \cite{Computing_Machinery_Intelligence}.

\end{remark}

\begin{definition}\label{def:steady_state_DDS}

Let $A,B\in\Set$ be sets, and let $F=(S,\rdt{f},\upd{f})$ be an $(A,B)$-discrete system. For $a\in A$ and $b\in B$, define an \emph{$(a,b)$-steady state} to be a state $s\in S$ such that $\upd{f}(a,s)=s$ and $\rdt{f}(s)=b$. We denote the set of all $(a,b)$-steady states by
\[\StstS(F)_{a,b}\coloneqq\left\{s\in S\;\;\middle|\;\;\rdt{f}(s)=b,\quad\upd{f}(a,s)=s\right\}\]
and its count by $\Stst(F)_{a,b}\coloneqq\#\StstS(F)_{a,b}$.

\end{definition}

\begin{example}\label{ex:discrete_system}

Let $A=\{\tn{T, F}\}$ and $B=\{\tn{Red, Green, Blue}\}$. Below is a small example of an $(A,B)$-discrete system (i.e., a possible inhabitant of the box $\smallBox{A}{B}$), shown both in tabular form and as a transition diagram. 
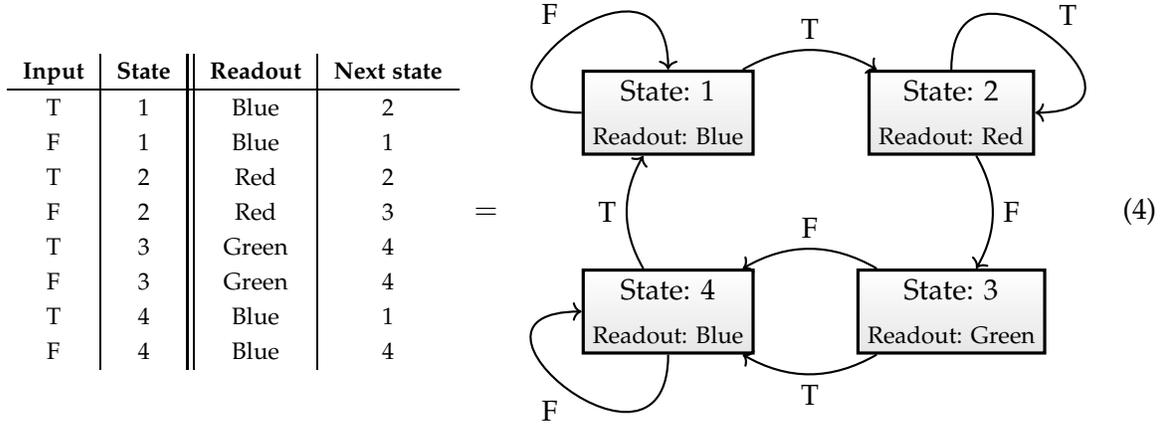
\begin{equation}\label{eqn:my_state_machine}
\parbox{2.1in}{\footnotesize
\begin{tabular}{c | c || c | c}
\textbf{Input}&\textbf{State}&\textbf{Readout}&\textbf{Next state}\\\hline
T&1&Blue&2\\
F&1&Blue&1\\
T&2&Red&2\\
F&2&Red&3\\
T&3&Green&4\\
F&3&Green&4\\
T&4&Blue&1\\
F&4&Blue&4
\end{tabular}
}
\hspace{.3in}
=
\hspace{-.4in}
\parbox{3.5in}{
\begin{tikzpicture}[
	box/.style={
		rectangle,
		minimum size=6mm,
		very thick,
		draw=black, 
		top color=white, 
		bottom color=white!50!black!20, 
		align=center,
		font=\normalfont
	}]
	\node [box] at (0,0) (st1) {State: 1\\\footnotesize Readout: Blue};
	\node [box, right = 1.5 of st1] (st2) {State: 2\\\footnotesize Readout: Red};
	\node [box, below = 1.5 of st2] (st3) {State: 3\\\footnotesize Readout: Green};
	\node [box, below = 1.5 of st1] (st4) {State: 4\\\footnotesize Readout: Blue};
	\draw[->,thick, bend left] (st1) edge["T"] (st2);
	\draw[->,thick, bend left] (st2) edge["F"] (st3);
	\draw[->,thick, bend right] (st3) edge["F"'] (st4);
	\draw[->,thick, bend left] (st3) edge ["T"] (st4);
	\draw[->,thick, bend left] (st4) edge ["T"] (st1);
	\draw[->,thick] (st1) edge [out=180, in=90,looseness=4,"F"] (st1);
	\draw[->,thick] (st2) edge [out=90, in=0,looseness=4,"T"] (st2);
	\draw[->,thick] (st4) edge [out=270, in=180,looseness=4,"F"] (st4);
	\end{tikzpicture}
}
\end{equation}
The state set is $S=\{\tn{1, 2, 3, 4}\}$. The "Readout" column depends only on the state (hence the pairs of repeated values); it represents the function $\rdt{f}\colon S\to B$. The "Next state" column depends on the input and the state; it represents the update function $\upd{f}\colon A\times S\to S$.

\end{example}

\begin{example}\label{ex:serial}

Let $(S_1,\rdt{f_1},\upd{f_1})$ and $(S_2,\rdt{f_2},\upd{f_2})$ be discrete systems on $X_1$ and $X_2$ respectively. Here we define their serial composition $(T,\rdt{g},\upd{g})$ on $Y$, shown diagrammatically below:
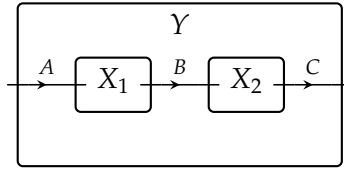
\[
\begin{tikzpicture}[oriented WD, bbx=1em, bby=1ex]
 \node[bb={1}{1},bb name=$X_1$] (X1) {};
 \node[bb={1}{1},right =2 of X1, bb name=$X_2$] (X2) {};
 \node[bb={1}{1}, fit={($(X1.north west)+(-1,3)$) ($(X1.south)+(0,-3)$) ($(X2.east)+(1,0)$)}, bb name = $Y$] (Y) {};
 \draw[ar] (Y_in1') to (X1_in1);
 \draw[ar] (X1_out1) to (X2_in1);
 \draw[ar] (X2_out1) to (Y_out1');
 \draw[label] 
	node at ($(Y_in1')!.5!(X1_in1)+(0,7pt)$)  {$A$}
	node at ($(X1_out1)!.5!(X2_in1)+(0,7pt)$)   {$B$}
	node at ($(X2_out1)!.5!(Y_out1')+(0,7pt)$)  {$C$};
\end{tikzpicture}
\]
To begin, suppose that the following four functions have been defined:
\begin{equation}\label{eqn:some_DDSs}
\begin{array}{lll}
	\rdt{f_1}\colon S_1\to B &&\upd{f_1}\colon A\times S_1\to S_1\\
	\rdt{f_2}\colon S_2\to C&&\upd{f_2}\colon B\times S_2\to S_2 
\end{array}
\end{equation}
For the composite system (in $Y$), define $T\coloneqq S_1\times S_2$, so that a state of the composite system is a pair $(s_1,s_2)$, where $s_1\in S_1$ and $s_2\in S_2$. Define the required functions for the composite system as follows:
\[
\begin{array}{lll}
	\rdt{g}\colon S_1\times S_2\to C&&\upd{g}\colon A\times S_1\times S_2\to S_1\times S_2\\
	\rdt{g}(s_1,s_2)\coloneqq\rdt{f_2}(s_2)&&\upd{g}(a,s_1,s_2)\coloneqq \Big(\upd{f_1}(a,s_1),\upd{f_2}\big(\rdt{f_1}(s_1),s_2\big)\Big)
\end{array}
\]

\end{example}

\begin{example}\label{ex:parallel}

Let $(S_1,\rdt{f_1},\upd{f_1})$ and $(S_2,\rdt{f_2},\upd{f_2})$ be discrete systems on $X_1$ and $X_2$ respectively. Here we define their parallel composition $(T,\rdt{g},\upd{g})$ on $Y$, shown diagrammatically below:
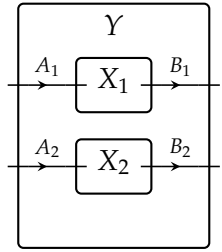
\[
\begin{tikzpicture}[oriented WD, bbx=1em, bby=1ex]
 \node[bb={1}{1},bb name=$X_1$] (X1) {};
 \node[bb={1}{1},below =2 of X1, bb name=$X_2$] (X2) {};
 \node[bb={2}{2}, fit={($(X1.north west)+(-1,3)$) ($(X2.south)+(0,-3)$) ($(X2.east)+(1,0)$)}, bb name = $Y$] (Y) {};
 \draw[ar] (Y_in1') to (X1_in1);
 \draw[ar] (Y_in2') to (X2_in1);
 \draw[ar] (X1_out1) to (Y_out1');
 \draw[ar] (X2_out1) to (Y_out2');
 \draw[label] 
	node at ($(Y_in1')!.5!(X1_in1)+(0,7pt)$)  {$A_1$}
	node at ($(Y_in2')!.5!(X2_in1)+(0,7pt)$)   {$A_2$}
	node at ($(X1_out1)!.5!(Y_out1')+(0,7pt)$)  {$B_1$}
	node at ($(X2_out1)!.5!(Y_out2')+(0,7pt)$) {$B_2$};
\end{tikzpicture}
\]
Suppose that the discrete systems on $X_1$ and $X_2$ have been defined, analogously to as in \eqref{eqn:some_DDSs}. For the composite system (in $Y$), define $T\coloneqq S_1\times S_2$, and define the required functions as follows:
\[
\begin{array}{lll}
	\rdt{g}\colon S_1\times S_2\to B_1\times B_2&&\upd{g}\colon A_1\times A_2\times S_1\times S_2\to S_1\times S_2\\
	\rdt{g}(s_1,s_2)\coloneqq\big(\rdt{f_1}(s_1),\rdt{f_2}(s_2)\big)&&\upd{g}(a_1,a_2,s_1,s_2)\coloneqq \big(\upd{f_1}(a_1,s_1),\upd{f_2}(a_2,s_2)\big)
\end{array}
\]

\end{example}

\begin{example}\label{ex:split}

Let $(S,\rdt{f},\upd{f})$ be a discrete system on $X$. Here we show what happens when wires split, in one of two ways, to form a discrete system $(T,\rdt{g},\upd{g})$ on $Y$, as shown diagrammatically below:
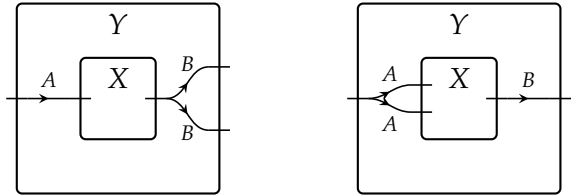
\[
\begin{tikzpicture}[oriented WD, bbx=1em, bby=1ex, bb port sep=3 pt]
 	\node[bb={1}{1},bb name=$X$] (X) {};
	\node[bb={1}{2}, fit={(X) ($(X.north east)+(1.2,3)$) ($(X.south west)+(-1.2,-3)$)}, bb name = $Y$] (Y) {};
	\draw[ar] (Y_in1') to (X_in1);
	\draw[ar] (X_out1) to (Y_out1');
	\draw[ar] (X_out1) to (Y_out2');
	\draw[label] 
		node at ($(Y_in1')!.5!(X_in1)+(0,7pt)$)  {$A$}
		node at ($(X_out1)!.5!(Y_out1')+(0,7pt)$)   {$B$}
		node at ($(X_out1)!.5!(Y_out2')-(0,7pt)$)   {$B$};
\end{tikzpicture}
\qquad\qquad
\begin{tikzpicture}[oriented WD, bbx=1em, bby=1ex, bb port sep=2 pt]
 	\node[bb={2}{1},bb name=$X$] (X) {};
	\node[bb={1}{1}, fit={(X) ($(X.north east)+(1.2,3)$) ($(X.south west)+(-1.2,-3)$)}, bb name = $Y$] (Y) {};
	\draw[ar] (Y_in1') to (X_in1);
	\draw[ar] (Y_in1') to (X_in2);
	\draw[ar] (X_out1) to (Y_out1');
	\draw[label] 
		node at ($(Y_in1')!.5!(X_in1)+(0,7pt)$)  {$A$}
		node at ($(Y_in1')!.5!(X_in2)-(0,7pt)$)   {$A$}
		node at ($(X_out1)!.5!(Y_out1')+(0,7pt)$)   {$B$};
\end{tikzpicture}
\]
Suppose that the discrete system for $X$ has been defined, analogous to that of $f_1$ in \eqref{eqn:some_DDSs}. In each case, define $T\coloneqq S$ for the composite system (in $Y$). In the left-hand (split-after $X$) case, define the required functions as follows:
\[
\begin{array}{lll}
	\rdt{g}\colon S\to B\times B&&\upd{g}\colon A\times S\to S\\
	\rdt{g}(s)\coloneqq\big(\rdt{f}(s),\rdt{f}(s)\big)&&\upd{g}(a,s)\coloneqq \upd{f}(a,s)
\end{array}
\]
In the right-hand (split-before $X$) case, define the required functions as follows:
\[
\begin{array}{lll}
	\rdt{g}\colon S\to B&&\upd{g}\colon A\times S\to S \\
	\rdt{g}(s)\coloneqq\rdt{f}(s)&&\upd{g}(a,s)\coloneqq \upd{f}(a,a,s)
\end{array}
\]

\end{example}

\begin{example}\label{ex:feedback}

Let $(S,\rdt{f},\upd{f})$ be a discrete system on $X$. Here we show what happens when there is feedback, to form a discrete system $(T,\rdt{g},\upd{g})$ on $Y$, as shown diagrammatically below:
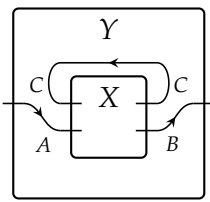
\[
\begin{tikzpicture}[oriented WD, bbx=1em, bby=1ex]
	\node[bb={2}{2}, bb name=$X$] (dom) {};
	\node[bb={1}{1}, fit={(dom) ($(dom.north east)+(1,4)$) ($(dom.south west)-(1,2)$)}, bb name = $Y$] (cod) {};
	\draw[ar,pos=20] (cod_in1') to (dom_in2);
	\draw[ar,pos=2] (dom_out2) to (cod_out1');
	\draw[ar] let \p1=(dom.north east), \p2=(dom.north west), \n1={\y2+\bby}, \n2=\bbportlen in (dom_out1) to[in=0] (\x1+\n2,\n1) -- (\x2-\n2,\n1) to[out=180] (dom_in1);
	\draw[label] 
		node[below left=2pt and 3pt of dom_in2]{$A$}
		node[below right=2pt and 3pt of dom_out2]{$B$}
		node[above left=4pt and 6pt of dom_in1] {$C$}
		node[above right=4pt and 6pt of dom_out1] {$C$};
\end{tikzpicture}
\]

To begin, suppose that the following functions have been defined:
\begin{equation*}
\begin{array}{lll}
	\rdt{f}\colon S\to B\times C&&\upd{f}\colon A\times C\times S\to S 
\end{array}
\end{equation*}
We will need to refer to the coordinate projections $\rdt{f}_B\colon S\to B$ and $\rdt{f}_C\colon S\to C$ of $\rdt{f}$, i.e., where $\rdt{f}=(\rdt{f}_B,\rdt{f}_C)$. Then for the composite system (in $Y$), define $T\coloneqq S$, and define the required functions as follows:
\[
\begin{array}{lll}
	\rdt{g}\colon S\to B&&\upd{g}\colon A\times S\to S\\
	\rdt{g}(s)\coloneqq \rdt{f}_B(s)&&\upd{g}(a,s)\coloneqq \upd{f}\big(a,\rdt{f}_C(s),s\big)
\end{array}
\]

\end{example}

\subsection{Discrete systems act as stream processors}

It is easy to see that initialized discrete systems can transform streams of input into streams of output. We briefly explain how this works. Suppose we have an initialized $(A,B)$-discrete system $(S,s_0,\rdt{f},\upd{f})$ inhabiting $\smallBox{A}{B}$. Given an input stream $(a_0,a_1,a_2\ldots)$ we can produce an state stream $(s_0,s_1,s_2,\ldots)$, where
\[s_{i+1}=\upd{f}(a_i,s_i)\]
and hence an output stream $(b_0,b_1,b_2,\ldots)$, where $b_i=\rdt{f}(s_i)$. 

\begin{example}\label{ex:stream_processor}

Consider the discrete system $(S,\rdt{f},\upd{f})$ given in Example~\ref{ex:discrete_system}, and say that the initial state is State 1. Using the formula above, this initialized $(A,B)$-discrete system can process any stream in $A=\{\tn{T, F}\}$ and produce an output stream in $B=\{\tn{Red, Blue, Green}\}$. 

For example, let $\sigma=\tn{[T, T, F, T, F]}\in\Strm(A)$ be an input stream. From it, the initialized discrete system of \eqref{eqn:my_state_machine} produces the state stream
\[(\tn{State 1, State 2, State 2, State 3, State 4, State 4})\]
and outputs the $B$-stream
\[(\tn{Blue, Red, Red, Green, Blue, Blue}).\]

\end{example}

\section{Measurable dynamical systems}\label{sec:MDS}

A slight modification of Definition~\ref{def:DDS} is useful, so that we are able to measure steady states more generally than merely by counting them. To do this, we need just a bit of measure theory. We loosely follow \cite{Bogachev_Measure_Theory}. Readers with less advanced mathematical background are invited to skim or skip to Section~\ref{sec:Mat}.

\begin{definition}\label{def:measurable_spaces}

Let $X$ be a set, and $\PP(X)$ its set of subsets. A \emph{$\sigma$-algebra on $X$} is a subset $\Sigma\ss\PP(X)$ that contains the empty set, is closed under taking complements, and is closed under taking countable unions. We say that a subset $S\ss X$ is \emph{measurable} if $S\in \Sigma$. A \emph{measurable space} is a pair $(X,\Sigma)$, where $X$ is a set and $\Sigma$ is a $\sigma$-algebra on $X$. A \emph{measurable function from $(X,\Sigma_X)$ to $(Y,\Sigma_Y)$} is a function $f\colon X\to Y$ such that if $V\in\Sigma_Y$ is measurable then its preimage $f^{-1}(V)\in\Sigma_X$ is also measurable. 

Let $(X,\Sigma)$ be a measurable space, and recall that $\RR_+=\RR_{\geq 0}\cup\{\infty\}$. A function $\mu\colon X\to\RR_+$ is called a \emph{measure} if the following holds for any empty, finite, or countably infinite set $I$: if we have a set $E_i\ss X$ for each $i\in I$, and if these sets are pairwise disjoint (i.e., if $i\neq j$ then $E_i\cap E_j=\emptyset$), then we have $\mu\left(\bigcup_{i\in I}E_i\right)=\sum_{i\in I}\mu(E_i)$; in particular, $\mu(\emptyset)=0$.

A measurable space is called \emph{countably-separated} if there is a countable subset $\mathcal{A}\ss\Sigma$ such that: if $x\neq y\in X$ are distinct points then there exists $A\in\mathcal{A}$ such that $x\in A$ and $y\not\in A$. We define the \emph{category of countably-separated measurable spaces}, which we denote $\Meas$, to have countably-separated measurable spaces as objects and measurable functions as morphisms. If $S$ is equipped with a measure $\mu\colon\Sigma\to\RR_+$, we call it a \emph{countably-separated measure space}.

\end{definition}

What makes $\Meas$ a good category for us is that it is closed under finite products and that one can measure fixed point (steady state) sets.

\begin{proposition}\label{prop:csms}

The category $\Meas$ of countably-separated measurable spaces has the following properties:
\begin{compactenum}
	\item If $T$ is a second-countable Hausdorff topological space (e.g., a manifold) then its Borel measurable space is countably-separated, $(T,\Sigma_T)\in\Meas$.
	\item The category $\Meas$ is closed under taking finite (in fact, countable) products.
	\item For any object $X\in\Meas$ and element $x\in X$, the singleton $\{x\}$ is measurable.
	\item For any object $X\in\Meas$, the diagonal $X\ss X\times X$ is measurable.
	\item If $f\colon X\to X$ is a morphism in $\Meas$ then the fixed point set $\{x\in X\mid f(x)=x\}\ss X$ is measurable.
\end{compactenum}
\end{proposition}

\begin{proof}
We go through each in turn.
\begin{compactenum}
	\item Let $\mathcal{A'}$ be the countable base of open sets in $X$, and let $\mathcal{A}=\mathcal{A'}\cup\{(X-U)\mid U\in\mathcal{A}'\}$ be its union with the complementary (closed) subsets of $X$. Then $\mathcal{A}$ separates points in $X$. 
	\item\label{enum:products} See \cite[343H.(v)]{Fremlin_Measure_Theory}.
	\item See \cite[Theorem 6.5.7]{Bogachev_Measure_Theory}.
	\item\label{enum:diagonal} See \cite[Theorem 6.5.7]{Bogachev_Measure_Theory}.
	\item The graph $\Gamma(f)\colon X\to X\times X$ of $f$, sending $x$ to $(x,f(x))$, is measurable by \eqref{enum:products}, and the fixed point set is the preimage of the diagonal, which is measurable by \eqref{enum:diagonal}.
\end{compactenum}\end{proof}

\begin{definition}\label{def:MDS}

Let $A,B\in\Meas$ be countably-separated measurable spaces. Define an \emph{$(A,B)$-open measurable dynamical system} or \emph{$(A,B)$-measurable system} for short, to be a four-tuple $(S,\mu,\rdt{f},\upd{f})$, where 
\begin{compactitem}
\item $(S,\mu)$ is a countably-separated measure space, called the \emph{state space}, 
\item $\rdt{f}\colon S\to B$ is a measurable function, called the \emph{readout function}, and 
\item $\upd{f}\colon A\times S\to S$ is a measurable function, called the \emph{update function}.
\end{compactitem}
Let $\MS(A,B)$ denote the set of $(A,B)$-measurable systems. 

\end{definition}

Every measurable system has an \emph{underlying discrete system} (see Definition~\ref{def:DDS}). Thus we can define initialized measurable systems by including an initial state $s_0\in S$ in the data above, and these can serve as stream processors as in Example~\ref{ex:stream_processor}. The steady states of a measurable system are the same as those of its underlying discrete system (Definition~\ref{def:steady_state_DDS}).

\begin{remark}

We can recover Definition~\ref{def:DDS} from Definition~\ref{def:MDS} when $A$, $B$, and $S$ are countable sets. In this case, considering them as discrete topological spaces, which are always Hausdorff, $A$, $B$, and $S$ will also be second countable. Then $S$ can be given the counting measure, and every function between discrete measurable spaces is measurable.

\end{remark}

Serial, parallel, splitting, and feedback composition for measurable systems works exactly as they do for discrete systems (see Section~\ref{sec:DDS}), as will be shown formally in Section~\ref{sec:formal_specs}. 

\section{Linear dynamical systems}\label{sec:LDS}

By $\Lin$ we mean the category of finite dimensional real vector spaces $\RR^n$, where $n\in\NN$, and linear maps between them (in fact any field in place of $\RR$ will do just as well). The following definition has been adapted from \cite[\S 5]{VSL}.

\begin{definition}\label{def:LDS}
Let $A,B\in\Lin$ be finite dimensional real vector spaces. We define an \emph{$(A,B)$-open linear dynamical system}, or \emph{$(A,B)$-linear system} for short, to be a 4-tuple $(S,\inp{M},\midp{M},\outp{M})$, where
\begin{compactitem}
\item $S\in\Lin$ is a finite dimensional real vector space, called the \emph{state space},
\item $\inp{M}\colon A\to S$,\; $\midp{M}\colon S\to S$,\; and\; $\outp{M}\colon S\to B$ are linear transformations, called the \emph{dynamic components}.
\end{compactitem}
Let $\LS(A,B)$ denote the set of all $(A,B)$-linear systems.

It is sometimes useful to choose a basis for each vector space and consider the linear maps as matrices. If $A\cong\RR^k$, $B\cong\RR^\ell$, and $S\cong\RR^n$, the three matrices can be arranged as blocks in a $(n+\ell)\times(n+k)$-matrix of the form:
\[
\left(
\begin{array}{l | l}
	\midp{M}&\inp{M}\\\hline
	\outp{M}&0
\end{array}
\right)
\]
\end{definition}

\begin{example}\label{ex:serial_linear_comp}

Let $(S_1,\inp{M_1},\midp{M_1},\outp{M_1})$ and $(S_2,\inp{M_2},\midp{M_2},\outp{M_2})$ be linear systems on $X_1$ and $X_2$, respectively. Here we define their serial composition $$(T,\inp{N},\midp{N},\outp{N})$$ on $Y$, shown diagrammatically below:
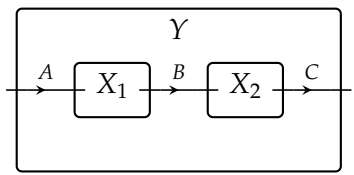
\[
\begin{tikzpicture}[oriented WD, bbx=1em, bby=1ex]
 \node[bb={1}{1},bb name=$X_1$] (X1) {};
 \node[bb={1}{1},right =2 of X1, bb name=$X_2$] (X2) {};
 \node[bb={1}{1}, fit={($(X1.north west)+(-1,3)$) ($(X1.south)+(0,-3)$) ($(X2.east)+(1,0)$)}, bb name = $Y$] (Y) {};
 \draw[ar] (Y_in1') to (X1_in1);
 \draw[ar] (X1_out1) to (X2_in1);
 \draw[ar] (X2_out1) to (Y_out1');
 \draw[label] 
	node at ($(Y_in1')!.5!(X1_in1)+(0,7pt)$)  {$A$}
	node at ($(X1_out1)!.5!(X2_in1)+(0,7pt)$)   {$B$}
	node at ($(X2_out1)!.5!(Y_out1')+(0,7pt)$)  {$C$};
\end{tikzpicture}
\]
The resulting state space will be the direct sum $T=S_1\oplus S_2$. Choosing bases for all vector spaces involved, as in Definition~\ref{def:LDS}, the dynamic components will be as in the following block $(n_1+n_2+\ell_2)\times(n_1+n_2+k_1)$-matrix
\[
\left(
\begin{array}{ll | l}
	\midp{M_1}&0					&	\inp{M_1}\\
	\inp{M_2}\outp{M_1}&\midp{M_2}	&	0\\\hline
	0&\outp{M_2}					&	0
\end{array}
\right)
\]
\end{example}

Though not entirely straightforward, there is a way to see the formula for serial composition, given in Example~\ref{ex:serial_linear_comp} as analogous to the formula in Example~\ref{ex:serial}. The examples for parallel, splitting, and feedback compositions are similarly adapted from Examples~\ref{ex:parallel},~\ref{ex:split},~and~\ref{ex:feedback}, and a complete formula---that works for an arbitrary wiring diagram---will be given in Section~\ref{sec:formal_specs}.

\section{Continuous dynamical systems}\label{sec:CDS}

By $\Euc$ we mean the category of finite dimensional Euclidean spaces $\RR^n$, where $n\in\NN$, and smooth maps between them. Each such Euclidean space $S$ has the structure of a manifold, an in particular has a \emph{tangent bundle} $TS\cong S\times S$ \cite{Warner}. A point in $T\RR^n$ is a pair $(p,v)$ where $p\in \RR^n$ is a point and $v\in\RR^n$ is a vector emanating from that point. There is a smooth \emph{projection} function $\pi_S\colon TS\to S$ for any $S\in\Euc$, given by $\pi_S(p,v)=p$.

A smooth function $f\colon S\to TS$ assigns to each point $p\in S$ a pair $f(p)=(q,v)$, i.e., a vector at some possibly different point $q\in S$. Requiring it to be the same point, $p=q$, is the same as requiring that $\pi_S\circ f=\id_S$, in which case $f$ is called a \emph{vector field on $S$}. 

Let $A\in\Euc$ be another Euclidean space. An \emph{$A$-parameterized vector field on $S$} is a smooth function $f\colon A\times S\to TS$ such that $f(a,p)=(q,v)$ where $p=q$. This is summarized in the following \emph{commutative diagram}, where $\pr_S\colon A\times S\to S$ is the coordinate projection:
\begin{equation}\label{eqn:slice_Euc}
\begin{tikzcd}
	A\times S\ar[r,"f"]\ar[dr,"\pr_S"']&TS\ar[d,"\pi_S"]\\
	&S
\end{tikzcd}
\end{equation}
A vector field can be identified with a $\RR^0$-parameterized vector field, in which case $\pr_S\cong\id_S$. Concretely, we may denote an $A$-parameterized vector field $f\colon A\times\RR^n\to T\RR^n$ as on the left: 
\begin{equation}\label{eqn:pvf}
\begin{aligned}
	\dot{x}_1&=f_1(a,x_1,\ldots,x_n),\\
	\dot{x}_2&=f_2(a,x_1,\ldots,x_n),\\
	&\vdots\\
	\dot{x}_n&=f_n(a, x_1,\ldots,x_n),
\end{aligned}
\end{equation}
where each $f_i\colon\RR^n\to\RR$ is a smooth function and where $\dot{x}_i$ means $\frac{dx_i}{dt}$.

The following definition has been adapted from \cite{VSL}, where the authors consider open continuous dynamical systems and wiring diagrams. We use the term "state space" rather than the more typical "phase space", to fit with the nomenclature for discrete dynamical systems. We refer the reader to \cite{Strogatz} for more on dynamical systems.

\begin{definition}\label{def:CDS}

Let $A,B\in\Euc$ be Euclidean spaces. We define an \emph{$(A,B)$-open continuous dynamical system}, or \emph{$(A,B)$-continuous system} for short, to be a 3-tuple $(S,\rdt{f},\dyn{f})$, where 
\begin{compactitem}
\item $S\in\Euc$ is a Euclidean space, called the \emph{state space}, 
\item $\rdt{f}\colon S\to B$ is a smooth function, called the \emph{readout function}, and
\item $\dyn{f}\colon A\times S\to TS$ is an $A$-parameterized vector field on $S$ as in \eqref{eqn:pvf}, called the \emph{dynamics}.
\end{compactitem}
Let $\CS(A,B)$ denote the set of all $(A,B)$-continuous systems.

\end{definition}

For a continuous system $(S,\rdt{f},\dyn{f})$, the dynamics $\dyn{f}$ is plays the same role as the update function $\upd{f}$ does in a discrete dynamical system. Recall that, given a continuous system such as the one shown in \eqref{eqn:pvf} and a parameter $a\in A$, a tuple $(x_1,\ldots,x_n)$ is called a steady state, or equilibrium, if all derivatives vanish $\dot{x}_i=0$ there, i.e., $f_i(a,x_1,\ldots,x_n)=0$ for all $1\leq i\leq n$.

\begin{example}

Let $(S_1,\rdt{f_1},\dyn{f_1})$ and $(S_2,\rdt{f_2},\dyn{f_2})$ be continuous systems on $X_1$ and $X_2$, respectively. Here we define their serial composition $(T,\rdt{g},\dyn{g})$ on $Y$, shown diagrammatically below:
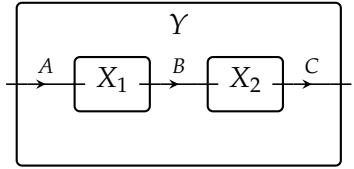
\[
\begin{tikzpicture}[oriented WD, bbx=1em, bby=1ex]
 \node[bb={1}{1},bb name=$X_1$] (X1) {};
 \node[bb={1}{1},right =2 of X1, bb name=$X_2$] (X2) {};
 \node[bb={1}{1}, fit={($(X1.north west)+(-1,3)$) ($(X1.south)+(0,-3)$) ($(X2.east)+(1,0)$)}, bb name = $Y$] (Y) {};
 \draw[ar] (Y_in1') to (X1_in1);
 \draw[ar] (X1_out1) to (X2_in1);
 \draw[ar] (X2_out1) to (Y_out1');
 \draw[label] 
	node at ($(Y_in1')!.5!(X1_in1)+(0,7pt)$)  {$A$}
	node at ($(X1_out1)!.5!(X2_in1)+(0,7pt)$)   {$B$}
	node at ($(X2_out1)!.5!(Y_out1')+(0,7pt)$)  {$C$};
\end{tikzpicture}
\]

To begin, suppose that the following four functions have been defined:
\begin{equation*}
\begin{array}{lll}
	b=\rdt{f_1}(x_1)&&\dot{x}_1= \dyn{f_1}(a,x_1)\\
	c=\rdt{f_2}(x_2)&&\dot{x}_2= \dyn{f_2}(b,x_2)
\end{array}
\end{equation*}
Here $x_1$ and $x_2$ are variables representing state vectors of arbitrary dimensions $n_1$ and $n_2$. For the composite system (in $Y$), the state variables are $x_1$ and $x_2$. The required formulas for the readout and dynamics are:
\begin{equation*}
\begin{array}{lll}
	&&\dot{x}_1= \dyn{f_1}(a,x_1) \\
	c=\rdt{f_2}(x_2)&&\dot{x}_2= \dyn{f_2}\big(\rdt{f_1}(x_1),x_2\big)
\end{array}
\end{equation*}

\end{example}

The examples for parallel, splitting, and feedback compositions are similarly adapted from Examples~\ref{ex:parallel},~\ref{ex:split},~and~\ref{ex:feedback}, and a complete formula will be given in Section~\ref{sec:formal_specs}.

\subsection{Steady states and $\epsilon$-approximation of continuous systems}

Regarding a continuous system in terms of its discrete approximation is compositional---it respects wiring diagrams of all sorts---as will be shown in Theorem~\ref{thm:approximation_CSDS}. Another compositional mapping is to regard each continuous system in terms of its bifurcation diagram. We will introduce these topics now, though the idea will be fleshed out in Section~\ref{sec:formal_specs}.

\begin{construction}\label{const:Euler}

Let $A$ and $B$ be spaces, and let $|A|$ and $|B|$ be their underlying sets. Let $f=(S,\rdt{f},\dyn{f})$ be an $(A,B)$-continuous system. Then for any real number $\epsilon>0$ we can construct an $(|A|,|B|)$-discrete system $(|S|,\rdt{f_\epsilon},\upd{f_\epsilon})$, called \emph{the $\epsilon$-approximation of $f$} as follows. For readouts define $\rdt{f_\epsilon}\coloneqq\rdt{f}$, and for updates use Euler's method:
\[
\upd{f_\epsilon}(a,x)\coloneqq x+\epsilon\cdot\dyn{f}(a,x).
\]

\end{construction}

\begin{definition}\label{def:steady_state_CDS}

Let $A,B\in\Euc$ be Euclidean spaces, and let $F=(S,\rdt{f},\dyn{f})$ be an $(A,B)$-continuous system. For $a\in A$ and $b\in B$, define an \emph{$(a,b)$-steady state} to be a state $s\in S$ such that $\dyn{f}(a,s)=0$ and $\rdt{f}(s)=b$. We denote the set of all $(a,b)$-steady states by
\[\StstS(F)_{a,b}\coloneqq\left\{s\in S\;\;\middle|\;\;\rdt{f}(s)=b,\quad\upd{f}(a,s)=0\right\}\]
and its count by $\Stst(F)_{a,b}\coloneqq\#\StstS(F)_{a,b}$.

\end{definition}

\begin{remark}\label{rem:bifurcation}

In case it is not clear, Definition~\ref{def:steady_state_CDS} is strongly related to the notion of \emph{bifurcation diagrams} \cite{Strogatz}, as we now explain.

Let $A,S\in\Euc$ be Euclidean spaces, and let $\dyn{f}\colon A\times S\to S$ be smooth. Suppose we take $B=S$, so the readout function can be the identity, $\rdt{f}=\id_B$, and let $F=(S,\rdt{f},\dyn{f})$, so we have $\Stst(F)\colon A\times B\to\NN_+$. However, for any $(a,b)\in A\times B$, the number of steady states $\StstS(F)(a,b)$ is either zero or one, because $\rdt{f}$ is injective. Thus the set of steady states can be drawn on an $A\times B$ coordinate system, by plotting a point at $(a,b)$ if and only if it is a steady state (or equilibrium). This almost gives the bifurcation diagram of the system, the exception being that it does not address stability issues, a discussion we save for Section~\ref{sec:linear_stability}. A major thrust of this paper is to show that when bifurcation diagrams are considered as matrices (see Corollary~\ref{cor:CS_steady_state}), they can be composed by matrix arithmetic when the corresponding dynamical systems are coupled via a wiring diagram. The matrix arithmetic of which we speak is discussed next, in Section~\ref{sec:Mat}.

\end{remark}

\section{Matrices (and wiring diagrams)}\label{sec:Mat}

We can also interpret boxes in a wiring diagram as being inhabited by matrices, whereby serial composition corresponds to matrix multiplication, parallel composition corresponds to matrix tensor (Kronecker) product, and feedback corresponds to partial trace. In this section we give several examples; a complete formula is given in Section~\ref{sec:formal_specs}. 

In order to multiply, tensor, or trace matrices, the entries do not have to be real or complex numbers. All that is necessary are associative addition and multiplication operations, and a distributive law. The mathematical object that handles this sort of thing is called a \emph{commutative semiring}. In fact, to be as general as possible we will want to allow infinite matrices, so we need to know what it means to add infinitely many numbers. For this we can extend to \emph{complete semirings}; see \cite{droste2009semirings} for details. Two important cases of complete semirings are: 
\begin{itemize}
\item the set of \emph{extended natural numbers}, denoted $\NN_+\coloneqq\NN\cup\{\infty\}$, and 
\item the set of \emph{extended real numbers} is $\RR_+\coloneqq\RR_{\geq 0}\cup\{\infty\}$.
\end{itemize}
In any complete semiring $R$, infinity plus anything is infinity; that is, for all $r\in R$, we have $r+\infty=\infty$. Also, $0\cdot\infty=0$ but for $r\neq 0$ we have $r\cdot\infty=\infty$.

\begin{definition}

Let $R$ be a complete semiring. For sets $A,B$, define an \emph{$(A,B)$-matrix in $R$} to be a function $M\colon A\times B\to R$. For elements $a\in A$ and $b\in B$, we refer to $M(i,j)\in R$ as the $(i,j)$-entry, and often denote it $M_{i,j}$. We denote the set of $(A,B)$-matrices of extended natural (resp. real) numbers by $\Mat_R(A,B)$. By default, we write $\Mat(A,B)$ when $R=\NN_+$.

\end{definition}

\begin{remark}

If $A$ and $B$ are finite sets, then a choice of total order on $A$ and $B$ is the same thing as a pair of bijection $A\cong\{1,2,\ldots,m\}$ and $B\cong\{1,2,\ldots,n\}$. This identification allows us to show the matrix as an array in the usual fashion:
\[
\left(
\begin{array}{cccc}
M_{1,1}&M_{1,2}&\cdots&M_{1,n}\\
M_{2,1}&M_{2,2}&\cdots&M_{2,n}\\
\vdots&\vdots&\ddots&\vdots\\
M_{m,1}&M_{m,2}&\cdots&M_{m,n}
\end{array}
\right)
\]
In Definitions~\ref{def:Mat_parallel}~and~\ref{def:Mat_wiring} we will give definitions for matrix manipulations (such as multiplication, Kronecker product, and trace), which are independent of ordering.

\end{remark}

\begin{example}\label{ex:serial_Mat}
We will give examples of matrices $M_1$ and $M_2$ inhabiting $X_1$ and $X_2$ and their serial composition $N$ inhabiting $Y$, shown diagrammatically below:
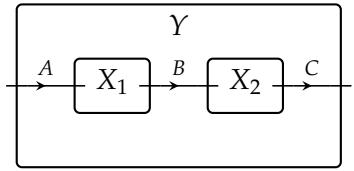
\[
\begin{tikzpicture}[oriented WD, bbx=1em, bby=1ex]
 \node[bb={1}{1},bb name=$X_1$] (X1) {};
 \node[bb={1}{1},right =2 of X1, bb name=$X_2$] (X2) {};
 \node[bb={1}{1}, fit={($(X1.north west)+(-1,3)$) ($(X1.south)+(0,-3)$) ($(X2.east)+(1,0)$)}, bb name = $Y$] (Y) {};
 \draw[ar] (Y_in1') to (X1_in1);
 \draw[ar] (X1_out1) to (X2_in1);
 \draw[ar] (X2_out1) to (Y_out1');
 \draw[label] 
	node at ($(Y_in1')!.5!(X1_in1)+(0,7pt)$)  {$A$}
	node at ($(X1_out1)!.5!(X2_in1)+(0,7pt)$)   {$B$}
	node at ($(X2_out1)!.5!(Y_out1')+(0,7pt)$)  {$C$};
\end{tikzpicture}
\]
Suppose that $|A|=2$, $|B|=2$, $|C|=3$, and let $M_1$ and $M_2$ be the following matrices:
\[
M_1\coloneqq\left(
\begin{array}{cc}
1&2\\
3&0
\end{array}
\right)
\qquad\qquad
M_2\coloneqq\left(
\begin{array}{ccc}
2&2&0\\
3&1&1
\end{array}
\right)
\]
Then their serial composition is just the usual matrix product $N=M_1M_2$, 
\[
N=\left(
\begin{array}{ccc}
8&4&2\\
6&6&0\\
\end{array}
\right)
\]
\end{example}

\begin{example}\label{ex:parallel_mat}
We will give examples of matrices $M_1$ and $M_2$ inhabiting $X_1$ and $X_2$ and their parallel composition $N$ inhabiting $Y$, shown diagrammatically below:
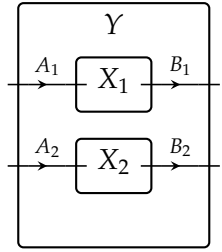
\[
\begin{tikzpicture}[oriented WD, bbx=1em, bby=1ex]
 \node[bb={1}{1},bb name=$X_1$] (X1) {};
 \node[bb={1}{1},below =2 of X1, bb name=$X_2$] (X2) {};
 \node[bb={2}{2}, fit={($(X1.north west)+(-1,3)$) ($(X2.south)+(0,-3)$) ($(X2.east)+(1,0)$)}, bb name = $Y$] (Y) {};
 \draw[ar] (Y_in1') to (X1_in1);
 \draw[ar] (Y_in2') to (X2_in1);
 \draw[ar] (X1_out1) to (Y_out1');
 \draw[ar] (X2_out1) to (Y_out2');
 \draw[label] 
	node at ($(Y_in1')!.5!(X1_in1)+(0,7pt)$)  {$A_1$}
	node at ($(Y_in2')!.5!(X2_in1)+(0,7pt)$)   {$A_2$}
	node at ($(X1_out1)!.5!(Y_out1')+(0,7pt)$)  {$B_1$}
	node at ($(X2_out1)!.5!(Y_out2')+(0,7pt)$) {$B_2$};
\end{tikzpicture}
\]
Suppose that $|A_1|=2$, $|B_1|=2$, $|A_2|=3$, and $|B_2|=2$, and let $M_1$ and $M_2$ be the following matrices:
\[
M_1\coloneqq\left(
\begin{array}{cc}
1&2\\
3&0
\end{array}
\right)
\qquad\qquad
M_2\coloneqq\left(
\begin{array}{cc}
2&2\\
3&1\\
1&0
\end{array}
\right)
\]
Then $N=M_1\otimes M_2$ is the Kronecker product \cite{SteebHardy}, 
\[
N=\left(
\begin{array}{cc|cc}
2&2&4&4\\
3&1&6&2\\
1&0&2&0\\\hline
6&6&0&0\\
9&3&0&0\\
3&0&0&0
\end{array}
\right)
\]

\end{example}

\begin{example}\label{ex:splitting_mat}
We will give examples of matrices $M_1$ and $M_2$ inhabiting $X_1$ and $X_2$ and their splitting compositions $N_1$ and $N_2$ inhabiting $Y_1$ and $Y_2$, shown diagrammatically below:
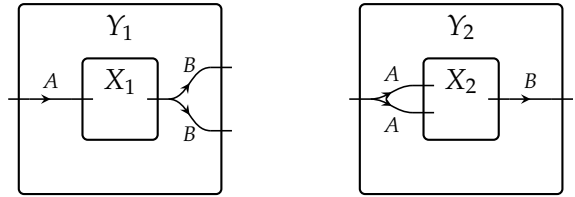
\[
\begin{tikzpicture}[oriented WD, bbx=1em, bby=1ex, bb port sep=3 pt]
 	\node[bb={1}{1},bb name=$X_1$] (X) {};
	\node[bb={1}{2}, fit={(X) ($(X.north east)+(1.2,3)$) ($(X.south west)+(-1.2,-3)$)}, bb name = $Y_1$] (Y) {};
	\draw[ar] (Y_in1') to (X_in1);
	\draw[ar] (X_out1) to (Y_out1');
	\draw[ar] (X_out1) to (Y_out2');
	\draw[label] 
		node at ($(Y_in1')!.5!(X_in1)+(0,7pt)$)  {$A$}
		node at ($(X_out1)!.5!(Y_out1')+(0,7pt)$)   {$B$}
		node at ($(X_out1)!.5!(Y_out2')-(0,7pt)$)   {$B$};
\end{tikzpicture}
\qquad\qquad
\begin{tikzpicture}[oriented WD, bbx=1em, bby=1ex, bb port sep=2 pt]
 	\node[bb={2}{1},bb name=$X_2$] (X) {};
	\node[bb={1}{1}, fit={(X) ($(X.north east)+(1.2,3)$) ($(X.south west)+(-1.2,-3)$)}, bb name = $Y_2$] (Y) {};
	\draw[ar] (Y_in1') to (X_in1);
	\draw[ar] (Y_in1') to (X_in2);
	\draw[ar] (X_out1) to (Y_out1');
	\draw[label] 
		node at ($(Y_in1')!.5!(X_in1)+(0,7pt)$)  {$A$}
		node at ($(Y_in1')!.5!(X_in2)-(0,7pt)$)   {$A$}
		node at ($(X_out1)!.5!(Y_out1')+(0,7pt)$)   {$B$};
\end{tikzpicture}
\]
Suppose that $|A|=2$, $|B|=3$, and let $M_1$ and $M_2$ be the following matrices (the vertical and horizontal bars below are only for ease of reading block matrices):
\[
M_1\coloneqq\left(
\begin{array}{ccc}
1&2&4\\
3&1&1
\end{array}
\right)
\qquad\qquad
M_2\coloneqq\left(
\begin{array}{ccc}
1&2&1\\
3&0&1\\\hline
2&1&2\\
0&1&4
\end{array}
\right)
\]
Then $N_1$ and $N_2$ are the matrices below:
\[
N_1\coloneqq\left(
\begin{array}{ccc|ccc|ccc}
1&0&0&0&2&0&0&0&4\\
3&0&0&0&1&0&0&0&1
\end{array}
\right)
\qquad\qquad
N_2\coloneqq\left(
\begin{array}{ccc}
1&2&1\\
0&1&4
\end{array}
\right)
\]

\end{example}

\begin{example}\label{ex:trace_mat}
We will give examples of a matrix $M$ inhabiting $X$ and its feedback composition $N$ inhabiting $Y$, shown diagrammatically below:
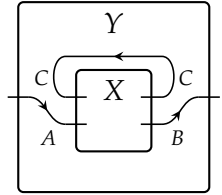
\[
\begin{tikzpicture}[oriented WD, bbx=1em, bby=1ex]
	\node[bb={2}{2}, bb name=$X$] (dom) {};
	\node[bb={1}{1}, fit={(dom) ($(dom.north east)+(1,4)$) ($(dom.south west)-(1,2)$)}, bb name = $Y$] (cod) {};
	\draw[ar,pos=20] (cod_in1') to (dom_in2);
	\draw[ar,pos=2] (dom_out2) to (cod_out1');
	\draw[ar] let \p1=(dom.north east), \p2=(dom.north west), \n1={\y2+\bby}, \n2=\bbportlen in (dom_out1) to[in=0] (\x1+\n2,\n1) -- (\x2-\n2,\n1) to[out=180] (dom_in1);
	\draw[label] 
		node[below left=2pt and 3pt of dom_in2]{$A$}
		node[below right=2pt and 3pt of dom_out2]{$B$}
		node[above left=4pt and 6pt of dom_in1] {$C$}
		node[above right=4pt and 6pt of dom_out1] {$C$};
\end{tikzpicture}
\]

Suppose that $|A|=2$, $|B|=3$, and $|C|=2$, and let $M$ be the following matrix:
\[
M\coloneqq\left(
\begin{array}{cc|cc|cc}
1&2&4&1&0&3\\
3&1&1&2&1&0\\\hline
1&2&1&0&3&2\\
0&1&2&3&4&2
\end{array}
\right)
\]
Then $N=\Tr^C_{A,B}(M)$ is the partial trace matrix, given by adding diagonals of each square block, as shown below:
\[
N\coloneqq\left(
\begin{array}{ccc}
2&6&0\\
2&4&5
\end{array}
\right)
\]

\end{example}

In general, if $M$ is a $(K\times I)\times(K\times J)$-matrix, its $K$-\emph{partial trace}, denoted $\Tr^K_{I,J}$ is the $(I\times J)$-matrix given by adding up the $K$-blocks; it is given explicitly by the formula
\begin{equation}\label{eqn:def_trace}
	\Tr^K_{I,J}(M)_{i,j}\coloneqq\sum_{k\in K}M_{(k,i),(k,j)}.
\end{equation}

\section{Introducing the compositionality of steady states}\label{sec:introducing_classification}

The classifying function $Q\colon\DS\to\Mat$ sends each discrete (or measurable) system to a matrix. What makes it interesting is that it is preserved under each type of composition: serial, parallel, splitting, and feedback. In other words, the matrix is a summary of the discrete system, but one that can be used losslessly in future computations. The following definition is completely analogous to Defintion~\ref{def:steady_state_CDS}; see Corollary~\ref{cor:CS_steady_state} for a formal comparison.

\begin{definition}\label{def:discrete_to_matrix}

Let $F=(S,\rdt{f},\upd{f})$ be an $(A,B)$-discrete system. For $a\in A$ and $b\in B$, recall the set of $(a,b)$-steady states from Definition~\ref{def:steady_state_DDS} and its count
\[
\Stst(F)_{a,b}=\#\{s\in S\mid \rdt{f}(s)=b \tn{ and } \upd{f}(a,s)=s\}
\]
We can consider this as a matrix $\Stst(F)\in\Mat(A,B)$, which we call the \emph{steady state matrix of $F$}.

\end{definition}

\begin{example}\label{ex:series_DS_matrix}

Let $A=\{\tn{T, F}\}$ and $B=\{\tn{Red, Green, Blue}\}$. In Example~\ref{ex:discrete_system} we wrote out an example of an $(A,B)$-discrete system $F_1=(S_1,\rdt{f_1},\upd{f_1})$. In this example, we put it in serial composition with a $(B,C)$-discrete system, where $C=\{\tn{Up, Down}\}$, and discuss the resulting system in terms of steady states.
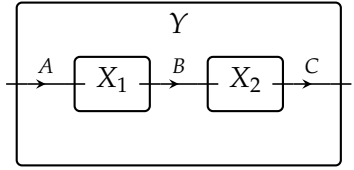
\[
\begin{tikzpicture}[oriented WD, bbx=1em, bby=1ex]
 \node[bb={1}{1},bb name=$X_1$] (X1) {};
 \node[bb={1}{1},right =2 of X1, bb name=$X_2$] (X2) {};
 \node[bb={1}{1}, fit={($(X1.north west)+(-1,3)$) ($(X1.south)+(0,-3)$) ($(X2.east)+(1,0)$)}, bb name = $Y$] (Y) {};
 \draw[ar] (Y_in1') to (X1_in1);
 \draw[ar] (X1_out1) to (X2_in1);
 \draw[ar] (X2_out1) to (Y_out1');
 \draw[label] 
	node at ($(Y_in1')!.5!(X1_in1)+(0,7pt)$)  {$A$}
	node at ($(X1_out1)!.5!(X2_in1)+(0,7pt)$)   {$B$}
	node at ($(X2_out1)!.5!(Y_out1')+(0,7pt)$)  {$C$};
\end{tikzpicture}
\]
For the second box, define $F_2=(S_2,\rdt{f_2},\upd{f_2})$ as shown here:
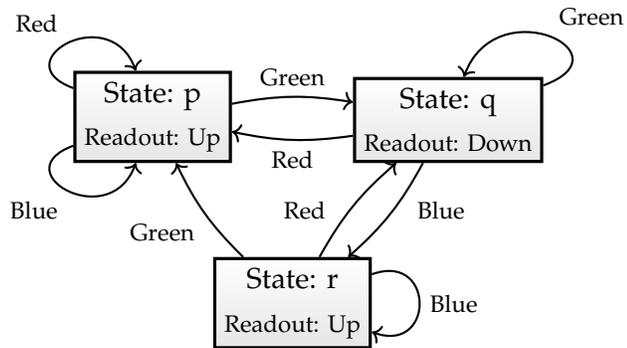
\begin{equation}\label{eqn:second_box_DS}
\parbox{2.4in}{\footnotesize
\begin{tabular}{c | c || c | c}
\textbf{Input}&\textbf{State}&\textbf{Readout}&\textbf{Next state}\\\hline
Red&p&Up&p\\
Blue&p&Up&p\\
Green&p&Up&q\\
Red&q&Down&p\\
Blue&q&Down&r\\
Green&q&Down&q\\
Red&r&Up&q\\
Blue&r&Up&r\\
Green&r&Up&p
\end{tabular}
}
\hspace{.3in}
\parbox{3.5in}{
\begin{tikzpicture}[
	box/.style={
		rectangle,
		minimum size=6mm,
		very thick,
		draw=black, 
		top color=white, 
		bottom color=white!50!black!20, 
		align=center,
		font=\normalfont
	}]
	\node [box] at (0,0) (st1) {State: p\\\footnotesize Readout: Up};
	\node [box, below right = 1.25 and -.25 of st1] (st3) {State: r\\\footnotesize Readout: Up};
	\node [box, above right=1.25 and -.25 of st3] (st2) {State: q\\\footnotesize Readout: Down};
	\draw[->,thick, bend left=10] (st1) edge["\footnotesize Green"] (st2);
	\draw[->,thick, bend left=10] (st2) edge["\footnotesize Red"] (st1);
	\draw[->,thick, bend left=10] (st2) edge["\footnotesize Blue" near start] (st3);
	\draw[->,thick, bend left=10] (st3) edge["\footnotesize Red" near start] (st2);
	\draw[->,thick, bend left=10] (st3) edge["\footnotesize Green"] (st1);
	\draw[->,thick] (st1) edge [out=160, in=110,looseness=3,"\footnotesize Red"] (st1);
	\draw[->,thick] (st1) edge [out=200, in=250,looseness=3,"\footnotesize Blue"'] (st1);
	\draw[->,thick] (st2) edge [out=20, in=70,looseness=3,"\footnotesize Green"'] (st2);
	\draw[->,thick] (st3) edge [out=20, in=-20,looseness=3,"\footnotesize Blue"] (st3);
	\end{tikzpicture}
}
\end{equation}
When the two systems are composed in series, the resulting system has twelve states (e.g., (2,p)), is driven by inputs in $\{\tn{T, F}\}$, and produces output values in $\{\tn{Up, Down}\}$. We will not write the system out here, but instead compute its matrix of steady states. Note that steady states appear as loops in \eqref{eqn:second_box_DS}.

As will be discussed more formally in Section~\ref{sec:introducing_classification}, the matrix associated to such a system organizes each of its steady states in terms of 
\begin{compactitem}
\item the inputs that it \textbf{is fixed by}, and
\item the signal that it \textbf{outputs}.
\end{compactitem}
Thus the steady state matrix for the discrete system above presents the number of steady states for each (fixed by, output) combination:
\[
\begin{tabular}{c||c|c}
\parbox{.55in}{\tiny ~\hspace{.13in}Outputs:\\Is fixed by:}
&Up&Down\\\hline\hline
Red&1&0\\\hline
Blue&2&0\\\hline
Green&0&1
\end{tabular}
\qquad\tn{i.e.,}\qquad
\left(
\begin{tabular}{cc}
1&0\\
2&0\\
0&1
\end{tabular}
\right)
\]
The steady states of the discrete system shown in \eqref{eqn:my_state_machine} are summarized by the following matrix:
\[
\begin{tabular}{c||c|c|c}
\parbox{.55in}{\tiny ~\hspace{.13in}Outputs:\\Is fixed by:}
&Red&Blue&Green\\\hline\hline
T&1&0&0\\\hline
F&0&2&0
\end{tabular}
\qquad\tn{i.e.,}\qquad
\left(
\begin{tabular}{ccc}
1&0&0\\
0&2&0
\end{tabular}
\right)
\]

Serial composition of discrete systems was discussed in Example~\ref{ex:serial}. One can check that it has 12 states, five of which are steady states, but doing so can be tedious, and if there were more than two inner boxes it would only get more difficult, as we will see in the extended example in Section~\ref{sec:extended_example}. The compositionality of the steady state function says that we can compute the steady state matrix for the combined system by multiplying the matrices associated to the subsystems. Indeed, multiplying the above matrix by that from Example~\ref{ex:discrete_system}, we have
\[
\left(
\begin{tabular}{ccc}
1&0&0\\
0&2&0
\end{tabular}
\right)
\left(
\begin{tabular}{cc}
1&0\\
2&0\\
0&1
\end{tabular}
\right)
=
\left(
\begin{tabular}{cc}
1&0\\
4&0
\end{tabular}
\right)
\]
The combined system indeed has five steady states, one of which outputs 'Up' and the other four of which output 'Down'. We know that all of these occur when the input is 'T'; an input of 'F' results in no steady states.

\end{example}

We will not give examples for the other kinds of composition, e.g., parallel and feedback composition here. However, we will give a complete formula in Section~\ref{sec:formal_specs}.

\subsection{Exponential savings from dealing directly with steady state matrices}

When dynamical systems are interconnected to form a larger system, the resulting system may require a huge amount of data, as compared to the resulting steady state matrix. Two different variables are at work here: the size of the input alphabet and the total number of states. The former tends to grow exponentially in the number of input wires, and the latter tends to grow exponentially in the number of internal boxes. The dynamical system itself grows exponentially in both, whereas the matrix of steady states grows only in the number of input wires.

For example, consider the network of neurons in Figure~\ref{fig:complex_WD}. If each input wire carries two signals (say 'resting' or 'active'), and each box carries three states (e.g., 'depolarized', 'polarized', or 'hyperpolarized') then to express the totalized dynamical system would require a table with roughly $2^43^6=11,648$ rows, whereas the matrix of steady states would require a relatively small $16\times 16$ matrix. As more internal boxes are encapsulated by the wiring diagram, an exponential savings is achieved by considering the steady state matrix, rather than the whole dynamical system.

\chapter{Category-theoretic formulation of wiring diagrams}\label{sec:CT_formulation}

In this section, we explain how wiring diagrams are expressed using sets and functions. The idea is that there are sets of ports---input and output for each box---and there are functions that specify how one port is fed by another. The only technicality is dealing with the fact that each port carries a certain alphabet of symbols, and we will need to take them into account. For example, if one port is connected to another, the two should be using the same alphabet.

In order to make these ideas precise, we use the language of category theory. We begin with a very brief background section.

\section{Category theory references}\label{sec:more_background}

We assume the reader is familiar with the basic definitions of category theory, namely \emph{categories}, \emph{functors}, and \emph{natural transformations}. For example, we will often consider $\Set$, the category of sets and functions, as well as functors $\cat{C}\to\Set$ where $\cat{C}$ is some other category.

Just to fix notation, we recall some basic definitions. A category $\cat{C}$ comes with a set $\Ob\cat{C}$ of \emph{objects}. If $X,Y\in\cat{C}$ are objects, the pair is assigned a set $\cat{C}(X,Y)$ of \emph{morphisms}; if $f\in\cat{C}(X,Y)$ is a morphism, it may be denoted $f\colon X\to Y$. The category also has an identity $\id_X\in\cat{C}(X,X)$ for each object $X$ and a composition formula $\circ\colon\cat{C}(Y,Z)\times\cat{C}(X,Y)\to\cat{C}(X,Z)$. We may write $X\in\cat{C}$ in place of $X\in\Ob\cat{C}$, e.g., we have been writing $X\in\Set$. Similarly, if $X\in\Set$ is a set, we may write $X\to\cat{C}$ to denote a function $X\to\Ob\cat{C}$. 

Some categories, such as $\Set$, are closed under taking finite products, denoted $\times$; we call such categories \emph{finite product categories}. In fact $\Set$ is also closed under taking finite coproducts (called disjoint unions and denoted $+$). We refer the reader to \cite{MacLane}, \cite{Awodey}, or \cite{CT4S} (in decreasing order of difficulty) for background on all the above ideas. 

Both products and coproducts are examples of \emph{monoidal structures} on $\Set$. We will also be interested in monoidal structures on other categories. We will also use \emph{lax monoidal functors}, which are functors that interact coherently with monoidal structures. We refer the reader to \cite{Leinster} for specific background on monoidal structures and lax monoidal functors. See also \cite{VSL} for a paper on wiring diagrams and continuous dynamical systems that uses the above ideas.

The category theory we use in this paper is not very sophisticated, and readers who are unfamiliar with category theory are encouraged to lightly skim those areas---such as Section~\ref{sec:TFS_DP}---which are purely about setting up categorical machinery, and focus instead on examples. The paper concludes with an extended example in Section~\ref{sec:extended_example}.

\section{Typed finite sets and their dependent products}\label{sec:TFS_DP}

We first want to define formally what we mean by boxes of arbitrary shape, e.g., 
\begin{equation}\label{eqn:arb_box}
\begin{tikzpicture}[oriented WD, bbx=.1cm, bby =.1cm, bb port sep=.1cm]
	\node [bb={2}{3}] (X) {$X$};
	\draw[label] 
		node[left=.1 of X_in1]  {$A_1$}
		node[left=.1 of X_in2]  {$A_2$}
		node[right=.1 of X_out1] {$B_1$}
		node[right=.1 of X_out2] {$B_2$}
		node[right=.1 of X_out3] {$B_3$};
\end{tikzpicture}
\end{equation}
where $A_1$, $A_2$, $B_1$, $B_2$, and $B_3$ are sets, measurable spaces, vector spaces, or Euclidean spaces. To do so, we now introduce the notion of typed finite sets.

\subsection{Typed finite sets}

The categories $\Set$, $\Lin$, $\Euc$, and $\Meas$ are finite product categories, as discussed in Section~\ref{sec:more_background}.

\begin{definition}\label{def:TFS}

Fix a finite product category $\cat{C}$. The category of \emph{$\cat{C}$-typed finite sets}, denoted $\TFS_{\cat{C}}$, is defined as follows. An object in $\TFS_{\cat{C}}$ is a finite set of objects in $\cat{C}$, 
\[\TFS_{\cat{C}}:=\{(P,\tau)\; |\; P\in\FinSet, \tau\colon P\to\cat{C})\}.\] 
If $\PP=(P,\tau)$ is a typed finite set, we call an element $p\in P$ a \emph{port}; we sometimes write $p\in\PP$ by abuse of notation. We call the object $\tau(p)\in\cat{C}$ the \emph{type} of port $p$. If $P=\{1,2,\ldots, n\}$ for some $n\in\NN$, it is often convenient to denote $(P,\tau)$ by the sequence $\seq{\tau(1),\ldots,\tau(n)}$. There is a unique typed finite set with an empty set $P=\emptyset$ of ports, which we denote by $0\coloneqq\langle\ \rangle$. 

A morphism $\gamma\colon(P,\tau)\to (P',\tau')$ in $\TFS_{\cat{C}}$ consists of a function $\gamma\colon P\to P'$ which \emph{respects types} in the sense that for every $p\in P$ one has $\tau'\big(\gamma(p)\big)=\tau(p)$, i.e., such that the following diagram of finite sets commutes:
\[
\begin{tikzcd}[column sep=0pt]
P \ar[rr,"\gamma"] \ar[rd,"\tau"']
& {}
& P' \ar[ld,"{\tau'}"]\\
&\cat{C}
\end{tikzcd}
\]
We refer to the morphisms of $\TFS_{\cat{C}}$ as \emph{$\cat{C}$-typed functions}. We may elide the reference to $\cat{C}$ if it is clear from context.

Given two typed finite sets, $\PP_1\coloneqq(P_1,\tau_1)$ and $\PP_2\coloneqq(P_2,\tau_2)$, we can form their sum $\PP_1+\PP_2\coloneqq(P_1+P_2,\tau_1+\tau_2)$, where $P_1+P_2$ is the disjoint union of $P_1$ and $P_2$, and $\tau_1+\tau_2$ is equal to $\tau_i$ when restricted to $P_i$, for $i=1,2$. Thus we have a \emph{symmetric monoidal structure} on $\TFS$, where the monoidal unit is $0$.

\end{definition}

Example~\ref{ex:TFS} skips ahead a little to show what we are building toward.

\begin{example}\label{ex:TFS}

Suppose the five labels ($A_1, A_2, B_1, B_2, B_3$) below refer to objects in some category $\cat{C}$.
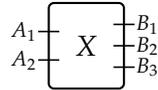
\[
\begin{tikzpicture}[oriented WD, bbx=.1cm, bby =.1cm, bb port sep=.1cm]
	\node [bb={2}{3}] (X) {$X$};
	\draw[label] 
		node[left=.1 of X_in1]  {$A_1$}
		node[left=.1 of X_in2]  {$A_2$}
		node[right=.1 of X_out1] {$B_1$}
		node[right=.1 of X_out2] {$B_2$}
		node[right=.1 of X_out3] {$B_3$};
\end{tikzpicture}
\]
The left-hand (input) side and the right-hand (output) side of box $X$ can be represented by the typed finite sets 
\begin{equation}\label{eqn:two_TFSs}
\inp{X}=\seq{A_1,A_2}\quad\tn{and}\quad\outp{X}=\seq{B_1,B_2,B_3}
\end{equation}
respectively. There are many ways to break $X$ up into the sum of smaller boxes while maintaining the $\cat{C}$-labels of each wire. For example, 
\begin{equation}\label{eqn:TFS_broken}
\begin{tikzpicture}[oriented WD, baseline=(X.center), bbx=.1cm, bby =.1cm, bb port sep=.1cm]
	\node [bb={2}{3}] (X) {$X$};
	\draw[label] 
		node[left=.1 of X_in1]  {$A_1$}
		node[left=.1 of X_in2]  {$A_2$}
		node[right=.1 of X_out1] {$B_1$}
		node[right=.1 of X_out2] {$B_2$}
		node[right=.1 of X_out3] {$B_3$};
\end{tikzpicture}
\quad = \quad
\begin{tikzpicture}[oriented WD, baseline=($(X1.north)!.5!(X2.south)$), bbx=.1cm, bby =.1cm, bb port sep=.1cm]
	\node [bb={1}{2}] (X1) {$X_1$};
	\node [below=.5 of X1] (boxplus) {$\boxplus$};
	\node [bb={1}{1}, below=.5 of boxplus] (X2) {$X_2$};
	\draw[label] 
		node[left=.1 of X1_in1]  {$A_1$}
		node[left=.1 of X2_in1]  {$A_2$}
		node[right=.1 of X1_out1] {$B_1$}
		node[right=.1 of X1_out2] {$B_2$}
		node[right=.1 of X2_out1] {$B_3$};
\end{tikzpicture}
\quad = \quad
\begin{tikzpicture}[oriented WD, baseline=($(X1.north)!.5!(X2.south)$), bbx=.1cm, bby =.1cm, bb port sep=.1cm]
	\node [bb={0}{1}] (X1) {$X_1'$};
	\node [below=.5 of X1] (boxplus) {$\boxplus$};
	\node [bb={2}{2}, below=.5 of boxplus] (X2) {$X_2'$};
	\draw[label] 
		node[left=.1 of X2_in1]  {$A_1$}
		node[left=.1 of X2_in2]  {$A_2$}
		node[right=.1 of X1_out1] {$B_1$}
		node[right=.1 of X2_out1] {$B_2$}
		node[right=.1 of X2_out2] {$B_3$};
\end{tikzpicture}
\quad=\quad \tn{etc....}
\end{equation}
This will be made precise in Definition~\ref{def:box}.

\end{example}

\subsection{Dependent product of a typed finite set}

Having multiple ports is useful for allowing different sorts of information to flow around within a wiring diagram. However, it terms of dynamical systems, having three input ports $\seq{A,B,C}$ is the same as having one input port $A\times B\times C$. The next definition simply formalizes this notion, and a similar one for morphisms of typed finite sets.

\begin{definition} \label{def:depprod}
Let $\cat{C}$ be a finite product category, and suppose that $\PP\coloneqq(P,\tau)\in\TFS_{\cat{C}}$ is a typed finite set. Its \emph{dependent product} $\dProd{\PP}\in\cat{C}$ is defined as the product in $\cat{C}$,
\[\dProd{(P,\tau)}:=\prod_{p\in P}\tau(p).\] 
Given a typed function $\gamma\colon (P,\tau)\to (P',\tau')$ in $\TFS_{\cat{C}}$ we define 
\[\dProd{\gamma}\colon \dProd{(P',\tau')}\to\dProd{(P,\tau)}\]
using the universal property of products in the evident way. It is a simultaneous generalization of projection $\pr\colon A\times B\to A$, diagonal $A\to A\times A$, and reordering $A\times B\to B\times A$. For example, suppose that $P=\{1,\ldots,p\}$ and $P'=\{1,\ldots,p'\}$ are finite ordinals. Then $\dProd{\gamma}$ is given on an element $(a_1,\ldots,a_{p'})\in\dProd{(P',\tau')}$ by the formula
\begin{equation}\label{eqn:dprod_formula}
\dProd{\gamma}(a_1,\ldots,a_{p'})\coloneqq(a_{\gamma(1)},\ldots,a_{\gamma(p)}).
\end{equation}
%the unique morphism for which the following diagram commutes for all \mbox{$p\in P$}:
%\[
%\begin{tikzcd}
%\prod_{p'\in P'}\tau'(p') \ar[r,"\dProd{\gamma}"] \ar[d,"{\pi_{\gamma(p)}}"']
%& \prod_{p\in P}\tau(p) \ar[d,"{\pi_p}"]\\
%\tau'(\gamma(p))\ar[r,equal]&\tau(p)
%\end{tikzcd} 
%\]
%By the universal property for products, this defines a functor, 
It is easy to check that dependent product defines a functor, 
\[\dProd{\;\cdot\;}\colon\TFS_{\cat{C}}\op\to\cat{C}.\]
\end{definition}

\begin{lemma}\label{lemma:dep_prod}
The dependent product functor sends coproducts in $\TFS_{\cat{C}}$ to products in $\cat{C}$. That is, we have a natural isomorphism
\[
\dProd{\PP_1}\times\dProd{\PP_2}\cong\dProd{\PP_1+\PP_2}.
\]
\end{lemma}

\begin{example}

Consider Example~\ref{ex:TFS}. The dependent products of the sets in \eqref{eqn:two_TFSs} are 
\[\vinp{X}=A_1\times A_2\quad\tn{and}\quad \voutp{X}=B_1\times B_2\times B_3\]
and similarly $\vinp{X}=\vinp{X_1}\times\vinp{X_2}$ and $\voutp{X}=\voutp{X_1}\times\voutp{X_2}$ in \eqref{eqn:TFS_broken}.

\end{example}

\section{The monoidal category $\cat{W}$ of wiring diagrams}

\begin{definition} \label{def:box}
Let $\cat{C}$ be a finite product category. A \emph{$\cat{C}$-box} $X$ (called simply a \emph{box} if $\cat{C}$ is clear from context) is an ordered pair of typed finite sets, 
\[X=(\inp{X},\outp{X})\in\TFS_{\cat{C}}\times\TFS_{\cat{C}}.\]
We refer to elements $a\in\inp{X}$ and $a'\in\outp{X}$ as \emph{input ports} and \emph{output ports}, respectively.

Given two boxes $X_1,X_2$, we define their \emph{sum} (or \emph{parallel composition}), denoted $X_1\boxplus X_2$, by
\[\inp{(X_1\boxplus X_2)}\coloneqq\inp{X_1}+\inp{X_2}\qquad\outp{(X_1\boxplus X_2)}\coloneqq\outp{X_1}+\outp{X_2}\]
We define the \emph{closed box}, denoted $\square$, to be the box with an empty set of input and output ports, 
\[\square\coloneqq(0,0).\]

If $X$ is a box, we denote by $\dProd{X}$ the pair $(\vinp{X},\voutp{X})\in\cat{C}\times\cat{C}$. Similarly, denote 
\[\dProd{X_1}\boxtimes\dProd{X_2}\coloneqq\left(\vinp{X_1}\times\vinp{X_2},\voutp{X_1}\times\voutp{X_2}\right).\] 

\end{definition}

\begin{remark}

By Lemma~\ref{lemma:dep_prod}, there is an isomorphism $\dProd{X_1\boxplus X_2}\cong\dProd{X_1}\boxtimes\dProd{X_2}$, and there is an isomorphism $\dProd{\square}\cong(1,1)$, where $1$ denotes any one-element set.

\end{remark}

The following definition is relative to a choice $\cat{C}$ of finite product category. That is, wherever we write "typed function", we mean $\cat{C}$-typed function in the sense of Definition~\ref{def:TFS}.

\begin{definition}\label{def:wiring_diagram}

Let $X=(\inp{X},\outp{X})$ and $Y=(\inp{Y},\outp{Y})$ be boxes. A \emph{wiring diagram} $\varphi\colon X\to Y$ is a pair $(\inp{f},\outp{f})$ of typed functions
\begin{align}\label{eqn:WD}
\inp{\varphi}&\colon\inp{X}\to\inp{Y}+\outp{X}\\\nonumber
\outp{\varphi}&\colon\outp{Y}\to\outp{X}
\end{align}
Define the \emph{identity wiring diagram}, denoted $\id_X\colon X\to X$, by setting $\inp{(\id_X)}$ to be the coproduct inclusion $\inp{X}\to\inp{X}+\outp{X}$, and setting $\outp{(\id_X)}$ to be the identity function, $\outp{X}\to\outp{X}$.

Given wiring diagrams $\varphi_1\colon X_1\to Y_1$ and $\varphi_2\colon X_2\to Y_2$, we define their \emph{sum}, denoted $\varphi_1\boxplus\varphi_2$, by using the cocartesian monoidal structure on $\FinSet$:
\[
\inp{(\varphi_1\boxplus\varphi_2)}\coloneqq\inp{\varphi_1}+\inp{\varphi_2}
\qquad
\outp{(\varphi_1\boxplus\varphi_2)}\coloneqq\outp{\varphi_1}+\outp{\varphi_2}
\]

\end{definition}

\begin{example}\label{ex:WD_explicit}

Consider the wiring diagram shown to the right below. It is obtained by taking the monoidal product of---i.e., putting in parallel---the inner boxes, $X=X_1\boxplus X_2$. Thus it is equivalent to the "operadic" diagram shown on the left:
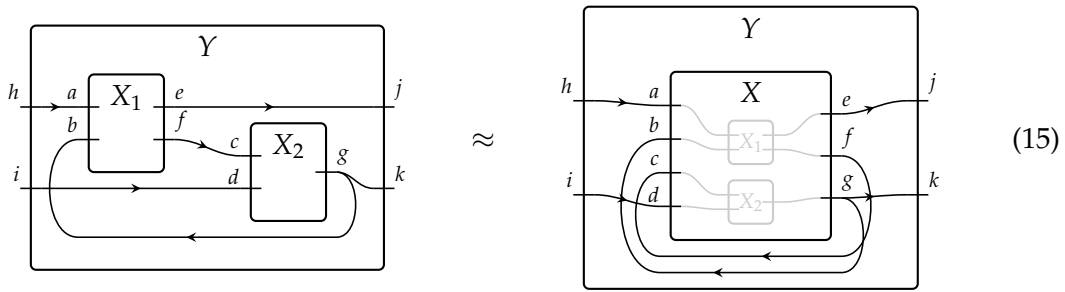
\begin{equation}\label{eq:WD_two_ways}
\begin{tikzpicture}[oriented WD,baseline=(Y.center), bbx=1em, bby=1.2ex]
 \node[bb={2}{2},bb name=$X_1$] (X1) {};
 \node[bb={2}{1},below right = -3 and 3 of X1, bb name=$X_2$] (X2) {};
 \node[bb={2}{2}, fit={($(X1.north west)+(-1,2)$) ($(X2.south)+(0,-2)$) ($(X2.east)+(1,0)$)}, bb name = $Y$] (Y) {};
 \draw[ar] (Y_in1') to (X1_in1);
 \draw[ar] (X1_out2) to (X2_in1);
 \draw[ar] (Y_in2') to (X2_in2);
 \draw[ar] (X1_out1) to (Y_out1');
 \draw (X2_out1) to (Y_out2');
 \draw[ar] let \p1=(X2.south east), \p2=(X1.south west), \n1={\y1-\bby}, \n2=\bbportlen in
 (X2_out1) to[in=0] (\x1+\n2,\n1) -- (\x2-\n2,\n1) to[out=180] (X1_in2);
 \draw [label]
 	node[above left=.4 and 0 of X1_in1] {$a$}
	node[above left=.4 and 0 of X1_in2] {$b$}
	node[above left=.4 and 0 of X2_in1] {$c$}
	node[above left=.2 and 0 of X2_in2] {$d$}
	node[above right=.4 and 0 of X1_out1] {$e$}
	node[above right=.4 and 0 of X1_out2] {$f$}
	node[above right=.4 and 0 of X2_out1] {$g$}
	node[above left=.4 and 0 of Y_in1] {$h$}
	node[above left=.4 and 0 of Y_in2] {$i$}
	node[above right=.4 and 0 of Y_out1] {$j$}
	node[above right=.4 and 0 of Y_out2] {$k$};
\end{tikzpicture}
\qquad\approx\qquad
\begin{tikzpicture}[oriented WD,baseline=(Y.center), bbx=2em, bby=1.2ex, bb port sep=1]
\begin{scope}[bbx=.25em, bb min width=.25em, bby=.5em, bb port sep=1, black!20!white]
	\node[bb={2}{2}] (X1) {$\scriptstyle X_1$};
	\node[bb={2}{1}, below =of X1] (X2) {$\scriptstyle X_2$};
\end{scope}
\node[bb={4}{3}, fit={($(X1.north)+(0,2)$) (X2)}, bb name=$X$] (X) {};
\begin{scope}[black!20!white]
	\draw (X_in1') to (X1_in1);
	\draw (X_in2') to (X1_in2);
	\draw (X_in3') to (X2_in1);
	\draw (X_in4') to (X2_in2);
	\draw (X1_out1) to (X_out1');
	\draw (X1_out2) to (X_out2');
	\draw (X2_out1) to (X_out3');
\end{scope}
\node[bb={2}{2}, fit={($(X.north east)+(.5,3)$) ($(X.south west)-(.5,2)$)}, bb name = $Y$] (Y) {};
\draw[ar] (Y_in1') to (X_in1);
\draw[ar] (Y_in2') to (X_in4);
\draw[ar] (X_out1) to (Y_out1');
\draw[ar] (X_out3) to (Y_out2');
\draw[ar] let \p1=(X.south east), \p2=(X.south west), \n1={\y1-\bby}, \n2=\bbportlen in
	(X_out2) to[in=0] (\x1+\n2,\n1) -- (\x2-\n2,\n1) to[out=180] (X_in3);
\draw[ar] let \p1=(X.south east), \p2=(X.south west), \n1={\y1-2*\bby}, \n2={\bbportlen} in
	(X_out3) to[in=0] (\x1+\n2,\n1) -- (\x2-\n2,\n1) to[out=180] (X_in2);	
\draw [label]
 	node[above left=.4 and 0 of X_in1] {$a$}
	node[above left=.4 and 0 of X_in2] {$b$}
	node[above left=.4 and 0 of X_in3] {$c$}
	node[above left=.2 and 0 of X_in4] {$d$}
	node[above right=.4 and 0 of X_out1] {$e$}
	node[above right=.4 and 0 of X_out2] {$f$}
	node[above right=.4 and 0 of X_out3] {$g$}
	node[above left=.4 and 0 of Y_in1] {$h$}
	node[above left=.4 and 0 of Y_in2] {$i$}
	node[above right=.4 and 0 of Y_out1] {$j$}
	node[above right=.4 and 0 of Y_out2] {$k$};
\end{tikzpicture}
\end{equation}
The right-hand picture shows a wiring diagram $\varphi\colon X\to Y$ in the sense of Definition~\ref{def:wiring_diagram}.%
\footnote{Inside the box labeled $X$ we have faintly drawn $X_1$ and $X_2$, because $X=X_1\boxplus X_2$; however, the morphism $\varphi\colon X\to Y$ does not refer to these inner boxes.}  
The typed functions $\inp{\varphi}\colon\inp{X}\to\outp{X}+\inp{Y}$ and $\outp{\varphi}\colon\outp{Y}\to\outp{X}$ defining $\varphi$, as in \eqref{eqn:WD}, are shown in the following table:
\begin{equation}\label{eqn:WDtables}
\begin{array}{| c | c |}
\hline
\tn{port}\in\inp{X}&\inp{\varphi}(\tn{port})\\\hline
a&h\\
b&g\\
c&f\\
d&i\\\hline
\end{array}
\qquad\qquad
\begin{array}{| c | c |}
\hline
\tn{port}\in\outp{Y}&\outp{\varphi}(\tn{port})\\\hline
j&e\\
k&g\\&\\&\\\hline
\end{array}
\end{equation}
For example, the fact that wire $g$ is shown splitting (feeding both $b$ and $k$) in the wiring diagram pictures above \eqref{eq:WD_two_ways} corresponds to the fact that $g$ appears twice, next to $b$ and $k$, in the tables \eqref{eqn:WDtables}.

\end{example}

Composition of wiring diagrams is visually straightforward. For example, the picture below shows four wiring diagrams: two "interior" wiring diagrams $\varphi_1\colon X_{11},X_{12},X_{13}\to Y_1$ and $\varphi_2\colon X_{21},X_{22}\to Y_2$, an "exterior" wiring diagram $\psi\colon Y_1, Y_2\to Z$ (shown again on the right):
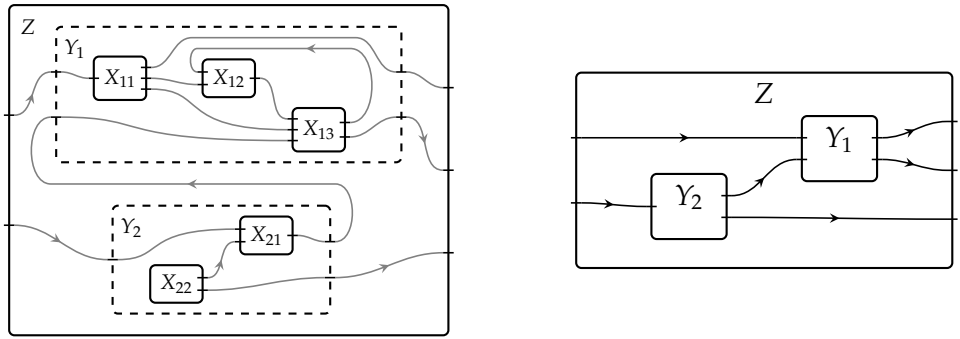
\[
\begin{tikzpicture}[oriented WD, baseline=(Z.center), bbx = .5cm, bby =.8ex, bb min width=.5cm, bb port length=2pt, bb port sep=1]
  \node[bb={1}{3}] (X11) {$\scriptstyle X_{11}$};
  \node[bb={2}{1}, right=1.5 of X11] (X12) {$\scriptstyle X_{12}$};
  \node[bb={3}{2}, below right=of X12] (X13) {$\scriptstyle X_{13}$};
  \node[bb={2}{2}, fit={(X11) (X12) (X13) ($(X12.north)+(0,2)$) ($(X13.east)+(.5,0)$)}, dashed] (Y1) {};
  \node[bb={2}{1}, below left=6 and 0 of X13] (X21) {$\scriptstyle X_{21}$};
  \node[bb={0}{2},below left=of X21] (X22) {$\scriptstyle X_{22}$};
  \node[bb={1}{2}, fit=(X21) (X22), dashed] (Y2) {};
  \node[bb={2}{3}, fit={($(Y1.north)+(0,1)$) ($(Y2.south)-(0,1)$) ($(Y1.west)-(.25,0)$) ($(Y1.east)+(.25,0)$)}] (Z) {};
  \draw[label] 
	node at ($(Y1.north west)+(.5,-2)$)  {$Y_1$}
	node at ($(Y2.north west)+(.5,-2)$)  {$Y_2$}
	node at ($(Z.north west)+(.5,-2)$)  {$Z$};
  \begin{scope}[gray]
  \draw[ar] (Z_in1') to (Y1_in1);
  \draw[ar] (Z_in2') to (Y2_in1);
  \draw (Y2_in1') to[in looseness=2] (X21_in1);
  \draw[ar] (X22_out1) to (X21_in2);
  \draw (X22_out2) to (Y2_out2');
  \draw[ar] (Y2_out2) to (Z_out3');
  \draw (Y1_in1') to (X11_in1);
  \draw (X13_out2) to (Y1_out2');
  \draw (Y1_in2') to[in looseness=2] (X13_in3);
  \draw (X11_out3) to[in looseness=2] (X13_in2);%
  \draw (X12_out1) to (X13_in1);
  \draw (X11_out2) to (X12_in2);
  \draw (Y1_out1) to (Z_out1');
  \draw[ar] (Y1_out2) to (Z_out2');
  \draw (X21_out1) to (Y2_out1');
  \draw[ar] let \p1=(Y2.north east), \p2=(Y1.south west), \n1={\y1+2*\bby}, \n2=\bbportlen in
  	(Y2_out1) to[in=0] (\x1+\n2,\n1) -- (\x2-\n2,\n1) to[out=180] (Y1_in2);
  \draw[ar] let \p1=(X13.north east), \p2=(X12.north west), \n1={\y2+\bby}, \n2=\bbportlen in
  	(X13_out1) to[in=0] (\x1+\n2,\n1) -- (\x2-\n2,\n1) to[out=180] (X12_in1);
  \draw let \p1=(X12.north west), \p2=(X13.north east), \n1={\y1+2*\bby}, \n2=\bbportlen in
  	(X11_out1) to (\x1-2*\n2,\n1) -- (\x2+2*\n2,\n1) to[out=0] (Y1_out1');
  \end{scope}
  \end{tikzpicture}
  \qquad\qquad
  \begin{tikzpicture}[oriented WD, baseline=(Z.center), bbx = 1cm, bby =1.6ex, bb min width=1cm, bb port length=2pt, bb port sep=1]
  	\node[bb={1}{2}, bb name={$Y_2$}] (Y2) {};
	\node[bb={2}{2}, above right = -1/3 and 1 of Y2, bb name={$Y_1$}] (Y1) {};
	\node[bb={2}{3}, fit={($(Y2.south west)+(0,-1/3)$) ($(Y1.north east)+(0,1)$)}, bb name={$Z$}] (Z) {};
	\draw[ar] (Z_in1') to (Y1_in1);
	\draw[ar] (Z_in2') to (Y2_in1);
	\draw[ar] (Y2_out1) to (Y1_in2);
	\draw[ar] (Y1_out1) to (Z_out1');
	\draw[ar] (Y1_out2) to (Z_out2');
	\draw[ar] (Y2_out2) to (Z_out3');
  \end{tikzpicture}
\]
From $\varphi_1,\varphi_2$, and $\psi$, we can erase the dashed boxes and derive a five-box wiring diagram $X_{11},X_{12}, X_{13}, X_{21}, X_{22}\to Z$. We call it their \emph{composition} and denote it $\omega=\psi\circ(\varphi_1,\varphi_2)$. This corresponds to the composition of $X\To{\varphi} Y\To{\psi} Z$ in a symmetric monoidal category $\cat{W}$ as described in Definition~\ref{def:comp_in_W}, where $X=X_{11}+X_{12}+X_{13}+X_{12}+X_{22}$ and $Y=Y_1+Y_2$.

\begin{definition}\label{def:comp_in_W}

Let $\cat{C}$ be a finite product category. Given wiring diagrams $\varphi\colon X\to Y$ and $\psi\colon Y\to Z$, we define their \emph{composition}, denoted $\psi\circ\varphi\colon X\to Z$, by the following compositions in $\TFS_{\cat{C}}$:
\[
\begin{tikzcd}[row sep=6ex, column sep=6em,baseline=(current bounding box.north)]
\inp{X}
	\ar[d,"\inp{\varphi}"']\ar[r,dashed,"\inp{(\psi\circ\varphi)}"]&
\inp{Z}+\outp{X}
	\\
\inp{Y}+\outp{X}
	\ar[d,"\inp{\psi}+\outp{X}"']\\
\inp{Z}+\outp{Y}+\outp{X}
	\ar[r,"\inp{Z}+\outp{\varphi}+\outp{X}"']&
\inp{Z}+\outp{X}+\outp{X}
	\ar[uu,"\inp{Z}+\nabla_{\outp{X}}"']
\end{tikzcd}
\qquad\qquad
\begin{tikzcd}[row sep=6ex,column sep=1em, baseline=(current bounding box.north)]
\outp{Z}
	\ar[rr,dashed,"\outp{(\psi\circ\varphi)}"]\ar[rd,"\outp{\psi}"']&&
\outp{X}
	\\
	&
\outp{Y}
	\ar[ru,"\outp{\varphi}"']
\end{tikzcd}
\]
It is straightforward to show that this composition formula is associative and unital. Thus we have defined \emph{the category of $\cat{C}$-boxes and wiring diagrams}, which we denote $\cat{W}_{\cat{C}}$. This category has a symmetric monoidal structure $(\square,\boxplus)$, where $\square$ is the closed box and $\boxplus$ is given by sums of boxes and wiring diagrams, as in Definition~\ref{def:TFS}.

\end{definition}

\begin{remark}
A wiring diagram $\varphi\colon X\to Y$, includes two typed functions $\inp{\varphi},\outp{\varphi}$, which have as dependent product the functions $\vinp{\varphi},\voutp{\varphi}$ (see Definition~\ref{def:depprod}) as shown below: 
\begin{align*}
\inp{\varphi}&\colon\inp{X}\to\inp{Y}+\outp{X}
	&
\outp{\varphi}&\colon\outp{Y}\to\outp{X}
	\\
\vinp{\varphi}&\colon\vinp{Y}\times\voutp{X}\to\vinp{X}
	&
\voutp{\varphi}&\colon\voutp{X}\to\voutp{Y}
\end{align*}
\end{remark}

The proof of following lemma is a straightforward rewriting of Definition~\ref{def:comp_in_W}.
\begin{lemma}\label{lemma:technical}

Suppose given wiring diagrams $\varphi\colon X\to Y$ and $\psi\colon Y\to Z$. Then the dependent products $\vinp{(\psi\circ\varphi)}\colon\vinp{Z}\times\voutp{X}\to\vinp{X}$ and $\voutp{(\psi\circ\varphi)}\colon\voutp{X}\to\voutp{Z}$ are given by the formulas
\begin{align*}
\vinp{(\psi\circ\varphi)}(z,x)&=\vinp{\varphi}\Big(\vinp{\psi}\big(z,\voutp{\varphi}(x)\big),x\Big)\\
\voutp{(\psi\circ\varphi)}(x)&=\voutp{\psi}\Big(\voutp{\varphi}(x)\Big)
\end{align*}

\end{lemma}

\chapter{Five formal interpretations of wiring diagrams}\label{sec:formal_specs}

In this final section, we give precise formulas for putting together subsystems according to an arbitrary wiring diagram, to form a larger system. These systems may be dynamical systems of various kinds (discrete, measurable, linear, continuous) or they may be matrices; we call these our five \emph{interpretations} of wiring diagrams. Two of them, namely linear and continuous systems, are taken from \cite{VSL}. Another, namely discrete systems, is loosely adapted from \cite{RupelSpivak}. Technically, each interpretation is a lax $\Set$-valued functor on $\cat{W}$, the symmetric monoidal category of wiring diagrams as defined in Section~\ref{sec:CT_formulation}. Experts may note that algebras on the operad of wiring diagrams appear related to traced monoidal categories, and indeed they are; see \cite{JSTraced} and \cite{SpivakSchultzRupel}.

We spell out how each interpretation works in several steps. In Section~\ref{sec:implementing_box}, we remind the reader what sort of thing is allowed to fill, or \emph{inhabit}, a given box shape, in each of our five interpretations. In Section~\ref{sec:ODS}, we explain what happens when boxes are put into parallel. In fact, whenever a wiring diagram includes several boxes, the default technique is that of Example~\ref{ex:WD_explicit}: First we put them in parallel to form one box, and then we use a wiring diagram that has only that one inner box (see Definition~\ref{def:wiring_diagram}). Thus we complete our description of our five interpretations in Section~\ref{sec:implenting_wiring} by saying what happens on wiring diagrams (with one inner box).%
\footnote{Note that in practice, this default "tensor then wire" technique is almost never the most efficient. For example, multiplying two matrices can be obtained by tensoring them and tracing the result, but this requires an order of magnitude more operations than the usual matrix product formula. In the extended example, Section~\ref{sec:extended_example}, we show an alternative technique. The proofs in the present section imply that all techniques will give the same final answer.}

In Section~\ref{sec:compositional_mappings} we give some compositional maps between interpretations. Most of these have been briefly discussed earlier in the paper, but we make formal theorems here. We conclude in Section~\ref{sec:extended_example} with an extended example.

\section{Inhabitants of a box}\label{sec:implementing_box}

Definitions~\ref{def:DS_box},~\ref{def:MS_box},~\ref{def:LS_box},~\ref{def:CS_box},~and~\ref{def:Mat_box} say precisely the set of inhabitants that are allowed to fill each box $X\in\cat{W}$ (e.g., \eqref{eqn:arb_box}), according to our five interpretations: discrete systems, measurable systems, linear systems, continuous systems, and matrices. In this section, we are simply gathering together Definitions from Section~\ref{sec:ODS_and_matrices}. 

\begin{definition}\label{def:DS_box}

Let $\cat{C}=\Set$, and let $X=(\inp{X},\outp{X})\in\cat{W}_{\Set}$ be a $\Set$-box. Define $\DS(X)\coloneqq\DS(\dProd{X})$ to be the set of $(\vinp{X},\voutp{X})$-discrete systems, as in Definition~\ref{def:DDS}. That is, 
\[\DS(X)\coloneqq\left\{(S,\rdt{f},\upd{f})\;\;\middle|\;\; S\in\Set,\quad \rdt{f}\in\Set\left(S,\voutp{X}\right),\quad \upd{f}\in\Set\left(\vinp{X}\times S, S\right)\right\}\]

\end{definition}

\begin{definition}\label{def:MS_box}

Let $\cat{C}=\Meas$, and let $X=(\inp{X},\outp{X})\in\cat{W}_{\Meas}$ be a $\Meas$-box. Define $\MS(X)\coloneqq\MS(\dProd{X})$ to be the set of $(\vinp{X},\voutp{X})$-measurable systems, as in Definition~\ref{def:MDS}. That is, 
\[
\DS(X)\coloneqq
\left\{(S,\mu,\rdt{f},\upd{f})\;\;\middle|\;\;
\parbox{3.1in}{$
	S\in\Meas,\quad \mu\tn{ is a measure on }S,\\ 
	\rdt{f}\in\Set\left(S,\voutp{X}\right),\quad \upd{f}\in\Set\left(\vinp{X}\times S, S\right)
$}
\right\}
\]

\end{definition}

\begin{definition}\label{def:LS_box}

Let $\cat{C}=\Lin$, and let $X=(\inp{X},\outp{X})\in\cat{W}_\Lin$ be a $\Lin$-box.
Define $\LS(X)\coloneqq\LS(\dProd{X})$ to be the set of $(\vinp{X},\voutp{X})$-linear systems, as in Definition~\ref{def:LDS}. That is, 
\[
\LS(X)\coloneqq
\left\{(S,\inp{M},\midp{M},\outp{M})\;\;\middle|\;\; 
S\in\Lin,\;
\begin{array}{ll}
\midp{M}\in\Lin(S,S)&\inp{M}\in\Lin\left(\vinp{X},S\right)\\
\outp{M}\in\Lin\left(S,\voutp{X}\right)
\end{array}
\right\}
\]

\end{definition}

\begin{definition}\label{def:CS_box}

Let $\cat{C}=\Set$, and let $X=(\inp{X},\outp{X})\in\cat{W}_{\Euc}$ be a $\Euc$-box. 
Define $\CS(X)\coloneqq\CS(\dProd{X})$ to be the set of $(\vinp{X},\voutp{X})$-continuous systems, as in Definition~\ref{def:CDS}. That is, 
\[\CS(X)\coloneqq\left\{(S,\rdt{f},\dyn{f})\;\;\middle|\;\; S\in\Euc,\quad \rdt{f}\in\Euc\left(S,\voutp{X}\right),\quad \upd{f}\in\Euc_{/S}\left(\vinp{X}\times S, TS\right)\right\}\]

\end{definition}

Recall that if $S\in\cat{C}$ is an object, then $\cat{C}_{/S}$ denotes the slice category of $\cat{C}$ over $S$. We will not need this again; it was used in Definition~\ref{def:CS_box} simply as  shorthand for the diagram \eqref{eqn:slice_Euc}. 

\begin{definition}\label{def:Mat_box}

Let $\cat{C}=\Set$, let $R$ be a complete semiring, and let $X=(\inp{X},\outp{X})\in\cat{W}_{\Set}$ be a $\Set$-box. Define $\Mat_R(X)\coloneqq\Mat(\dProd{X})$ to be the set of $(\vinp{X}\times\voutp{X})$-matrices in $R$. This can be identified with the set of functions 
\[\Mat(X)\cong\left\{M\colon\vinp{X}\times\voutp{X}\to R\right\}.\]

\end{definition}

\section{Parallelizing inhabitants}\label{sec:ODS}

In this section we explain how parallel composition works for each of our five interpretations, discrete systems, measurable systems, linear systems, continuous systems, and matrices. One may refer to Example~\ref{ex:TFS} and Definition~\ref{def:box}.

\begin{definition}\label{def:DDS_parallel}

Suppose we are given discrete systems $F_1=(S_1,\rdt{f_1},\upd{f_1})\in\DS(X_1)$ and $F_2=(S_2,\rdt{f_2},\upd{f_2})\in\DS(X_2)$. Their \emph{parallel composition}, denoted by $F_1\boxtimes F_2=(T,\rdt{g},\upd{g})\in\DS(X_1\boxplus X_2)$ is given as follows. Its state set is the product $T\coloneqq S_1\times S_2$ in $\Set$, its readout function $\rdt{g}=\rdt{(f_1\boxtimes f_2)}$ is the product 
\[
\rdt{(f_1\boxtimes f_2)}\coloneqq\rdt{f_1}\times\rdt{f_2}\colon S_1\times S_2\to B_1\times B_2,
\]
and its update function $\upd{g}=\upd{(f_1\boxtimes f_2)}$ is the product $\upd{f_1}\times\upd{f_2}$ as shown here:
\[
\begin{tikzcd}[column sep=5em]
A_1\times A_2\times S_1\times S_2\ar[d,"\cong"'] \ar[r,dashed,"\upd{(f_1\boxtimes f_2)}"]
&
S_1\times S_2
\\
A_1\times S_1\times A_2\times S_2\ar[r,"{\upd{f_1}\times\upd{f_2}}"']
&
S_1\times S_2\ar[u,equal]
\end{tikzcd}
\]

\end{definition}

\begin{remark}\label{rem:MDS_parallel}

Definition~\ref{def:DDS_parallel} also makes sense when $F_1$ and $F_2$ are assumed to be measurable systems, i.e., we can form a measurable system $F_1\boxtimes F_2$, called their \emph{parallel composition}, in the identical way. In particular, the set $S_1\times S_2$ is given the product measure $\mu_1\otimes\mu_2$ (see \cite{Fremlin_Measure_Theory}).

\end{remark}

\begin{definition}\label{def:LDS_parallel}

Suppose we are given linear systems $M_1=(S_1,\inp{M_1},\midp{M_1},\outp{M_1})$ and $M_2=(S_2,\inp{M_2},\midp{M_2},\outp{M_2})$. Their \emph{parallel composition}, denoted $M_1\oplus M_2$ is simply given by direct sums of the respective linear maps:
\[
\inp{(M_1\oplus M_2)}\coloneqq\inp{M_1}\oplus\inp{M_2}\qquad
\midp{(M_1\oplus M_2)}\coloneqq\midp{M_1}\oplus\midp{M_2}\qquad
\outp{(M_1\oplus M_2)}\coloneqq\outp{M_1}\oplus\outp{M_2}
\]

\end{definition}

\begin{definition}\label{def:CDS_parallel}

Suppose we are given continuous systems $F_1=(S_1,\rdt{f_1},\dyn{f_1})\in\CS(X_1)$ and $F_2=(S_2,\rdt{f_2},\dyn{f_2})\in\CS(X_2)$. Their \emph{parallel composition}, denoted by $F_1\boxtimes F_2=(T,\rdt{g},\dyn{g})\in\CS(X_1\boxplus X_2)$ is given as follows. Its state set is the product $T\coloneqq S_1\times S_2$ in $\Euc$, its readout function $\rdt{g}=\rdt{(f_1\boxtimes f_2)}$ is the product 
\[
\rdt{(f_1\boxtimes f_2)}\coloneqq\rdt{f_1}\times\rdt{f_2}\colon S_1\times S_2\to B_1\times B_2,
\]
and its update function $\upd{g}=\upd{(f_1\boxtimes f_2)}$ is, up to isomorphism, the product $\upd{f_1}\times\upd{f_2}$ as shown here:
\[
\begin{tikzcd}[column sep=5em]
A_1\times A_2\times S_1\times S_2\ar[d,"\cong"'] \ar[r,dashed,"\dyn{(f_1\boxtimes f_2)}"]
&
T(S_1\times S_2)
\\
A_1\times S_1\times A_2\times S_2\ar[r,"{\dyn{f_1}\times\dyn{f_2}}"']
&
TS_1\times TS_2\ar[u,"\cong"']
\end{tikzcd}
\]

\end{definition}

\begin{definition}\label{def:Mat_parallel}

Let $R$ be a semiring. Suppose we are given $R$-matrices $M^1\in\Mat_R(X_1)$ and $M^2\in\Mat_R(X_2)$. Their \emph{parallel composition}, denoted by $M^1\otimes M^2\in\Mat_R(X_1\boxplus X_2)$ is given as the Kronecker product, given by component-wise product (in $R$):
\begin{equation}\label{eqn:matrix_tensor}
(M^1\otimes M^2)_{(i_1,i_2),(j_1,j_2)}\coloneqq M^1_{i_1,j_1}\cdot M^2_{i_2,j_2}
\end{equation}

\end{definition}

\section{Wiring together inhabitants}\label{sec:implenting_wiring}

Any complex wiring diagram $\varphi\colon X_1,\ldots,X_n\to Y$, such as the one shown in \eqref{fig:complex_WD}, can be constructed by first putting the input boxes in parallel $X=X_1\boxplus\cdots\boxplus X_n$ as in Definition~\ref{def:box}, and then using a wiring diagram $X\to Y$ with a single inner box (see Example~\ref{ex:WD_explicit}). For each of our five interpretations (dynamical systems and matrices), the formula for putting together inhabitants of $X_1,\ldots,X_n$ to form an inhabitant of $Y$ is likewise done in these two steps. Parallelizing inhabitants was discussed in Section~\ref{sec:ODS} and how a single inhabitant, wired into a larger box, produces an inhabitant of that larger box, is described in this section.

We not only give the formula, we also prove Theorems~\ref{thm:DS_sym_mon_func},~\ref{thm:MS_sym_mon_func},~\ref{thm:LS_sym_mon_func},~\ref{thm:CS_sym_mon_func},~and~\ref{thm:Mat_sym_mon_func}, which say that these formulas are coherent for each of our five interpretations. More formally, we prove they constitute lax monoidal functors.

\subsection{Discrete systems}

\begin{definition}\label{def:DS_wiring}

Let $\varphi\colon X\to Y$ be a wiring diagram in $\cat{W}_{\Set}$, and suppose that $F=(S,\rdt{f},\upd{f})\in\DS(X)$ is an $\dProd{X}$-discrete system. We define \emph{the $\DS$-application of $\varphi$ to $F$}, denoted $\DS(\varphi)(F)\in\DS(Y)$, to be the $\dProd{Y}$-discrete system $\DS(\varphi)(F)=(T,\rdt{g},\upd{g})$ where
\begin{equation}\label{eqn:DS_formulas}
T\coloneqq S,\quad \rdt{g}(s)\coloneqq \voutp{\varphi}\left(\rdt{f}(s)\right),\quad\upd{g}(y,s)\coloneqq \upd{f}\left(\vinp{\varphi}\left(y,\rdt{f}(s)\right),s\right)
\end{equation}

\end{definition}

\begin{theorem}\label{thm:DS_sym_mon_func}

The assignments $X\mapsto\DS(X)$ and $\varphi\mapsto\DS(\varphi)$, with parallel composition as in Definition~\ref{def:DDS_parallel} constitute a symmetric monoidal functor $\DS\colon\cat{W}_\Set\to\Set$.

\end{theorem}

\begin{proof}

We need to check that for any $X\To{\varphi}Y\To{\psi}Z$ and discrete system $F=(S,\rdt{f},\upd{f})\in\DS(X)$, the following equation holds:
\[\DS(\psi)\big(\DS(\varphi)(F)\big)=\DS(\psi\circ\varphi)(F).\]
For ease of notation, let $G=(T,\rdt{g},\upd{g})\coloneqq\DS(\varphi)(F)$, let $H_1=(U_1,\rdt{h_1},\upd{h_1})\coloneqq\DS(\psi)(G)$, and let $H_2=(U_2,\rdt{h_2},\upd{h_2})\coloneqq\DS(\psi\circ\varphi)(F)$. We want to show that $H_1=H_2$. 

It is easy to see that they have the same state set, $U_1=U_2=S$, and the same readout function $\rdt{h_1}=\rdt{h_2}=\voutp{\psi}\circ\voutp{\varphi}\circ\rdt{f}$. We compute the update functions and see they are the same for any $z\in\vinp{Z}$ and $s\in S$:
\begin{align*}
	\upd{h_1}(z,s)&=
	\upd{g}\Big(\vinp{\psi}(z,\rdt{g}(s)\big),s\Big)\\&=
	\upd{f}\bigg(\vinp{\varphi}\Big(\vinp{\psi}\big(z,\rdt{g}(s)\big),\rdt{f}(s)\Big),s\bigg)\\&=
	\upd{f}\Bigg(\vinp{\varphi}\bigg(\vinp{\psi}\Big(z,\voutp{\varphi}\big(\rdt{f}(s)\big)\Big),\rdt{f}(s)\bigg),s\Bigg)\\&=
	\upd{f}\Big(\vinp{(\psi\circ\varphi)}\big(z,\rdt{f}(s)\big),s\Big)\\&=
	\upd{h_2}(z,s)
\end{align*}
where the penultimate equality is an application of Lemma~\ref{lemma:technical}, and the rest are merely untangling \eqref{eqn:DS_formulas}.

We also need to check that $\DS$ is symmetric monoidal. This is straightforward; it follows from the fact that taking dependent products is itself symmetric monoidal, sending coproducts to products, as in Lemma~\ref{lemma:dep_prod}.
\end{proof}

\subsection{Measurable systems}

\begin{definition}\label{def:MS_wiring}

Let $\varphi\colon X\to Y$ be a wiring diagram in $\cat{W}_{\Meas}$, and suppose that $F=(S,\mu,\rdt{f},\upd{f})\in\MS(X)$ is an $\dProd{X}$-measurable system. We define \emph{the $\MS$-application of $\varphi$ to $F$}, denoted $\MS(\varphi)(F)\in\DS(Y)$, to be the $\dProd{Y}$-measurable system $\MS(\varphi)(F)=(T,\nu,\rdt{g},\upd{g})$ where
\begin{equation}\label{eqn:MS_formulas}
(T,\nu)\coloneqq (S,\mu),\quad \rdt{g}(s)\coloneqq \voutp{\varphi}\left(\rdt{f}(s)\right),\quad\upd{g}(y,s)\coloneqq \upd{f}\left(\vinp{\varphi}\left(y,\rdt{f}(s)\right),s\right)
\end{equation}

\end{definition}

\begin{theorem}\label{thm:MS_sym_mon_func}

The assignments $X\mapsto\MS(X)$ and $\varphi\mapsto\MS(\varphi)$, with parallel composition as in Remark~\ref{rem:MDS_parallel} constitute a symmetric monoidal functor $\MS\colon\cat{W}_{\Meas}\to\Set$.

\end{theorem}

\begin{proof}

The underlying set functor $U\colon\Meas\to\Set$ is faithful; that is, for any measurable functions $f,g\colon X\to Y$, if they agree on underlying sets, $U(f)=U(g)$ then they are equal $f=g$. Suppose $F=(S,\rdt{f},\upd{f})\in\MS(X)$. Checking that the equation
\[\MS(\psi)\big(\MS(\varphi)(F)\big)=\MS(\psi\circ\varphi)(F)\]
holds is a matter of checking that both sides have the same state space (they do: both are $S$) and the same readout and update functions. Thus the functoriality follows from Theorem~\ref{thm:DS_sym_mon_func} by the faithfulness of $U$.

To see that $\MS$ is monoidal, notice that 
\[\MS(X)=\bigsqcup_{(S,\rdt{f},\upd{f})\in\DS(X)}\{\mu\mid\mu\tn{ is a measure on }S\}.\]
Consider the functor $\Meas\to\Set$ given by assigning the set of measures to a measurable space. It is monoidal, using the product measure construction \cite{Fremlin_Measure_Theory}, and the result follows.
\end{proof}

\subsection{Linear systems}

To define how wiring diagrams act on linear systems, we must first define the derivative of a wiring diagram.

\begin{definition}\label{def:derivative_wd}

Let $\varphi\colon X\to Y$ be a wiring diagram in $\cat{W}_{\Lin}$. We will define its \emph{derivative} to be three linear functions $\inp{\Phi}\in\Lin(\vinp{Y},\vinp{X})$,$\midp{\Phi}\in\Lin(\voutp{X},\vinp{X})$, and $\outp{\Phi}\in\Lin(\voutp{X},\voutp{Y})$ by taking derivatives, as follows. 

Recall that the dependent product $\dProd{\varphi}$ consists of two parts $\vinp{\varphi}\colon\voutp{X}\times\vinp{Y}\to\vinp{X}$ and $\voutp{\varphi}\colon\voutp{X}\to\voutp{Y}$. We define the three matrices as the following derivatives:
\begin{equation}\label{eqn:derivatives_wd}
\inp{\Phi}\coloneqq \partial_{\inp{Y}}\vinp{\varphi},\qquad
\midp{\Phi}\coloneqq \partial_{\outp{X}}\vinp{\varphi},\qquad
\outp{\Phi}\coloneqq \partial_{\outp{X}}\voutp{\varphi}.
\end{equation}
Note that because each of $\inp{\Phi}$, $\midp{\Phi}$, and $\outp{\Phi}$ is in fact a matrix of 1's and 0's because $\vinp{\varphi}$ and $\voutp{\varphi}$ are dependent products (see Definition~\ref{def:depprod}), meaning that they are made up simply of projections and diagonal maps.

\end{definition}

\begin{lemma}\label{lemma:derivative_comp_wd}
Suppose that $\varphi\colon X\to Y$ and $\psi\colon Y\to Z$ are wiring diagrams with derivatives $\Phi=(\inp{\Phi},\midp{\Phi},\outp{\Phi})$ and $\Psi=(\inp{\Psi},\midp{\Psi},\outp{\Psi})$ as in Definition~\ref{def:derivative_wd}. The the composite wiring diagram $\omega=\psi\circ\varphi\colon X\to Z$, as in Definition~\ref{def:comp_in_W} has derivatives 
\[
\inp{\Omega}=\inp{\Phi}\inp{\Psi},\qquad
\midp{\Omega}=\midp{\Phi}+\inp{\Phi}\midp{\Psi}\outp{\Phi},\qquad
\outp{\Omega}=\outp{\Psi}\outp{\Phi}.
\]

\end{lemma}

\begin{proof}

This is a chain rule computation, taking derivatives of the formulas in Lemma~\ref{lemma:technical}.\end{proof}

\begin{definition}\label{def:LS_wiring}

Let $\varphi\colon X\to Y$ be a wiring diagram in $\cat{W}_{\Lin}$ and let $(\inp{\Phi}, \midp{\Phi}, \outp{\Phi})$ be its derivatives as in Definition~\ref{def:derivative_wd}. Suppose that $M=(S,\inp{M},\midp{M},\outp{M})\in\LS(X)$ is an $\dProd{X}$-linear system. We define \emph{the $\LS$-application of $\varphi$ to $M$}, denoted $\LS(\varphi)(M)\in\LS(Y)$, to be the $\dProd{Y}$-linear system $\LS(\varphi)(M)=(T,\inp{N},\midp{N},\outp{N})$, where $T\coloneqq S$ and
\begin{equation}\label{eqn:LS_formulas}
\inp{N}\coloneqq \inp{M}\inp{\Phi},\qquad
\midp{N}\coloneqq \midp{M}+\inp{M}\midp{\Phi}\outp{M},\qquad
\outp{N}\coloneqq \outp{\Phi}\outp{M}.
\end{equation}

\end{definition}

\begin{theorem}\label{thm:LS_sym_mon_func}

The assignments $X\mapsto\LS(X)$ and $\varphi\mapsto\LS(\varphi)$, with parallel composition as in Definition~\ref{def:LDS_parallel} constutute a symmetric monoidal functor $\LS\colon\cat{W}_{\Lin}\to\Set$.

\end{theorem}

\begin{proof}

For any pair $X\To{\varphi}Y\To{\psi}Z$ of composable wiring diagrams and linear system $M\in\LS(X)$, one must check that $\LS(\psi\circ\varphi)(M)=\LS(\psi)\big(\LS(\varphi)(M)\big)$, but this follows easily from Lemma~\ref{lemma:derivative_comp_wd} and Definition~\ref{def:LS_wiring}. It is lax monoidal because direct sum appropriately commutes with the formulas in \eqref{eqn:LS_formulas}.\end{proof}

\subsection{Continuous systems}

\begin{definition}\label{def:CS_wiring}

Let $\varphi\colon X\to Y$ be a wiring diagram in $\cat{W}_{\Euc}$, and suppose that $F=(S,\rdt{f},\dyn{f})\in\CS(X)$ is an $\dProd{X}$-continuous system. We define \emph{the $\CS$-application of $\varphi$ to $F$}, denoted $\CS(\varphi)(F)\in\CS(Y)$, to be the $\dProd{Y}$-continuous system $\CS(\varphi)(F)=(T,\rdt{g},\dyn{g})$ where
\begin{equation}\label{eqn:CS_formulas}
T\coloneqq S,\quad \rdt{g}(s)\coloneqq \voutp{\varphi}\left(\rdt{f}(s)\right),\quad\dyn{g}(y,s)\coloneqq \dyn{f}\left(\vinp{\varphi}\left(y,\rdt{f}(s)\right),s\right)
\end{equation}

\end{definition}

\begin{theorem}\label{thm:CS_sym_mon_func}

The assignments $X\mapsto\CS(X)$ and $\varphi\mapsto\CS(\varphi)$, with parallel composition as in Definition~\ref{def:CDS_parallel} constitute a symmetric monoidal functor $\CS\colon\cat{W}_\Euc\to\Set$.

\end{theorem}

\begin{proof}

Although this Theroem takes place in a different context than that of Theorem~\ref{thm:DS_sym_mon_func}, namely that of continuous rather than discrete dynamical systems, the formulas \eqref{eqn:DS_formulas} and \eqref{eqn:CS_formulas} are identical, and one can check that a virtually identical proof suffices here.\end{proof}

\subsection{Matrices}

\begin{definition}\label{def:Mat_wiring}

Let $\varphi\colon X\to Y$ be a wiring diagram in $\cat{W}_{\Set}$, and suppose that $M\in\Mat(X)$ is a $(\vinp{X}\times\voutp{X})$-matrix. We define \emph{the $\Mat$-application of $\varphi$ to $M$}, denoted $N=\Mat(\varphi)(M)\in\Mat(Y)$, to be the $(\vinp{Y}\times\voutp{Y})$-matrix with $(i,j)$-entry
\begin{equation}\label{eqn:Mat_formulas}
N_{i,j}\coloneqq \sum_{k\in(\voutp{\varphi})^{-1}(j)}M_{\vinp{\varphi}(i,k),k}
\end{equation}
for any $i\in\vinp{Y}$ and $j\in\voutp{Y}$.

\end{definition}

\begin{example}\label{ex:serial_works}

We want to show that the formula in Definition~\ref{def:Mat_wiring} reduces to the usual matrix multiplication formula in the case of serial composition. We begin by converting our serial composition diagram into a single-inner-box wiring diagram by parallelizing, as discussed in the beginning of Section~\ref{sec:implenting_wiring}.
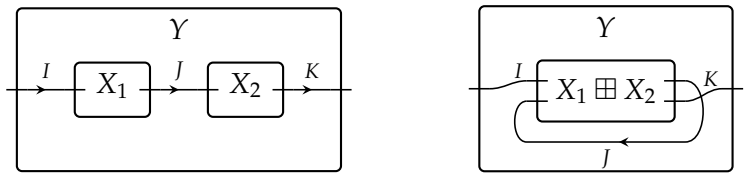
\[
\begin{tikzpicture}[oriented WD, bbx=1em, bby=1ex]
 \node[bb={1}{1},bb name=$X_1$] (X1) {};
 \node[bb={1}{1},right =2 of X1, bb name=$X_2$] (X2) {};
 \node[bb={1}{1}, fit={($(X1.north west)+(-1,3)$) ($(X1.south)+(0,-3)$) ($(X2.east)+(1,0)$)}, bb name = $Y$] (Y) {};
 \draw[ar] (Y_in1') to (X1_in1);
 \draw[ar] (X1_out1) to (X2_in1);
 \draw[ar] (X2_out1) to (Y_out1');
 \draw[label] 
	node at ($(Y_in1')!.5!(X1_in1)+(0,7pt)$)  {$I$}
	node at ($(X1_out1)!.5!(X2_in1)+(0,7pt)$)   {$J$}
	node at ($(X2_out1)!.5!(Y_out1')+(0,7pt)$)  {$K$};
\end{tikzpicture}
\qquad\qquad
\begin{tikzpicture}[oriented WD, bbx=1em, bby=1ex, bb port sep=1.5]
	 \node[bb={2}{2},] (X) {$\;X_1\boxplus X_2\;$};
	 \node[bb={1}{1}, fit={($(X.north west)+(-1,3)$) ($(X.south east)+(1,-4/1.5)$)}, bb name = $Y$] (Y) {};
	\draw (Y_in1') to (X_in1);
	\draw (X_out2) to (Y_out1');
	\draw[ar] let \p1=(X.south east), \p2=(X.south west), \n1={\y1-1.5\bby}, \n2=\bbportlen in
 		(X_out1) to[in=0]  (\x1+\n2,\n1) -- (\x2-\n2,\n1) to[out=180] (X_in2);
	\draw[label] 
		node[above left=.15 and .1 of X_in1]  {$I$} 
		node[below = 2 of X.south] {$J$}
		node[above left=.15 and .1 of Y_out1'] {$K$};
\end{tikzpicture}
\]
Thus we let $X=X_1\boxplus X_2$, let $M^1\in\Mat(X_1)$ and $M^2\in\Mat(X_2)$, and define $M=M^1\otimes M^2\in\Mat(X)$ to be the Kronecker product; see Definition~\ref{def:Mat_parallel}. Note that $\vinp{Y}=I$, $\voutp{Y}=K$, $\vinp{X}=I\times J$, and $\voutp{X}=J\times K$. Then the wiring diagram $\varphi\colon X\to Y$ above acts as follows (see \eqref{eqn:dprod_formula}) on entries:
\begin{compactenum}
	\item for $i\in\vinp{Y}$ and $(j_1,j_2)\in\voutp{X}$, we have $\vinp{\varphi}(i,j_1,j_2)=(i,j_1)$,
	\item for $(j_1,j_2)\in\voutp{X}$, we have $\voutp{\varphi}(j_1,j_2)=j_2.$
\end{compactenum}
Define $N=\Mat(\varphi)(M)$. To show that $N=M^1M^2$, we compute its entries using Equations~\eqref{eqn:matrix_tensor}~and~\eqref{eqn:Mat_formulas}:
\begin{align*}
	N_{i,j}&=
	\sum_{\parbox{.7in}{\centering $\scriptstyle k\in(\voutp{\varphi})^{-1}(j)$}}
		M_{\vinp{\varphi}(i,k),k}
	\\&=
	\sum_{\parbox{.7in}{\centering $\scriptstyle \{(j_1,j_2)\mid j_2=j\}$}}
		(M^1\otimes M^2)_{(i,j_1),(j_1,j_2)}
	\\&=
	\sum_{\parbox{.7in}{\centering $\scriptstyle j_1$}}
		M^1_{i,j_1}\cdot M^2_{j_1,j}
	\\&=
	\parbox{.7in}{\centering $M^1M^2$}
\end{align*}
Thus we have shown that, armed with the Kronecker product formula for parallel composition (Definition~\ref{def:Mat_parallel}) and the formula for arbitrary wiring diagrams (Definition~\ref{def:Mat_wiring}), we reproduce the the matrix multiplication formula for serial wiring diagrams, as in Example~\ref{ex:serial_Mat}. 

\end{example}

\begin{example}\label{ex:trace_works}

We want to show that the formula in Definition~\ref{def:Mat_wiring} reduces to the usual partial trace formula in the case of feedback composition. Consider the following wiring diagram $\varphi\colon X\to Y$:
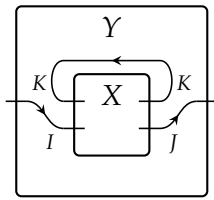
\[
\begin{tikzpicture}[oriented WD, bbx=1em, bby=1ex]
	\node[bb={2}{2}, bb name=$X$] (dom) {};
	\node[bb={1}{1}, fit={(dom) ($(dom.north east)+(1,4)$) ($(dom.south west)-(1,2)$)}, bb name = $Y$] (cod) {};
	\draw[ar,pos=20] (cod_in1') to (dom_in2);
	\draw[ar,pos=2] (dom_out2) to (cod_out1');
	\draw[ar] let \p1=(dom.north east), \p2=(dom.north west), \n1={\y2+\bby}, \n2=\bbportlen in (dom_out1) to[in=0] (\x1+\n2,\n1) -- (\x2-\n2,\n1) to[out=180] (dom_in1);
	\draw[label] 
		node[below left=2pt and 3pt of dom_in2]{$I$}
		node[below right=2pt and 3pt of dom_out2]{$J$}
		node[above left=4pt and 6pt of dom_in1] {$K$}
		node[above right=4pt and 6pt of dom_out1] {$K$};
\end{tikzpicture}
\]
Analogously to Example~\ref{ex:serial_works}, we find that $\vinp{\varphi}(k,j,i)=(k,i)$ and $\voutp{\varphi}(k,j)=j$. Define $N=\Mat(\varphi)(M)$. To show that $N=\Tr^K_{I,J}(M)$ is the partial trace, as defined in \eqref{eqn:def_trace}, we compute its entries using Equation~\eqref{eqn:Mat_formulas}:
\begin{align*}
	N_{i,j}&=
	\sum_{\parbox{.9in}{\centering $\scriptstyle (k,j)\in(\voutp{\varphi})^{-1}(j)$}}
		M_{\vinp{\varphi}(k,j,i),(k,j)}
	\\&=
	\sum_{\parbox{.9in}{\centering $\scriptstyle k\in K$}}
		M_{(k,i),(k,j)}
	=
	\Tr^K_{I,J}(M).
\end{align*}

\end{example}

One can repeat Examples~\ref{ex:serial_works}~and~\ref{ex:trace_works} for splitting wires as in Example~\ref{ex:splitting_mat}; we leave this to the reader. We now prove (in Theorem~\ref{thm:Mat_sym_mon_func}) that one can make arbitrarily complex wiring diagrams and the matrix formula given in \eqref{eqn:Mat_formulas} is consistent with regard to nesting. This theorem holds for matrices over any semiring $R$, so we elide the subscript.

\begin{theorem}\label{thm:Mat_sym_mon_func}

The assignments $X\mapsto\Mat(X)$ and $\varphi\mapsto\Mat(\varphi)$, with parallel composition as in Definition~\ref{def:Mat_parallel} constitute a symmetric monoidal functor $\Mat\colon\cat{W}_\Set\to\Set$.

\end{theorem}

\begin{proof}

We need to check that for any $X\To{\varphi}Y\To{\psi}Z$ and matrix $M\in\Mat(X)$, the following equation holds:
\[\Mat(\psi)\big(\Mat(\varphi)(M)\big)=\Mat(\psi\circ\varphi)(M).\]
We again simply compute the $(i,j)$-entries, using Equation~\eqref{eqn:Mat_formulas}, and show they agree:
\begin{align*}
	\Mat(\psi)\big(\Mat(\varphi)(M)\big)_{i,j}
	&=
	\sum_{\parbox{.9in}{\centering $\scriptstyle \ell\in(\voutp{\psi})^{-1}(j)$}}
		\Mat(\varphi)(M)_{\vinp{\psi}(i,\ell),\ell}
	\\&=
	\sum_{\parbox{.9in}{\centering $\scriptstyle \ell\in(\voutp{\psi})^{-1}(j)$}}
		\left(\;\sum_{k\in(\voutp{\varphi})^{-1}(\ell)}
			M_{\vinp{\varphi}\left(\vinp{\psi}(i,\ell),k\right),k}\right)
	\\&=
	\sum_{\parbox{.9in}{\centering $\scriptstyle k\in\left(\voutp{(\psi\circ\varphi)}\right)^{-1}(j)$}}
		M_{\vinp{\varphi}\left(\vinp{\psi}\left(i,\voutp{\varphi}(k)\right),k\right),k}
	\\&=
	\sum_{\parbox{.9in}{\centering $\scriptstyle k\in\left(\voutp{(\psi\circ\varphi)}\right)^{-1}(j)$}}
		M_{\vinp{(\psi\circ\varphi)}(k),k}
	\\&=
	\Mat(\psi\circ\varphi)(M)_{i,j}
\end{align*}
where the penultimate equation follows from Lemma~\ref{lemma:technical}.

Checking that it is monoidal involves a similar computation. Let $\varphi_1\colon X_1\to X_1'$ and $\varphi_2\colon X_2\to X_2'$, let $M^1\in\Mat(X_1)$ and $M^2\in\Mat(X_2)$. Then for $(i_1,i_2)\in\vinp{X_1}\times\vinp{X_2}$ and $(j_1,j_2)\in\voutp{X_1}\times\voutp{X_2}$, we have
\begin{align*}
	\Mat(\varphi_1\boxplus\varphi_2)(M^1\otimes M^2)_{(i_1,i_2),(j_1,j_2)}
	&=
	\sum_{\parbox{1.2in}{\centering $\scriptstyle k\in\left(\voutp{(\varphi_1\boxplus\varphi_2)}\right)^{-1}(j_1,j_2)$}}
		(M^1\otimes M^2)_{\voutp{(\varphi_1\boxplus\varphi_2)}((i_1,i_2),k),k}
	\\&=
	\sum_{\substack{
		\parbox{1.2in}{\centering $\scriptstyle 
			k_1\in\left(\voutp{\varphi_1}\right)^{-1}(j_1)
		$}
		\\
		\parbox{1.2in}{\centering $\scriptstyle 
			k_2\in\left(\voutp{\varphi_2}\right)^{-1}(j_2)
		$}
	}}
		M^1_{\voutp{\varphi_1}(i_1,k_1),k_1}\cdot M^2_{\voutp{\varphi_2}(i_2,k_2),k_2}
	\\&=
	\Mat(\varphi_1)(M^1)\otimes\Mat(\varphi_2)(M^2)
\end{align*} 
\end{proof}

\section{Compositional mappings between open systems and matrices}\label{sec:compositional_mappings}

In this section we define a few maps between various interpretations of the wiring diagram syntax. Each of these will be compositional, meaning that one can interconnect a system of systems and then apply the map, or apply the maps and then interconnect, and the result will be the same. 

First we show that Euler's method of approximating an ordinary differential equation by $\epsilon$-steps is compositional, whether one targets discrete systems or measurable systems. Second we show that the steady state matrix---starting from discrete systems, measurable systems, or continuous systems---is also compositional. At this point, we will have achieved our goal of compositionally classifying any of these sorts of open dynamical systems using steady state matrices. Third we show that while the entries of ordinary matrices are used to count the number of steady states, we can actually "remember them" by having the entries be the sets themselves. Finally, we discuss linear stability for continuous dynamical systems and show that it too is compositional, thus recovering a fully compositional generalization of bifurcation diagrams.

\subsection{Euler's $\epsilon$-approximation is compositional}

Note that any Euclidean space $S$ has an underlying vector space (which we denote the same way). For any point $s\in S$ there is a canonical linear isomorphism $TS_s\To{\cong}S$. Thus for any real number $\epsilon$ and element $v\in TS_s$, the formula $s+\epsilon\cdot v$ makes sense, where $\cdot$ represents scalar multiplication. The following definition formalizes Construction~\ref{const:Euler}. 

\begin{definition}\label{def:epsilon_approx}

Let $X\in\cat{W}$ be a box, and let $F=(S,\rdt{f},\dyn{f})\in\CS(X)$ be a continuous dynamical system. Its $\epsilon$-approximation is the discrete dynamical system $\Appx(F)=(S,\rdt{f},\upd{f}_\epsilon)\in\DS(X)$, with the same state set and readout function, but where for any $x\in\vinp{X}$ and $s\in S$, we define
\[\upd{f}_\epsilon(x,s)\coloneqq s+\epsilon\cdot\upd{f}(x,s).\]

\end{definition}

Above we have elided forgetful functors, namely the underlying vector space and underlying set functors, $\Euc\to\Vect$ and $\Euc\to\Set$. The following theorem says that if one discretizes each continuous dynamical system in a coupled network and puts them together, the result will be the same as putting the continuous systems together and then discretizing.

\begin{theorem}\label{thm:approximation_CSDS}

For any $\epsilon>0$, the $\epsilon$-approximation function $\Appx\colon\CS\to\DS$ is compositional, i.e., a monoidal natural transformation of $\cat{W}$-algebras.

\end{theorem}

\begin{proof}

We want to show that $\epsilon$-approximation is a monoidal natural transformation,
\[
\begin{tikzcd}[column sep=small, row sep=large]
	\cat{W}_{\Euc}\ar[rr,"\cat{W}_U"]\ar[dr,"\CS"']&\ar[d,phantom,near start, "{\overset{\Appx}{\Rightarrow}}"]&\cat{W}_{\Set}\ar[dl,"\DS"]\\
	&\Set
\end{tikzcd}
\]
where $\cat{W}_U$ is the forgetful functor that comes from the product-preserving functor $U\colon\Euc\to\Set$ sending a Euclidean space to its underlying set of points. In the discussion below, we drop subscripts for ease of exposition.

First we must check that for every wiring diagram $\varphi\colon X\to Y$ in $\cat{W}$, the diagram below commutes:
\[
\begin{tikzcd}
	\CS(X)\ar[r,"\CS(\varphi)"]\ar[d,"\Appx"']&\CS(Y)\ar[d,"\Appx"]\\
	\DS(X)\ar[r,"\DS(\varphi)"']&\DS(Y)
\end{tikzcd}
\]
which establishes that $\epsilon$-approximation is a natural transformation. This is a matter of combining Definitions~\ref{def:DS_wiring}~and~\ref{def:CS_wiring} with Definition~\ref{def:epsilon_approx}: for any $F=(S,\rdt{f},\upd{f})\in\CS(X)$, both sides give
\[\DS(\varphi)\big(\Appx(F)\big)=s+\epsilon\cdot\upd{f}\Big(\vinp{\varphi}\big(y,\rdt{f}(s)\big),s\Big)=\Appx\big(\CS(\varphi)(F)\big).\]

Second we check that $\Appx$ is monoidal, i.e., that for any boxes $X_1,X_2\in\cat{W}$, the diagram below commutes:
\[
\begin{tikzcd}
	\CS(X_1)\boxtimes\CS(X_2)\ar[r,"\boxtimes"]\ar[d,"\Appx\times\Appx"']&\CS(X_1\boxplus X_2)\ar[d,"\Appx"]\\
	\DS(X_1)\boxtimes\DS(X_2)\ar[r,"\boxtimes"']&\DS(X_1\boxplus X_2)
\end{tikzcd}
\]
By Definitions~\ref{def:DDS_parallel}~and~\ref{def:CDS_parallel}, this comes down to checking that for $f_1\in\CS(X_1)$ and $f_2\in\CS(X_2)$, we have
\[\left(\upd{f_1}\right)_\epsilon\times\left(\upd{f_2}\right)_\epsilon=\left(\upd{f_1}\times\upd{f_2}\right)_\epsilon.\]
This in turn follows from the fact that $\epsilon$-approximation (Definition~\ref{def:epsilon_approx}) preserves products, i.e., for $(a_1,s_1)\in\CS(X_i)$ we have
\[\left(s_1+\epsilon\cdot\upd{f_1}(a_1,s_1),s_2+\epsilon\cdot\upd{f_2}(a_2,s_2)\right)=(s,t)+\epsilon\cdot\left(\upd{f_1}(a_1,s_1),\upd{f_2}(a_2,s_2)\right)\]
completing the proof.
\end{proof}

\begin{remark}

We could also consider the $\epsilon$-approximation function as a map $\Appx\colon\CS\to\MS$. First we need a monoidal functor $\cat{W}_U\colon\cat{W}_{\Euc}\to\cat{W}_{\MS}$; this is given by the product-preserving functor $U\colon\Euc\to\MS$ sending a Euclidean space to its underlying countably-separated measurable space of Borel sets (see Proposition~\ref{prop:csms}). The only other difference with Definition~\ref{thm:approximation_CSDS} is that we must specify a measure on the underlying measurable space $U(S)$. We use the canonical measure, given by integrating the volume form, that exists on any Euclidean space, or more generally, on any oriented manifold. 

\end{remark}

\subsection{Steady state matrices}

In Definition~\ref{def:steady_state_DDS} we introduced the notion of steady states for discrete dynamical systems. In Definition~\ref{def:steady_states_matrix} we gather these into a matrix, and in Theorem~\ref{thm:stst_compositional} we show that this mapping is compositional. That means that if we only know the steady states of each individual dynamical system in a coupled network, we can determine the steady states of their interconnection.%
\footnote{In fact, Theorem~\ref{thm:stst_compositional} is only about measuring the sets of steady states. If one wants access to the actual states themselves,  see Theorem~\ref{thm:Mat_sets_sym_mon_func}.
}

\begin{definition}\label{def:steady_states_matrix}

Let $X\in\cat{W}$ be a box, and let $F=(S,\rdt{f},\upd{f})\in\DS(X)$ be a discrete dynamical system. Its \emph{matrix of steady states} is the $(\vinp{X}\times\voutp{X})$-matrix, $\Stst(F)\in\Mat(X)$ given in Definition~\ref{def:discrete_to_matrix}. That is, its $(i,j)$-entry is defined by the number of steady states
\begin{equation}\label{eqn:steady_states}
M_{i,j}=\#\left\{s\in S\mid \rdt{f}(s)=j, \quad \upd{f}(i,s)=s\right\}
\end{equation}
for $i\in\vinp{X}, j\in\voutp{X}$.

\end{definition}

\begin{theorem}\label{thm:stst_compositional}

The steady state map $\Stst\colon\DS\to\Mat$ is compositional, i.e., a monoidal natural transformation of $\cat{W}$-algebras.

\end{theorem}

\begin{proof}

First we must check that for every wiring diagram $\varphi\colon X\to Y$ in $\cat{W}$, the diagram below commutes:
\[
\begin{tikzcd}
	\DS(X)\ar[r,"\DS(\varphi)"]\ar[d,"\Stst"']&\DS(Y)\ar[d,"\Stst"]\\
	\Mat(X)\ar[r,"\Mat(\varphi)"']&\Mat(Y)
\end{tikzcd}
\]
We compute both sides, using Equations~\eqref{eqn:DS_formulas},~\eqref{eqn:steady_states},~and~\eqref{eqn:Mat_formulas}, on an arbitrary $F=(S,\rdt{f},\upd{f})\in\DS(X)$:
\begin{align*}
	\Stst\big(\DS(\varphi)(F)\big)_{i,j}
&=
	\#\left\{s\in S\mid\voutp{\varphi}\big(\rdt{f}(s)\big)=j,\tn{ and } \upd{f}\Big(\vinp{\varphi}\big(i,\rdt{f}(s)\big),s\Big)=s\right\}
\\&=
	\sum_{k\in(\voutp{\varphi})^{-1}(j)}\#\left\{s\in S\mid\rdt{f}(s)=k,\tn{ and }\upd{f}\big(\vinp{\varphi}(i,k),s\big)=s\right\}
\\&=
	\Mat(\varphi)\big(\Stst(F)\big)_{i,j}.
\end{align*}
The middle equality follows because the sets $\left\{s\in S\mid\rdt{f}(s)=k,\tn{ and }\upd{f}\big(\vinp{\varphi}(i,k),s\big)=s\right\}$ are disjoint for varying values of $k$. It is easy to show that the functor $\Stst$ is monoidal; in particular,
\begin{align*}
	\Stst(F_1\boxtimes F_2)_{(i_1,i_2),(j_1,j_2)}
	&=
	\#\left\{(s_1,s_2)\in S_1\times S_2\;\;\middle|\;\;\parbox{2.4in}{$
		\rdt{(f_1\boxtimes f_2)}(s_1,s_2)=(j_1,j_2),\\
		\upd{(f_1\boxtimes f_2)}\big((i_1,i_2),(s_1,s_2)\big)=(s_1,s_2)
	$}\right\}
	\\&=
	\#\left\{(s_1,s_2)\in S_1\times S_2\;\;\middle|\;\;\parbox{2in}{$
		\rdt{f_1}(s_1)=j_1, \upd{f_1}(i_1,s_1)=s_1,\\
		\rdt{f_2}(s_2)=j_2, \upd{f_2}(i_2,s_2)=s_2
	$}\right\}
	\\&=\big(\Stst(F_1)\otimes\Stst(F_2)\big)_{(i_1,i_2),(j_1,j_2)}
\end{align*}
\end{proof}

\begin{lemma}\label{lemma:prep}

Suppose that $f\colon A\times B\to B$ and $g\colon B\to C$ are measurable functions between countably-separated measurable spaces. Then for any $a\in A$ and $c\in C$, the set 
\[X=\{b\in B\mid g(b)=c,\quad f(a,b)=b\}\] 
is measurable.

\end{lemma}

\begin{proof}

By Proposition~\ref{prop:csms}, the singleton $\{c\}\ss C$ is measurable, so the set $X_1\coloneqq g^{-1}(c)\ss B$ is measurable. Consider now the composite function 
\[B\To{\cong}\{a\}\times B\to A\times B\To{f} B.\]
It is the composition of measurable functions, so again by Proposition~\ref{prop:csms} its fixed point set $X_2=\{b\mid f(a,b)=b\}$ is measurable. Then $X=X_1\cap X_2$ is the intersection of measurable sets.
\end{proof}

\begin{definition}

Let $X\in\cat{W}$ be a box, and let $F=(S,\mu,\rdt{f},\upd{f})\in\MS(X)$ be a measurable system. For any $i\in\vinp{X}$ and $j\in\voutp{X}$, the set
\begin{equation}
\widetilde{M}_{i,j}=\left\{s\in S\mid \rdt{f}(s)=j, \quad \upd{f}(i,s)=s\right\}
\end{equation}
is measurable by Lemma~\ref{lemma:prep}. Thus we can define the \emph{matrix of steady states} of $F$ to be the $(\vinp{X}\times\voutp{X})$-matrix $\StstS(F)=M$ with $(i,j)$-entry defined by the measure 
\[M_{i,j}=\mu\left(\widetilde{M}_{i,j}\right).\]

\end{definition}

\begin{corollary}\label{cor:stst_compositional_MS}

The steady state mapping $\Stst\colon\MS\to\Mat$ is compositional, i.e., a monoidal natural transformation of $\cat{W}$-algebras.

\end{corollary}

\begin{proof}

The proof is very similar to that of Theorem~\ref{thm:stst_compositional}, with the measure $\mu$ substituted for the count $\#$.
\end{proof}

The same idea works for continuous dynamical systems. 

\begin{corollary}\label{cor:CS_steady_state}

The steady state mapping for continuous dynamical systems, as in Definition~\ref{def:steady_state_CDS}, is compositional, and for any $\epsilon>0$ the following diagram of $\cat{W}$-algebras commutes and is natural in $X$:
\[
\begin{tikzcd}[column sep=0em]
	\CS(X)\ar[dr,"\Stst(X)"']\ar[rr,"\Appx(X)"]&&\DS(X)\ar[dl,"\Stst(X)"]\\
	&\Mat(X)
\end{tikzcd}
\]

\end{corollary}

\begin{proof}

If the diagram commutes for any $X$, then clearly $\Stst\colon\CS\to\Mat$ is compositional by Theorems~\ref{thm:approximation_CSDS}~and~\ref{thm:stst_compositional}. To see that it commutes for any $X$, we simply appeal to Definitions~\ref{def:steady_states_matrix},~\ref{def:steady_state_CDS},~and~\ref{def:epsilon_approx}. That is, $s\in S$ is a steady state for the $\epsilon$-approximation $\Appx(f)$ of $f$ at input $x$ when 
\[s=\upd{f}_\epsilon(x,s)=s+\epsilon\cdot\dyn{f}(x,s)\]
Since $\epsilon>0$, this equation holds if and only if $\dyn{f}(x,s)=0$, i.e., when $s$ is a steady state of $f$. 
\end{proof}

\subsection{Remembering, rather than measuring, the steady states}

Above, we counted the number of steady states, but it is often useful to keep track of the steady states themselves. For this we provide another algebra on $\cat{W}_\Set$---i.e., a sixth "intrepretation"---whose formulas will look very much like those of matrices. 

\begin{definition}\label{def:Lambda_box}

Let $A, B\in\Set$ be sets. We define an \emph{$(A,B)$-matrix of sets} to be a function $M\colon A\times B\to\Set$; it assigns to each pair $(a,b)$ a set $M_{a,b}$.

If $X=(\inp{X},\outp{X})\in\cat{W}_\Set$ is a $\Set$-box, define $\MatS(X)$ to be the set of $(\vinp{X}, \voutp{X})$-matrices of sets: 
\[\MatS(X)\coloneqq\left\{M\colon\vinp{X}\times\voutp{X}\to\Set\right\}.\]

\end{definition}

\begin{remark}

Compare Definition~\ref{def:Lambda_box} with Definition~\ref{def:Mat_box} for ordinary matrices; the difference is that here is that $\Set$ is a category-theoretic version of semiring, with addition given by disjoint union (denoted $\sqcup$) and mutliplication given by cartesian product. Similarly, it is useful to compare Definition~\ref{eqn:Mat_sets_parallel} with Definition~\ref{def:Mat_parallel} for parallel composition, and Definition~\ref{def:Lambda_wiring} with Definition~\ref{def:Mat_wiring} for wiring.

\end{remark}

\begin{definition}\label{eqn:Mat_sets_parallel}

Suppose we are given $M^1\in\MatS(X_1)$ and $M^2\in\MatS(X_2)$. Their \emph{parallel composition}, denoted $M^1\times M^2\in\MatS(X_1\boxplus X_2)$, is defined by
\[(M^1\times M_2)_{(i_1,j_1),(i_2,j_2)}\coloneqq M^1_{i_1,j_1}\times M^2_{i_2,j_2}\]
for any $(i_1,i_2)\in\vinp{X_1}\otimes\vinp{X_2}$ and $(j_1,j_2)\in\vinp{X_2}\times\voutp{X_2}$.

\end{definition}

\begin{definition}\label{def:Lambda_wiring}

Let $\varphi\colon X\to Y$ be a wiring diagram in $\Set$, and suppose that $M\in\MatS(X)$ is a $(\vinp{X}\times\voutp{X})$-matrix of sets. We define \emph{the $\MatS$-application of $\varphi$ to $M$}, denoted $N=\MatS(\varphi)(M)\in\MatS(Y)$ to be the $(\vinp{Y}\times\voutp{Y})$-matrix of sets with entries
\begin{equation}\label{eqn:Mat_sets_formulas}
N_{i,j}\coloneqq\bigsqcup_{k\in(\voutp{\varphi})^{-1}(j)}M_{\vinp{\varphi}(i,k),k}
\end{equation}
for any $i\in\vinp{Y}$ and $j\in\voutp{Y}$.

\end{definition}

\begin{theorem}\label{thm:Mat_sets_sym_mon_func}

The assignments $X\mapsto\MatS(X)$ and $\varphi\mapsto\MatS(\varphi)$, with parallel composition as in Definition~\ref{eqn:Mat_sets_parallel} constitute a symmetric monoidal functor $\MatS\colon\cat{W}_\Set\to\Set$.

\end{theorem}

\begin{proof}

If $X\To{\varphi}Y\To{\psi}Z$ are maps and $(S,M)\in\MatS(X)$ is a matrix of sets, the proof that $\MatS(\psi)\big(\MatS(\varphi)(M)\big)=\MatS(\psi\circ\varphi)(M)$ is similar to that of Theorem~\ref{thm:Mat_sym_mon_func}. The only difference is that $\sum$ is replaced by $\bigsqcup$; compare \eqref{eqn:Mat_formulas} and \eqref{eqn:Mat_sets_formulas}.
\end{proof}

\begin{definition}\label{def:steady_states_matrix_sets}

Let $X\in\cat{W}$ be a box, and let $F=(S,\rdt{f},\upd{f})\in\DS(X)$ be a discrete dynamical system. Its \emph{matrix of steady state-sets} is the $(\vinp{X}\times\voutp{X})$-matrix of sets, denoted $M=\StstS(F)\in\MatS(X)$, with $(i,j)$-entry defined by the set steady states
\begin{equation}\label{eqn:steady_states_set}
M_{i,j}\coloneqq\left\{s\in S\mid \rdt{f}(s)=j, \quad \upd{f}(i,s)=s\right\}
\end{equation}
for $i\in\vinp{X}, j\in\voutp{X}$.
The above can be repeated with $\CS$ in place of $\DS$, using Definition~\ref{def:steady_state_CDS}.
\end{definition}

\begin{theorem}\label{thm:steady_states_matrix_sets}

The steady state-set mapping $\StstS\colon\DS\to\MatS$ is compositional, i.e., a monoidal natural transformation of $\cat{W}$-algebras, as is the count function $\#\colon\MatS\to\Mat$, and the following diagram of $\cat{W}_{\Set}$-algebras commutes:
\[
\begin{tikzcd}
	\DS\ar[r,"\StstS"]\ar[dr,"\Stst"']&\MatS\ar[d,"\#"]\\
	&\Mat
\end{tikzcd}
\]
The same holds with $\DS$ replaced by $\CS$.

\end{theorem}

\begin{proof}

Recalling Equations~\eqref{eqn:steady_states}~and~\eqref{eqn:steady_states_set}, it is clear that for any $X\in\cat{W}$, the above diagram commutes at $X$. It remains to show that $\StstS$ and $\#$ are monoidal natural transformations. The proof that $\StstS$ is monoidal is almost identical to the proof of Theorem~\ref{thm:stst_compositional} (which says $\Stst$ is monoidal), except with $\sum$ replaced by $\bigsqcup$. It is easy to show that $\#\colon\MatS\to\Mat$ is a monoidal transformation; for example one invokes the fact that count preserves products, $\#A\cdot \#B=\#(A\times B)$. 

The results for $\MS$ and $\CS$ follow similarly, from Corollary~\ref{cor:CS_steady_state}.
\end{proof}

\begin{example}

We repeat Example~\ref{ex:series_DS_matrix}, except with state set matrices rather than merely their counts. Recall that two dynamical systems are put into series, the first of which comes from Example~\ref{ex:discrete_system}. The state-set matrix for $X_1$ and $X_2$, as well as their product, are below
\[
\left(
\begin{array}{ccc}
\{2\}&\emptyset&\emptyset\\
\emptyset&\{1,4\}&\emptyset
\end{array}
\right)
\left(
\begin{array}{cc}
\{p\}&\emptyset\\
\{p,r\}&\emptyset\\
\emptyset&\{q\}
\end{array}
\right)
=
\left(
\begin{array}{cc}
\{(2,p)\}&\emptyset\\
\{(1,p),(1,r),(4,p),(4,r)\}&\emptyset
\end{array}
\right)
\]

\end{example}

\subsection{Linear stability}\label{sec:linear_stability}

As discussed in Remark~\ref{rem:bifurcation}, what we have been calling "steady state matrices" are related to bifurcation diagrams; they are more expressive in some respects and less expressive in others. They are more expressive in that we allow arbitrary readout functions, and consider the number of fixed points for each parameter and readout value; this is a very mild generalization. Conversely, steady state matrices are much less expressive than birufcation diagrams in that they do not express stability properties. In this section, we briefly discuss how to classify the linear stability of the steady states for an interconnected system $Y$ of dynamical systems, in terms of the stability properties of the individual components $X_1,\ldots,X_p$.%
\footnote{We do not discuss other sorts of stability (e.g., Lyopunov) or any more detailed classification of fixed points here, prefering to leave those for a more advanced study.}

The problem is that doing so is impossible if one's understanding of "classifying linear stability" is too narrow. For open continuous dynamical systems as shown here:
\begin{equation*}
\begin{aligned}
	\dot{x}_1&=f_1(a_1,\ldots,a_k,x_1,\ldots,x_n),\\
	\dot{x}_2&=f_2(a_1,\ldots,a_k,x_1,\ldots,x_n),\\
	&\vdots\\
	\dot{x}_n&=f_n(a_1,\ldots,a_k, x_1,\ldots,x_n),
\end{aligned}
\qquad\qquad
\begin{aligned}
	b_1&=g_1(x_1,\ldots,x_n),\\
	b_2&=g_2(x_1,\ldots,x_n),\\
	&\vdots\\
	b_\ell&=f_\ell(x_1,\ldots,x_n),
\end{aligned}
\end{equation*}
it is not enough to know the Jacobian of $f$, where the parameters $a_1,\ldots,a_k$ are held constant. Instead, one must remember three Jacobians: $\partials{f}{a}$, $\partials{f}{x}$, and $\partials{g}{x}$ for each subsystem. By \emph{generalized bifurcation diagram} we mean the steady state matrix, together with these three Jacobians at each steady state. We will show that given the generalized bifurcation diagram for each dynamical system in a coupled network, we can recover the generalized bifurcation diagram for the interconnected system. To make this precise, we will define one more algebra on $\cat{W}$, but first a brief example.

\begin{example}\label{ex:one_matrix_not_enough}
We will put a dynamical system into $X$, and interconnect it as shown right, to form a dynamical system in box $Y$.
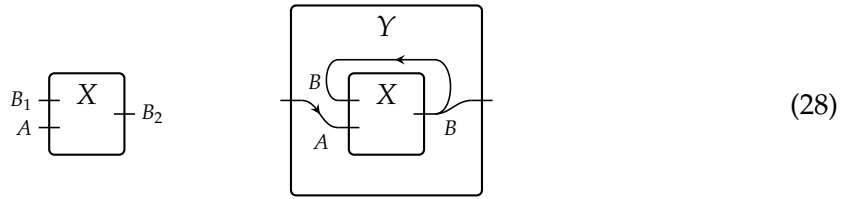
\begin{equation}\label{dia:wd_for_conway}
\begin{tikzpicture}[oriented WD, baseline=(dom.center), bbx=1em, bby=1ex]
	\node[bb={2}{1}, bb name=$X$] (dom) {};
	\draw[label] 
		node[left=2pt of dom_in2]{$A$}
		node[left=2pt of dom_in1] {$B_1$}
		node[right=2pt of dom_out1]{$B_2$};
\end{tikzpicture}
\qquad\qquad
\begin{tikzpicture}[oriented WD, baseline=(dom.center), bbx=1em, bby=1ex]
	\node[bb={2}{1}, bb name=$X$] (dom) {};
	\node[bb={1}{1}, fit={(dom) ($(dom.north east)+(1,4)$) ($(dom.south west)-(1,2)$)}, bb name = $Y$] (cod) {};
	\draw[ar,pos=20] (cod_in1') to (dom_in2);
	\draw[pos=2] (dom_out1) to (cod_out1');
	\draw[ar] let \p1=(dom.north east), \p2=(dom.north west), \n1={\y2+\bby}, \n2=\bbportlen
	in (dom_out1) to[in=0] (\x1+\n2,\n1) -- (\x2-\n2,\n1) to[out=180] (dom_in1);
	\draw[label] 
		node[below left=2pt and 3pt of dom_in2]{$A$}
		node[above left=4pt and 6pt of dom_in1] {$B$}
		node[below right=2pt and 3pt of dom_out1]{$B$};
\end{tikzpicture}
\end{equation}
We suppose a 1-dimensional state space, $S=\RR$, with dynamics $\dot{x}=2x-3b_1+a$ and readout map $b_2=x$. Note that for any constant input parameter $(a,b)$, the Jacobian is the $1\times 1$-matrix $(2)$, which has a positive real eigenvalue of 2, so the equilibrium at $x=\frac{3b_1-a}{2}$ is unstable. By Definition~\ref{def:CS_wiring}, the interconnected system formed in $Y$ will have dynamics $\dot{x}=-x+a$ and readout map $b=x$. This system is stable for any constant input, because the Jacobian is $(-1)$. 

Thus, when systems are going to be interconnected, it is not enough to remember the Jacobian of the individual systems as though their parameters were constant. Indeed, when one system's output is fed into another system as input, the parameters in the second are dependent on the readout---and thus on the state evolution---of the first. It follows that we need to keep around how the readout changes with respect to the state, and how the state changes with respect to the input. 

\end{example} 

The way one typically approaches a dynamical system is to find the steady states (equilibria) and linearize around them. We have developed a matrix-arithmetic approach that tells us the steady states of an interconnected system $Y$ of dynamical systems $X_1,\ldots,X_n$, given the steady states of each $X_i$; this is the content of Theorem~\ref{thm:stst_compositional}. Now we want to understand the linear characteristics of each steady state in $Y$ given the linear characteristics of their contributors in the $X_i$; however, as we showed in Example~\ref{ex:one_matrix_not_enough}, we need to keep around extra data. Our goal has become to explain what extra data must be assigned to each steady state at the local level in order to produce the same data again at the global level.

We begin with an abstract construction, which says that for any interpretation $Q$ of wiring diagrams, one can form matrices of steady states, each of which is assigned an element of $Q$. The only time we will actually use this is when $Q=\LS$ is the linear system interpretation.

Recall that for any set $P$, an object of $\Set_{/P}$---called the \emph{slice category of $\Set$ over $P$}---is a pair $(X,p)$, where $X$ is a set and $p\colon X\to P$ is a function. 

\begin{definition}\label{def:sets_of_Q}

Let $Q\colon\cat{W}\to\Set$ be a symmetric monoidal functor. For any box $X=(\inp{X},\outp{X})\in\cat{W}$, define an \emph{$X$-matrix of $Q(X)$'s}, denoted $\MatS_{Q}(X)$ to be a function $A\times B\to\Set_{/Q(X)}$; it assigns to each pair $(a,b)\in\vinp{X}\times\voutp{X}$ a set $M_{a,b}$, called the \emph{underlying matrix}, and a function $m_{a,b}\colon M_{a,b}\to Q(X)$, called the \emph{$Q$-assignment for $M$}.

Given $(M_1,m_1)\in\MatS_{Q}(X_1)$ and $(M_2,m_2)\in\MatS_{Q}(X_2)$, the underlying matrix for the parallel composition is the parallel composition $(M_1\otimes M_2)\in\MatS(X_1\boxplus X_2)$ of the underlying matrices (as in Definition~\ref{eqn:Mat_sets_parallel}). The $Q$-assignment is given by composing with the parallel composition in $Q$. That is, for any $(a_1,b_1)\in\vinp{X_1}\times\voutp{X_1}$ and $(a_2,b_2)\in\vinp{X_2}\times\voutp{X_2}$, use the composite
\[
(M_1\otimes M_2)_{(a_1,b_1),(a_2,b_2)}=(M_1)_{a_1,b_1}\times (M_2)_{a_2,b_2}\To{m_{a_1,b_1}\times m_{a_2,b_2}}Q(X_1)\times Q(X_2)\to Q(X_1\boxplus X_2)
\]
where the last map is parallel composition in $Q$.

Suppose that $\varphi\colon X\to Y$ is a wiring diagram, and that $(M,m)\in\MatS_{Q}(X)$ is an $X$-matrix of $Q(X)$'s. We have defined a matrix $N\coloneqq\MatS(\varphi)(M)$ in Definition~\ref{def:Lambda_wiring}. By the formula \eqref{eqn:Mat_sets_formulas} shown there, for any $(i,j)\in\vinp{Y}\times\voutp{Y}$, the set $N_{i,j}$ is the coproduct of various sets of the form $M_{i',j'}$, each of which comes with a function $m_{i',j'}\colon M_{i',j'}\to Q(X)$. By the universal property of coproducts, these maps induce a unique map $N_{i,j}\to Q(X)$, and we let $n_{i,j}$ be the composite $N_{i,j}\to Q(X)\To{Q(\varphi)}Q(Y)$.

\end{definition}

Given our work above, e.g., the proof of Theorem~\ref{thm:Mat_sets_sym_mon_func}, the following theorem becomes straightforward.
\begin{theorem}

For any symmetric monoidal functor $Q\colon\cat{W}\to\Set$, the construction of $\MatS_Q$ given in Definition~\ref{def:sets_of_Q} constitutes a symmetric monoidal functor $\MatS_Q\colon\cat{W}\to\Set$.

\end{theorem}

\paragraph{Linearizing around steady states is a compositional mapping.} We will now give a compositional mapping $\CS\to\MatS_\LS$. That is, for any continuous dynamical system $F\in\CS(X)$, we will assign a matrix of steady state-sets $\StstS(F)$ as in Definition~\ref{def:steady_states_matrix_sets}, and each steady state $s_0$ will be assigned an open linear dynamical system on $X$, namely the derivative of $F$ at $s_0$. The fact that this is compositional means that if we keep track of the generalized bifurcation diagram (steady states depending on parameters, and the respective Jacobians) for each continuous dynamical system in a coupled network, we can recover the generalized bifurcation diagram for the interconnected system.

\begin{definition}\label{def:stst_linearization}

Let $X\in\cat{W}$ be a box, and let $F=(S,\rdt{f},\dyn{f})\in\CS(X)$ be a continuous system. For any $i\in\vinp{X}$ and $j\in\voutp{X}$, let $M=\StstS(F)$ be the matrix of steady state-sets of $F$, as in Definition~\ref{def:steady_states_matrix_sets}. We define the \emph{linearization assignment} to be the function $D_{i,j}\colon M_{i,j}\to\LS(X)$ which assigns to each steady state $s_0\in M_{i,j}$ the linear system $\big(T(s_0),\inp{D_{i,j}}(s_0),\midp{D_{i,j}}(s_0),\outp{D_{i,j}}(s_0)\big)\in\LS(X)$ given by taking various derivatives, as follows. Let $T(s_0)\coloneqq S$ and put:
\begin{equation}\label{eqn:derivatives}
\inp{D_{i,j}}(s_0)\coloneqq\partial_{\inp{X}}\dyn{f}(i, s_0),\qquad
\midp{D_{i,j}}(s_0)\coloneqq\partial_{S}\dyn{f}(i, s_0),\qquad
\outp{D_{i,j}}(s_0)\coloneqq\partial_{S}\rdt{f}(s_0).
\end{equation}
Note that these are all independent of $j$. We call the pair $(M, D)\in\MatS_{\LS}(X)$ the \emph{steady state linearization} of $F$, denoted $(M,D)=\StLin(F)$.
\end{definition}

\begin{example}

Let $\varphi\colon X\to Y$ be the wiring diagram shown in \eqref{dia:wd_for_conway}, and let $F\in\CS(X)$ be the dynamical system discussed there, with dynamics $\dyn{f}(a,b_1,x)=2x-3b_1+a$ and readout map $\rdt{f}=x$. The matrix of steady state sets $M$ is given by $\StstS(F)(b_1,a,b_2)=\{s\in\RR\mid 2s-3b_1+a = 0, s=b_2\}$. Taking derivatives as shown in \eqref{eqn:derivatives}, the linearization assignment $D$, at each $(b_1,a,b_2)$, is:
\[
\inp{D}(s)=\left(\begin{array}{ll}-3&1\end{array}\right),\qquad
\midp{D}(s)=\left(\begin{array}{l}2\end{array}\right),\qquad
\outp{D}(s)=\left(\begin{array}{l}1\end{array}\right).
\]
The matrices happen to be independent of $(b_1,a,b_2)$, so we dropped the subscripts. It is the middle matrix $\midp{D}=(2)$ that we discussed in Example~\ref{ex:one_matrix_not_enough}, which showed that $F$ is unstable at the fixed point.

The steady state matrix $\StstS(G)$ of the interconnected system $G\coloneqq\CS(\varphi)(F)$ can be computed as $N=\MatS(\varphi)(F)$, by Theorem~\ref{thm:Mat_sets_sym_mon_func}. Hence, we find that $N_{a,b}=\{s\in\RR\mid -s+a=0\}$. According to Definition~\ref{def:sets_of_Q}, the steady state linearization of $G$ is $(N,E)$, where $E$ is the linearization assignment $E_{a,b}\colon N_{a,b}\to\LS(Y)$ given by composing $D$ with $\LS(\varphi)$, as in Definition~\ref{def:LS_wiring}. To do so, we need to know the derivative of the wiring diagram, as in Definition~\ref{def:derivative_wd}, which is
\[
\inp{\Phi}=\left(\begin{array}{l}0\\1\end{array}\right),\qquad
\midp{\Phi}=\left(\begin{array}{l}1\\0\end{array}\right),\qquad
\outp{\Phi}=\left(\begin{array}{l}1\end{array}\right).
\] 
Thus we compute the linearization assignment $E$ (again dropping subscripts), by applying \eqref{eqn:LS_formulas}:
\begin{gather*}
\inp{E}(s)\coloneqq
\left(\begin{array}{ll}-3&1\end{array}\right)\left(\begin{array}{l}0\\1\end{array}\right)=\left(\begin{array}{l}1\end{array}\right)
\qquad
\outp{E}(s)\coloneqq
\left(\begin{array}{l}1\end{array}\right)\left(\begin{array}{l}1\end{array}\right)
=\left(\begin{array}{l}1\end{array}\right)
\\
\midp{E}(s)\coloneqq
\left(\begin{array}{l}2\end{array}\right)+\left(\begin{array}{ll}-3&1\end{array}\right)\left(\begin{array}{l}1\\0\end{array}\right)\left(\begin{array}{l}1\end{array}\right)
=\left(\begin{array}{l}-1\end{array}\right)
\end{gather*}
As we discussed in Example~\ref{ex:one_matrix_not_enough}, the middle matrix $\midp{E}=(-1)$ tells us that $\CS(\varphi)(F)$ is stable at its fixed point.

\end{example}

\begin{theorem}
Let $\LS\colon\cat{W}\to\Set$ be the algebra of linear systems as in Theorem~\ref{thm:LS_sym_mon_func}, and let $\MatS_{\LS}\colon\cat{W}\to\Set$ be the algebra of matrices of linear systems, as in Definition~\ref{def:sets_of_Q}. The steady state linearization $\StLin\colon\CS\to\MatS_{\LS}$ from Definition~\ref{def:stst_linearization} is compositional, i.e., it is a monoidal natural transformation of $\cat{W}$-algebras.

\end{theorem}

\begin{proof}

Given a wiring diagram $\varphi\colon X\to Y$ and a continuous system $F\in\CS(X)$, we need to show that $\StLin\big(\CS(\varphi)(F)\big)=\LS(\varphi)\big(\StLin(F)\big)$. The underlying matrices of steady state sets are the same by Theorem~\ref{thm:steady_states_matrix_sets}, say $N\colon\vinp{Y}\times\voutp{Y}\to\Set$, where $N=\MatS(\varphi)(M)$ for $M=\StstS(F)\colon\vinp{X}\times\voutp{X}\to\Set$.

It remains to check that the linearization assignments are the same. Fix $i\in\vinp{Y}$ and $j\in\voutp{Y}$, and suppose that $s_0\in N_{i,j}$. Then by \eqref{eqn:Mat_sets_formulas}, there is some $i'\in\vinp{X}$ and $j'\in\voutp{X}$ such that $\voutp{\varphi}(j')=j$ and $\vinp{\varphi}(i,j')=i$, and $s_0\in M_{i',j'}$. Then by \eqref{eqn:derivatives}, \eqref{eqn:derivatives_wd}, and \eqref{eqn:LS_formulas} the linearization assignment $E$ of $\LS(\varphi)\big(\StLin(F)\big)$ at $s_0$ are 
\begin{gather*}
\inp{E}=\partial_{\inp{X}}\dyn{f}(i', s_0) \partial_{\inp{Y}}\vinp{\varphi},\qquad
\outp{E}=\partial_{\outp{X}}\voutp{\varphi} \partial_{S}\rdt{f}(s_0),\\
\midp{E}=\partial_{S}\dyn{f}(i', s_0)+\partial_{\inp{X}}\dyn{f}(i', s_0) \partial_{\outp{X}}\vinp{\varphi} \partial_{S}\rdt{f}(s_0).
\end{gather*}
On the other hand, to compute $\StLin\big(\CS(\varphi)(F)\big)$, one must take derivatives as in \eqref{eqn:derivatives} of the equations for $\CS(\varphi)$ found in \eqref{eqn:CS_formulas}:
\[\rdt{g}(s_0)=\voutp{\varphi}\left(\rdt{f}(s_0)\right),\quad\dyn{g}(i,s_0)=\dyn{f}\left(\vinp{\varphi}\left(i,\rdt{f}(s_0)\right),s_0\right).\]
By mildly tedious but straightforward chain rule calculations, quite analagous to those in Lemma~\ref{lemma:derivative_comp_wd}, one finds that these derivatives agree with the three matrices above.

\end{proof}

\section{Extended example}\label{sec:extended_example}

In this example, we string together eight discrete dynamical systems. Let's refer to the following wiring diagram as $\varphi\colon W,X,X,X,X,X,X,Y\too Z$:
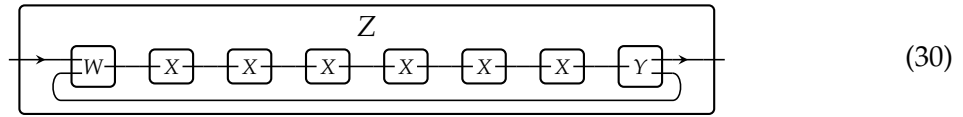
\begin{equation}\label{eqn:long_wiring_diagram}
\begin{tikzpicture}[oriented WD, baseline=(X.center), bbx=.6em, bb min width=1.5em, bby=1ex, bb port sep=1]
	\node[bb={2}{1}] (W) {$\scriptstyle W$};
	\node[bb={1}{1}, right=2 of W] (X1) {$\scriptstyle X$};
	\node[bb={1}{1}, right=2 of X1] (X2) {$\scriptstyle X$};
	\node[bb={1}{1}, right=2 of X2] (X3) {$\scriptstyle X$};
	\node[bb={1}{1}, right=2 of X3] (X4) {$\scriptstyle X$};
	\node[bb={1}{1}, right=2 of X4] (X5) {$\scriptstyle X$};
	\node[bb={1}{1}, right=2 of X5] (X6) {$\scriptstyle X$};
	\node[bb={1}{2}, right=2 of X6] (Y) {$\scriptstyle Y$};
	\node[bb={1}{1}, fit={($(W.south west)+(-2,-1)$) ($(Y.north east)+(2,2)$)}, bb name={$Z$}] (Z) {};
	\draw[ar] (Z_in1') to (W_in1);
	\draw (W_out1) to (X1_in1);
	\draw (X1_out1) to (X2_in1);
	\draw (X2_out1) to (X3_in1);
	\draw (X3_out1) to (X4_in1);
	\draw (X4_out1) to (X5_in1);
	\draw (X5_out1) to (X6_in1);
	\draw (X6_out1) to (Y_in1);
	\draw[ar] (Y_out1) to (Z_out1');
	\draw let \p1=(W.south west), \p2=(Y.south east), \n1={\y2-\bby}, \n2=\bbportlen in
		(Y_out2) to[in=0] (\x2+\n2,\n1) -- (\x1-\n2,\n1) to[out=180] (W_in2);
\end{tikzpicture}
\end{equation}
Suppose that each interior box is inhabited by a discrete dynamical system. Below, the transition diagrams are shown: the leftmost one, $w\in\DS(W)$, is shown left; the middle six are all the same, $x\in\DS(X)$, and are shown in the middle; and the rightmost one, $y\in\DS(Y)$, is shown right:
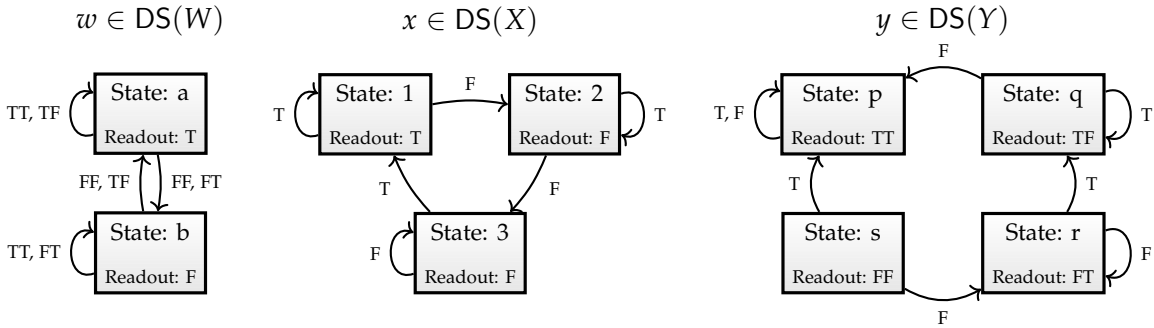
\[
\begin{tikzpicture}[baseline=(st2.south),
	box/.style={
		rectangle,
		minimum size=6mm,
		very thick,
		draw=black, 
		top color=white, 
		bottom color=white!50!black!20, 
		align=center,
		font=\normalfont
	}]
	\node [box] at (0,0) (st1) {\footnotesize State: a\\\tiny Readout: T};
	\node [box, below=.75 of st1] (st2) {\footnotesize State: b\\\tiny Readout: F};
	\draw[->,thick, bend left=10] (st1) edge["{\tiny FF, FT}"] (st2);
	\draw[->,thick, bend left=10] (st2) edge ["{\tiny FF, TF}"] (st1);
	\draw[->,thick] (st1) edge [out=200, in=160,looseness=2,"{\tiny TT, TF}"] (st1);
	\draw[->,thick] (st2) edge [out=200, in=160,looseness=2,"{\tiny TT, FT}"] (st2);
	\node[above=2.2 of st2] {$w\in\DS(W)$};
\end{tikzpicture}
\quad
\begin{tikzpicture}[baseline=(st3.south),
	box/.style={
		rectangle,
		minimum size=6mm,
		very thick,
		draw=black, 
		top color=white, 
		bottom color=white!50!black!20, 
		align=center,
		font=\normalfont
	}]
	\node [box] at (0,0) (st1) {\footnotesize State: 1\\\tiny Readout: T};
	\node [box, below right = .75 and -.25 of st1] (st3) {\footnotesize State: 3\\\tiny Readout: F};
	\node [box, above right=.75 and -.25 of st3] (st2) {\footnotesize State: 2\\\tiny Readout: F};
	\draw[->,thick, bend left=10] (st1) edge["\tiny F"] (st2);
	\draw[->,thick, bend left=10] (st2) edge["\tiny F" near start] (st3);
	\draw[->,thick, bend left=10] (st3) edge["\tiny T" near end] (st1);
	\draw[->,thick] (st1) edge [out=200, in=160,looseness=2,"\tiny T"] (st1);
	\draw[->,thick] (st2) edge [out=20, in=-20,looseness=2,"\tiny T"] (st2);
	\draw[->,thick] (st3) edge [out=200, in=160,looseness=2,"\tiny F"] (st3);
	\node[above=2.2 of st3] {$x\in\DS(X)$};
\end{tikzpicture}
\quad
\begin{tikzpicture}[baseline=(st4.south),
	box/.style={
		rectangle,
		minimum size=6mm,
		very thick,
		draw=black, 
		top color=white, 
		bottom color=white!50!black!20, 
		align=center,
		font=\normalfont
	}]
	\node [box] at (0,0) (st1) {\footnotesize State: p\\\tiny Readout: TT};
	\node [box, right = 1 of st1] (st2) {\footnotesize State: q\\\tiny Readout: TF};
	\node [box, below = .75 of st2] (st3) {\footnotesize State: r\\\tiny Readout: FT};
	\node [box, below = .75 of st1] (st4) {\footnotesize State: s\\\tiny Readout: FF};
	\draw[->,thick, bend right] (st2) edge["\tiny F"'] (st1);
	\draw[->,thick, bend right] (st3) edge["\tiny T"'] (st2);
	\draw[->,thick, bend right] (st4) edge["\tiny F"'] (st3);
	\draw[->,thick, bend left] (st4) edge["\tiny T"] (st1);
	\draw[->,thick] (st3) edge [out=20, in=-20,looseness=2,"\tiny F"] (st3);
	\draw[->,thick] (st2) edge [out=20, in=-20,looseness=2,"\tiny T"] (st2);
	\draw[->,thick] (st1) edge [out=200, in=160,looseness=2,"{\tiny T, F}"] (st1);
	\node[above=2.2 of $(st3.north)!.5!(st4.north)$] {$y\in\DS(Y)$};
\end{tikzpicture}
\]

The composed dynamical system $z\coloneqq\DS(\varphi)(w,x,x,x,x,x,x,y)$ has $2\cdot 3^6\cdot 4=5832$ states. One can imagine it as a stack of eight parallel layers: a $w$, then six $x$'s, then a $y$, each sending information to the next---with feedback at the end---i.e., the readout of one layer sent forward as the state-change command for the next. A composite state is a choice of one state in each layer. 

Rather than write out the 5832-state transition diagram of the layered system, suppose we just want to understand its steady states. We begin by writing down the matrix of steady state sets associated to each system, e.g., $\StstS(w)\in\MatS(w)$. They are shown below with row- and column-labels to keep things clear:
\[
\begin{array}{c| ll}
\multicolumn{3}{c}{\StstS(w)=}\\
\multicolumn{3}{c}{}\\
\parbox{.55in}{\tiny ~\hspace{.13in}Outputs:\\Is fixed by:}
&T&F\\\hline
TT&\{a\}&\{b\}\\
TF&\{a\}&\emptyset\\
FT&\emptyset&\{b\}\\
FF&\emptyset&\emptyset
\end{array}
\qquad\quad
\begin{array}{c| ll}
\multicolumn{3}{c}{\StstS(x)=}\\
\multicolumn{3}{c}{}\\
\parbox{.55in}{\tiny ~\hspace{.13in}Outputs:\\Is fixed by:}
&T&F\\\hline
T&\{1\}&\{2\}\\
F&\emptyset&\{3\}\\&&\\&&
\end{array}
\qquad\quad
\begin{array}{c| llll}
\multicolumn{5}{c}{\StstS(y)=}\\
\multicolumn{3}{c}{}\\
\parbox{.55in}{\tiny ~\hspace{.13in}Outputs:\\Is fixed by:}
&TT&TF&FT&FF\\\hline
T&\{p\}&\{q\}&\emptyset&\emptyset\\
F&\{p\}&\emptyset&\{r\}&\emptyset\\&&\\&&
\end{array}
\]
We will abbreviate tuples by removing parentheses and commas, so that the two-element set $\{(a,b,c),(c,e)\}$ is written $\{abc,ce\}$. 

We can now use the steady-state formulas~\eqref{eqn:Mat_sets_parallel}~and~\eqref{eqn:Mat_sets_formulas} to compute the steady state-set matrix $\StstS(z)$. However, using these fully general formulas is not always the most efficient approach. Following Examples~\ref{ex:serial_works}~and~\ref{ex:trace_works}, we can instead multiply the matrices for the serial composition 
\[\StstS(w)\StstS(x)^6\StstS(y),\]
and then trace the result. The fact that this will work comes down to the functoriality of $\StstS$ (proven as Theorem~\ref{thm:Mat_sets_sym_mon_func}).

The reader should try calculating the matrix multiplication $\StstS(x)\StstS(x)$. One can then see that the serial composition of the middle $x$'s, namely $\StstS(x)^6$, is
\[
\begin{array}{c| ll}
\parbox{.55in}{\tiny ~\hspace{.13in}Outputs:\\Is fixed by:}
&T&F\\\hline
T&\{111111\}&\{111112,111123,111233,112333,123333,233333\}\\
F&\emptyset&\{3333333\}
\end{array}
\]
We continue in this way, and calculate $\StstS(w)\StstS(x^6)\StstS(y)$:
\[
\begin{array}[t]{c| llll}
	\multicolumn{5}{c}{\StstS(wx^6y)=}\\
\parbox{.55in}{\tiny ~\hspace{.13in}Outputs:\\Is fixed by:}
	&TT&TF&FT&FF\\\hline\\
TT
	&\left\{\parbox{1.4in}{$a111111p,a111112p,\\a111123p,a111233p,\\a112333p,a123333p,\\a233333p,b333333p$}\right\}
	&\{a111111q\}
	&\left\{\parbox{1.4in}{$a111112r,a111123r,\\a111233r,a112333r,\\a123333r,a233333r,\\b333333r$}\right\}
	&\emptyset
	\\\\
TF
	&\left\{\parbox{1.4in}{$a111111p,a111112p,\\a111123p,a111233p,\\a112333p,a123333p,\\a233333p$}\right\}
	&\{a111111q\}
	&\left\{\parbox{1.4in}{$a111112r,a111123r,\\a111233r,a112333r,\\a123333r,a233333r$}\right\}
	&\emptyset
	\\\\
FT
	&\{b333333p\}
	&\emptyset
	&\{b333333r\}
	&\emptyset
	\\\\
FF
	&\emptyset
	&\emptyset
	&\emptyset
	&\emptyset
\end{array}
\]
Finally, we take the partial trace of this matrix to obtain the desired result, $\StstS(z)$:
\[
\begin{array}[t]{c| ll}
\multicolumn{3}{c}{\StstS(x)=}\\
\parbox{.55in}{\tiny ~\hspace{.13in}Outputs:\\Is fixed by:}
&T&F\\\hline\\
T
	&\left\{\parbox{1.4in}{$a111111p,a111112p,\\a111123p,a111233p,\\a112333p,a123333p,\\a233333p,b333333p,\\a111111q$}\right\}
	&\left\{\parbox{1.4in}{$a111112r,a111123r,\\a111233r,a112333r,\\a123333r,a233333r,\\b333333r$}\right\}
	\\\\
F
	&\{b333333p\}
	&\{b333333r\}
\end{array}
\]
This matrix tells us the steady states of the composite dynamical system $z$. 

Let's interpret the results by invoking our image of $z$ as a layered system of $w$, the six $x$'s, and $y$. If we input 'T' to the system our matrix tells us that 'a111223p' is a steady state outputting 'T' and that 'b333333r' is a steady state outputting 'F'. These 8-character strings are the composite states; one in each layer.

It is easy to check that when $y$ is in state p, the system $z$ outputs 'T' and that when $y$ is in state r, the system outputs 'F'.%
\footnote{Looking at the wiring diagram \eqref{eqn:long_wiring_diagram}, the top output of $y$ is output by the system and the second is fed back to $w$, so if $y$ outputs 'FT' then $z$ outputs 'F'.} 
Thus it suffices to see that these two states are fixed by an input of T. We leave it to the reader to check that the output of every layer indeed leaves the state of the next layer fixed.

\bibliographystyle{alpha}
\bibliography{Library}
\end{document}